\documentclass[oneside]{article}

\usepackage[mathlines]{lineno} 
\usepackage{amsmath}           
\usepackage{etoolbox}          

\newcommand*\linenomathpatch[1]{%
  \cspreto{#1}{\linenomath}%
  \cspreto{#1*}{\linenomath}%
  \csappto{end#1}{\endlinenomath}%
  \csappto{end#1*}{\endlinenomath}%
}
\newcommand*\linenomathpatchAMS[1]{%
  \cspreto{#1}{\linenomathAMS}%
  \cspreto{#1*}{\linenomathAMS}%
  \csappto{end#1}{\endlinenomath}%
  \csappto{end#1*}{\endlinenomath}%
}

\expandafter\ifx\linenomath\linenomathWithnumbers
  \let\linenomathAMS\linenomathWithnumbers
  \patchcmd\linenomathAMS{\advance\postdisplaypenalty\linenopenalty}{}{}{}
\else
  \let\linenomathAMS\linenomathNonumbers
\fi

\linenomathpatch{equation}
\linenomathpatchAMS{gather}
\linenomathpatchAMS{multline}
\linenomathpatchAMS{align}
\linenomathpatchAMS{alignat}
\linenomathpatchAMS{flalign}

\makeatletter
\patchcmd{\mmeasure@}{\measuring@true}{
  \measuring@true
  \ifnum-\linenopenaltypar>\interdisplaylinepenalty
    \advance\interdisplaylinepenalty-\linenopenalty
  \fi
  }{}{}
\makeatother

\usepackage[english]{babel}
\usepackage[T1]{fontenc}
\usepackage{lmodern}

\usepackage{geometry}
\usepackage{amsfonts}
\usepackage{amssymb}
\usepackage{mathtools}
\usepackage{cmll}
\usepackage{stmaryrd}
\usepackage{aligned-overset}
\usepackage{enumitem}
\allowdisplaybreaks

\usepackage{hyperref}
\hypersetup{
    colorlinks = true,
    linkcolor = blue,
    citecolor = blue,
    urlcolor = red
}

\usepackage[noabbrev,capitalise,nameinlink]{cleveref}
\apptocmd{\thebibliography}{\raggedright}{}{}

\usepackage{tikz-cd}
\usepackage{graphicx}
\usepackage{xcolor}
\usepackage{epigraph}
\usepackage{ebproof}

\usepackage{amsthm}
\usepackage{thmtools}

\newtheorem{theorem}{Theorem}[section]
\newtheorem{definition}[theorem]{Definition}
\newtheorem{proposition}[theorem]{Proposition}
\newtheorem{lemma}[theorem]{Lemma}

\newtheorem{remark}[theorem]{Remark}
\newtheorem{example}[theorem]{Example}

\setlist[enumerate,1]{
  align=left,
   labelsep=1em,
  leftmargin=2em,
}

\makeatletter
\@addtoreset{enumi}{theorem}

\AtBeginEnvironment{align}{%
  \stepcounter{enumi}%
}

\let\oldequation\equation
\let\endoldequation\endequation
\renewenvironment{equation}{%
  \refstepcounter{enumi}
  \oldequation
}{%
  \endoldequation
}
\makeatother

\crefname{enumi}{}{}    
\Crefname{enumi}{}{}    
    
\crefname{equation}{}{}
\crefname{equations}{}{}

\title{Extracting an $\mathbb{N}$-filtered differential modality from a differential modality}

\author{Jean-Baptiste Vienney}

\begin{document}

\maketitle

\begin{abstract}
A differential modality is a comonad on an additive symmetric monoidal 
category $(\mathsf{C},\otimes,I)$, whose underlying functor we denote 
$\oc\colon\mathsf{C} \rightarrow \mathsf{C}$, together with some additional 
structure including a differential operator $\partial\colon\oc A \otimes A 
\rightarrow \oc A$. 
A morphism $f\colon\oc A \rightarrow B$ is interpreted as a smooth function 
from $A$ to $B$. The notion of an $\mathbb{N}$-filtered differential modality 
is a variant in 
which a notion of degree is present. Instead of a single functor $\oc\colon 
\mathsf{C} \rightarrow \mathsf{C}$, we ask for a family of functors 
$\oc_{\le n}\colon\mathsf{C} \rightarrow \mathsf{C}$ where $n \in \mathbb{N}$. 
Now, a morphism $f\colon\oc_{\le n} A \rightarrow B$ is interpreted as a smooth 
function 
from $A$ to $B$, with degree less than $n$ for some notion of degree. We prove 
that under mild conditions, every differential modality on an additive symmetric 
monoidal 
category with underlying functor $\oc\colon \mathsf{C} \rightarrow \mathsf{C}$ 
yields an $\mathbb{N}$-filtered differential modality with underlying functors 
$\oc_{\le n}\colon\mathsf{C} \rightarrow \mathsf{C}$. A morphism 
$f\colon\oc_{\le n}A \rightarrow B$ corresponds to a polynomial map of 
degree less than $n$ from $A$ to $B$, 
in the sense that the $(n+1)$-th derivative of $f$ is $0$.
\end{abstract}

\paragraph*{Acknowledgments.} I would like to thank Richard Blute and Jean-Simon Pacaud Lemay for useful discussions and their support of this research. This work was financially supported by the University of Ottawa and NSERC, under the grant awarded to Richard Blute.

\section{Introduction}
This work is concerned with differentiation in additive symmetric monoidal categories. An additive symmetric monoidal category
is a symmetric monoidal category enriched over commutative monoids. The main example is the category $(\mathsf{Vec}_k,\otimes,k)$ of vector spaces over some field $k$. 
A differential modality \cite{DIFCAT,DIFCATREV} on an additive symmetric monoidal category is a comonad $(\oc,m,u)$ together with natural transformations 
$\Delta\colon \oc A \rightarrow \oc A \otimes \oc A$ and $\epsilon\colon\oc A \rightarrow I$ making every object $\oc A$ into a cocommutative comonoid, and with a deriving 
transformation which is a natural transformation $\partial\colon \oc A \otimes A \rightarrow \oc A$ satisfying axioms from calculus. A morphism $f\colon\oc A \rightarrow B$ 
must be understood as a smooth map from $A$ to $B$, while $\partial;f\colon \oc A \otimes A \rightarrow B$ must be understood as the differential of $f$.

In an $R$-graded differential modality \cite{GRADIFCAT} for some rig $R$, the comonad is replaced by an $R$-graded comonad $((\oc_r),(m_{r,s}),u)$ where we have 
an endofunctor $\oc_r\colon \mathsf{C} \rightarrow \mathsf{C}$ for every $r \in R$, a natural transformation $m_{r,s}\colon \oc_{rs}A \rightarrow \oc_r\oc_sA$ for 
all $r,s \in R$ and a natural transformation $u\colon \oc_1A \rightarrow A$. In the same way, the natural transformations $\Delta \colon \oc A \rightarrow \oc A \otimes \oc A$ 
and $\epsilon\colon \oc A \rightarrow I$ are replaced by natural transformations $\Delta_{r,s}\colon \oc_{r+s}A \rightarrow \oc_rA \otimes \oc_sA$ and 
$\epsilon \colon \oc_0A \rightarrow I$. Finally, the deriving transformation is now given by a family of natural transformations 
$\partial_r \colon \oc_rA \otimes A \rightarrow \oc_{r+1}A$. The intuition is that a morphism of type $\oc_r A \rightarrow B$ is a smooth map from $A$ to $B$ 
with degree $r$ for some notion of degree, similar to the degree of a polynomial.

In the conclusion of \cite{GRADIFCAT}, it was asked whether it would be possible to somehow extract a graded differential modality from a differential modality. 
It was also written that it would be interesting to investigate what is added by considering the grading over an ordered semiring. In this work, we answer these two 
questions simultaneously. If $(R,\le)$ is an ordered rig, we define an $R$-filtered differential modality 
as an $R$-graded differential modality together with a natural transformation $t_{r,s}\colon \oc_rA \rightarrow \oc_sA$ whenever $r \ge s$. Given a differential modality, 
we can define by induction a natural transformation $\partial^n\colon\oc A \otimes A^{\otimes n} \rightarrow \oc A$ for every $n \ge 0$. If $f \colon \oc A \rightarrow B$, 
then $\partial^n;f\colon \oc A \otimes A^{\otimes n} \rightarrow B$ must be understood as the $n$-th iterated differential of $f$.

\paragraph{Main result.} 
Recall that in an additive category (that is, a category enriched over commutative monoids), a cokernel of a morphism $f\colon A \rightarrow B$ is a coequalizer of the diagram
\begin{equation*}
\begin{tikzcd}
A \arrow[rr, "f", shift left=2] \arrow[rr, "0"', shift right=2] &  & B~.
\end{tikzcd}
\end{equation*}
Our main result is the following theorem, and most of the paper is devoted to its proof.
\begin{theorem} \label{theorem:extracting}
Let $(\oc,\epsilon,u,\Delta,m,\partial)$ be a differential modality on an additive symmetric monoidal category $(\mathsf{C},\otimes,I)$. Suppose that for all object $A$ and 
$n \in \mathbb{N}$,
\begin{itemize}
\item[1.] $s_n\colon \oc A \rightarrow \oc_{\le n}A$ is a cokernel of $\partial^{n+1}\colon \oc A \otimes A^{\otimes(n+1)} \rightarrow \oc A$,
\item[2.] $s_n \otimes 1_A\colon \oc A \otimes A \rightarrow \oc_{\le n}A \otimes A$ is a cokernel of 
$\partial^{n+1} \otimes 1_A\colon \oc A \otimes A^{\otimes (n+2)} \rightarrow \oc A \otimes A$.
\end{itemize}
 Then, there exists a unique $\mathbb{N}$-graded differential modality $(\oc_{\le n},\epsilon^\le,u^\le,\Delta^\le_{n,p},m^\le _{n,p},\partial^\le_n)$ on $(\mathsf{C},\otimes,I)$ such that:
\begin{itemize}
\item[1.] the endofunctor $\oc_{\le n}\colon\mathsf{C} \rightarrow \mathsf{C}$ applied to any object $A$ gives $\oc_{\le n}A$,
\item[2.] $(s_n)_{n \in \mathbb{N}}$ is a $(0\colon\mathbb{N} \rightarrow 0)$-graded morphism of graded differential modalities.
\end{itemize}
Moreover, there exists a unique family of morphisms $t_{n,p}\colon \oc_{\le n}A \rightarrow \oc_{\le p}A$ defined for all $(n,p) \in \mathbb{N}^2$ with $n \ge p$ 
and $A \in \mathsf{C}$ such that:
\begin{itemize}
\item[1.] $(\oc_{\le n},\epsilon^\le,u^\le,\Delta^\le_{n,p},m^\le _{n,p},\partial^\le_n,t_{n,p})$ is an $\mathbb{N}$-filtered differential modality,
\item[2.] $(s_n)_{n \in \mathbb{N}}$ is a $(0\colon\mathbb{N} \rightarrow 0)$-filtered morphism of $\mathbb{N}$-filtered differential modalities.
\end{itemize}
\end{theorem}
\paragraph{Structure of the paper.} We introduce the main definitions in \cref{sec:2}. In \cref{sec:3}, we deal with some technicalities on inductive definitions, partitions 
and permutations that will be needed later. In \cref{sec:4}, we state the higher-order rules for $\partial^n$ which are a crucial ingredient in the proof of
\cref{theorem:extracting}. The longest proofs for these higher-order rules are deferred to \cref{sec:appendix}. We prove \cref{theorem:extracting} in \cref{sec:5} 
which is the longest section of this paper. Finally, in \cref{sec:example-rel}, we apply our theorem to the multiset differential modality on $\mathsf{Rel}$ and in 
\cref{sec:example}, we apply our theorem to the symmetric algebra differential modality on $(\mathsf{Vec}_k^{\mathrm{op}},\otimes,k)$. We will see that in the example 
of the symmetric algebra, the $\mathbb{N}$-filtered differential modality thus obtained does not behave as could be expected at first.
\paragraph{Notations and conventions.} 
~\\~\\
\emph{Numbers.}
\begin{itemize}
\item $\mathbb{N}$ denotes the set of all the integers $n \ge 0$.
\end{itemize}
\emph{Permutations.}
\begin{itemize}
\item For every $n \ge 0$, $S_n$ denotes the symmetric group of order $n$. Note that $S_0=\{\emptyset\}$.
\item Let $n \ge 2$ and $i,j \in \{1,\dots,n\}$. The permutation $(i,j) \in S_n$ is defined by $(i,j)(k)=k$ for every $k \in \{1,\dots,n\} \backslash \{i,j\}$, 
$(i,j)(i)=j$ and $(i,j)(j)=i$.
\end{itemize}
\emph{Rigs.}
\begin{itemize}\item The letter $R$ will denote an arbitrary rig.\footnote{We define a rig as a set $R$ which is a commutative monoid $(R,+,0)$ and a monoid $(R,\cdot,1)$, 
such that $\cdot$ distributes over $+$ and $0 \cdot r=r \cdot 0=0$. If $R$ and $S$ are two rigs, then a rig homomorphism $\rho$ from $R$ to $S$ is a function from $R$ to $S$ 
which a monoid homomorphism from $(R,+,0)$ to $(S,+,0)$ and a monoid homomorphism from $(R,\cdot,1)$ to $(S,\cdot,1)$.}
\item We will use the notation $(\alpha_r)$ for a family indexed by $r \in R$ and the notation $(\alpha_{r,s})$ for a family indexed by $(r,s) \in R \times R$.
\item For every $n \ge 0$, $\mathbb{N}[S_n]$ denotes the rig whose elements are finite formal sums $\sigma_1+\dots+\sigma_r$ for any $r \ge 0$ (if $r=0$, this formal sum 
is denoted $0$) and $\sigma_1,\dots,\sigma_r \in S_n$. The product is given by 
\begin{equation*}
(\sigma_1+\dots+\sigma_r)(\tau_1+\dots+\tau_s)=\underset{\substack{1 \le i \le r \\ 1 \le j \le s}}{\sum}\sigma_i \circ \tau_j.
\end{equation*}
The multiplicative identity is the identity permutation in $S_n$.
\end{itemize}
\emph{Categories.}
\begin{itemize}
\item If $f\colon A \rightarrow B$ and $g\colon B \rightarrow C$ are morphisms in a same category, we write either $f;g$ or $g \circ f$ for the composite $A \rightarrow C$.
\item If $\mathsf{C}$ is a category and $A \in \mathsf{C}$, we write $1_A$ for the identity morphism on $A$.
\item If $\mathsf{C}$ is a category, then $\mathsf{End}(\mathsf{C})$ is the category of endofunctors of $\mathsf{C}$, that is, all the functors from $\mathsf{C}$ to $\mathsf{C}$, 
and natural transformations between them.
\item Very often, if $\nu$ is a natural transformation from a functor $F$ to a functor $G$, we will write $\nu$ instead of $\nu_A$ for the corresponding morphism from $FA$ to $GA$. It will make diagrams and calculations more readable.
\item If $(\mathsf{C},\otimes,I)$ is a symmetric monoidal category, then we say that it is strict monoidal if the unitors and associator are equalities of functors:
\begin{equation*}
I \otimes - = 1_{\mathsf{C}} = - \otimes I\colon\mathsf{C} \rightarrow \mathsf{C},
\end{equation*}
\begin{equation*}
(- \otimes -) \otimes - = - \otimes (- \otimes -)\colon\mathsf{C}^3 \rightarrow \mathsf{C}.
\end{equation*}
If $\gamma_{A,B}:A \otimes B \rightarrow B \otimes A$ is the symmetry, this implies that
\begin{equation*}
\gamma_{I,A}=\gamma_{A,I}=1_A.
\end{equation*}
In definitions, results and proofs, we ignore the coherence isomorphisms and work as if all symmetric monoidal categories were strict monoidal. Adding the missing coherence 
isomorphisms would be tedious but could be carried out with enough time.
\item We will use the following operator precedence:
\begin{equation*}
\otimes\quad>\quad;\quad>\quad\underset{i \in I}{\sum}\quad>\quad+\quad.
\end{equation*}
For instance, an expression of the form
\begin{equation*}
\underset{i \in I}{\sum}f \otimes g;h+i
\end{equation*}
must be interpreted as
\begin{equation*}
\Big(\underset{i \in I}{\sum}((f \otimes g);h)\Big)+i.
\end{equation*}
\item Let $(\mathsf{C},\otimes,I)$ be a symmetric monoidal category. If $n=0$, we define
\begin{equation*}
\overline{\sigma}:=1_I\colon I \rightarrow I
\end{equation*}
where $\sigma$ is the unique permutation in $S_0$. If $n=1$, we define the natural transformation
\begin{equation*}
\overline{\sigma}_A:=1_A \colon A \rightarrow A
\end{equation*}
where $\sigma$ is the unique permutation in $S_1$. Suppose now that $n \ge 2$ and $\sigma \in S_n$. We can write $\sigma=(i_1,i_1+1);\dots;(i_r,i_r+1)$ for 
some $r \ge 0$ and $i_1,\dots,i_r \in \{1,\dots,n\}$ (if $r=0$, the RHS must be understood as the identity on $\{1,\dots,n\}$). We then define the natural 
isomorphism\footnote{The naturality of $\overline{\sigma}$ is expressed as follows: if $f_1\colon A_1 \rightarrow B_1$, \dots, $f_n\colon A_n \rightarrow B_n$ 
are morphisms in $\mathsf{C}$, then we have 
$f_1\otimes \dots \otimes f_n;\overline{\sigma}_{A_1,\dots,A_n}=\overline{\sigma}_{B_1,\dots,B_n};f_{\sigma^{-1}(1)} \otimes \dots \otimes f_{\sigma^{-1}(n)}$. 
Moreover, we have $(\overline{\sigma}_{A_1,\dots,A_n})^{-1}=(\overline{\sigma^{-1}})_{A_{\sigma^{-1}(1)},\dots,A_{\sigma^{-1}(n)}}$.}
\begin{equation*}
\overline{\sigma}_{A_1,\dots,A_n}\colon A_1 \otimes \dots \otimes A_{n} \rightarrow A_{\sigma^{-1}(1)} \otimes \dots \otimes A_{\sigma^{-1}(n)}
\end{equation*}
as follows. First we define
\begin{equation*}
\tau_k:=\Big((i_1,i_1+1);\dots;(i_k,i_k+1)\Big)^{-1}
\end{equation*}
for every $1 \le k \le r$. Then, we define
\begin{equation*}
A^k_s:=A_{\tau_{k-1}(s)}
\end{equation*}
for all $1 \le k \le r$ and $1 \le s \le n$. Finally, we define
\begin{align*}
\overline{\sigma}_{A_1,\dots,A_n}:=&~1_{A_1^1 \otimes \dots \otimes A_{i_1-1}^1} \otimes \gamma_{A_{i_1}^1,A_{i_1+1}^1} \otimes 1_{A_{i_1+2}^1 \otimes \dots \otimes A_n^1}; \\
&~1_{A_1^2 \otimes \dots \otimes A_{i_1-1}^2} \otimes \gamma_{A_{i_2}^2,A_{i_2+1}^2} \otimes 1_{A_{i_2+2}^2 \otimes \dots \otimes A_n^2}; \\
&~\dots \\
&~1_{A_1^r \otimes \dots \otimes A_{i_r-1}^r} \otimes \gamma_{A_{i_r}^r,A_{i_r+1}^r} \otimes 1_{A_{i_r+2}^r \otimes \dots \otimes A_n^r}\colon \\
&~A_1 \otimes \dots \otimes A_n \longrightarrow A_{\sigma^{-1}(1)} \otimes \dots \otimes A_{\sigma^{-1}(n)}.
\end{align*}
By Mac Lane's coherence theorem for symmetric monoidal categories \cite{COHERENCE}, the natural transformation $\overline{\sigma}_{A_1,\dots,A_n}$ does not depend on 
the decomposition of the permutation $\sigma$ as a product of transpositions.
\end{itemize}
\emph{Multisets.}
\begin{itemize}
\item If $I$ is a set, we define a multiset $M \in \mathcal{M}(I)$ as a function $M\colon I \rightarrow \mathbb{N}$ such that $M(i)=0$ for all but a finite number of $i \in I$.
\item If $M \in \mathcal{M}(I)$, we define the cardinal of $M$ as $|M|:=\underset{i \in I}{\sum}M(i) \in \mathbb{N}$.
\item For every $n \ge 0$, we define $\mathcal{M}_n(I):=\{M \in \mathcal{M}(I),~|M|=n\}$.
\item For all $n \ge 0$ and $i_1,\dots,i_n \in I$, we define the multiset $[i_1,\dots,i_n] \in \mathcal{M}_n(I)$ by 
\begin{equation*}
[i_1,\dots,i_n](j):=|\{1 \le k \le n~|~i_k=j\}|
\end{equation*}
for every $j \in I$.
\item If $N \in \mathcal{M}(\mathcal{M}(I))$, we define $\sum N \in \mathcal{M}(I)$ by $\big(\sum N\big)(i):=\underset{M \in \mathcal{M}(I)}{\sum}N(M)\cdot M(i)$ 
for every $i \in I$. This sum is finite since there exists only a finite number of $M \in \mathcal{M}(I)$ such that $N(M) \neq 0$.
\item If $M,N \in \mathcal{M}(I)$, we define $M+N \in \mathcal{M}(I)$ by $M+N=\sum[M,N]$, or equivalently $(M+N)(i):=M(i)+N(i)$ for every $i \in I$.
\item For every $\alpha \in \mathcal{M}(I)$, we define 
\begin{equation*}
\alpha!:=\underset{\substack{i \in I \\ \alpha(i) \neq 0}}{\prod}\alpha(i)!~.
\end{equation*}
\item For all $\alpha,\beta \in \mathcal{M}(I)$, we write $\alpha \le \beta$ if $\alpha(i) \le \beta(i)$ for every $i \in I$.
\item For all $\alpha,\beta \in \mathcal{M}(I)$ such that $\alpha \le \beta$, we define $\beta-\alpha \in \mathcal{M}(I)$ by the formula $(\beta-\alpha)(i)=\beta(i)-\alpha(i)$ 
for every $i \in I$.
\item For all $\alpha,\beta \in \mathcal{M}(I)$ such that $\alpha \le \beta$, we define the falling factorial
\begin{equation*}
\beta^{\underline{\alpha}}:=\underset{\substack{i \in I \\ \beta(i) \neq 0}}{\prod}\beta(i)(\beta(i)-1)\cdots(\beta(i)-(\alpha(i)-1)) \in \mathbb{N} \backslash \{0\}.
\end{equation*}
For instance, suppose that $i_1, i_2 \in I$ such that $i_1 \neq i_2$. If $\alpha(i_1)=2, \alpha(i_2)=1$ and $\alpha(i)=0$ for every $i \in I \backslash \{i_1,i_2\}$, and 
if $\beta \in \mathcal{M}(I)$ such that $\beta \ge \alpha$, then
\begin{equation*}
\beta^{\underline{\alpha}}=\beta(i_1)(\beta(i_1)-1)\beta(i_2).
\end{equation*}
\end{itemize}
\emph{Finitely supported functions}
\begin{itemize}
\item If $k$ is a field, $M$ a $k$-vector space and $I$ a set, we denote by $M^{(I)}$ the set of all functions $f\colon I \rightarrow M$ such that $f(i)=0$ for all but finitely many $i \in I$.
\end{itemize}
\emph{Abbreviations.}
\begin{itemize}
\item add.\,: additivity;
\item coassoc.\,: coassociativity;
\item def.\,: definition;
\item func.\,: functoriality;
\item iff\,: if and only if;
\item hyp.\,: hypothesis;
\item LHS\,: left-hand Side;
\item nat.\,: naturality;
\item RHS\,: right-hand side;
\item sep.\,: separating.
\item prop.\,: property;
\item th.\,: theorem.
\end{itemize}
\section{Main definitions} \label{sec:2}
In this section, we introduce the definitions which are needed to state \cref{theorem:extracting}. Along the way, we prove a few simple results which will be used in 
the proof of \cref{theorem:extracting}.
\begin{definition} \label{add-sym-cat}
An \emph{additive symmetric monoidal category} is a symmetric monoidal category $(\mathsf{C},\otimes,I)$ such that every homset is a commutative monoid and we have: 
\begin{enumerate}
\item $0;f=0$, \label{add-1}
\item $f;0=0$, \label{add-2}
\item $(f+g);h = (f;h)+(g;h)$, \label{add-3}
\item $f;(g+h)=(f;g)+(f;h)$, \label{add-4}
\item $0 \otimes f=0$, \label{add-5}
\item $f \otimes 0=0$, \label{add-6}
\item $(f+g) \otimes h = (f \otimes h)+(g \otimes h)$, \label{add-7}
\item $f \otimes (g+h)=(f \otimes g)+(f \otimes h)$. \label{add-8}
\end{enumerate}
whenever it makes sense.
\end{definition}
From now on and until the end of this section, we will work in a given additive symmetric monoidal category $(\mathsf{C},\otimes,I)$. The letter $R$ denotes an arbitrary rig.
\begin{definition} \label{def-graded-comonad}
An $R$-graded comonad is a tuple $(\oc,m,u)$ where $\oc=(\oc_r\colon\mathsf{C} \rightarrow \mathsf{C})$ is a family of endofunctors, $m=(m_{r,s}\colon\oc_{rs} A 
\rightarrow \oc_r \oc_s A)$ is a family of natural transformations and $u\colon\oc_1 A \rightarrow A$ is a natural transformation such that the following diagrams commute:
\begin{equation*}
\begin{tikzcd}
\oc_{npq}A \arrow[rr, "{m_{np,q}}"] \arrow[d, "{m_{n,pq}}"'] &  & \oc_{np}\oc_qA \arrow[d, "{m_{n,p}}"] &  & \oc_rA \arrow[r, "{m_{r,1}}"] \arrow[rd, equal] \arrow[d, "{m_{1,r}}"'] & \oc_r\oc_1A \arrow[d, "\oc_ru"] \\
\oc_n\oc_{pq}A \arrow[rr, "{\oc_nm_{p,q}}"']                 &  & \oc_n\oc_p\oc_qA                      &  & \oc_1\oc_rA \arrow[r, "u"']                                      & \oc_rA                         
\end{tikzcd}
\end{equation*}
\end{definition}
\begin{definition} \label{def:graded-comonoid}
A cocommutative $R$-graded comonoid is a tuple $(M,\Delta,\epsilon)$ where $M=(M_r)$ is a family of objects, $\Delta=(\Delta_{r,s}\colon M_{r+s} \rightarrow M_r \otimes M_s)$ 
is a family of morphisms and $\epsilon\colon M_0 \rightarrow I$ is a morphism such that the following diagrams commute:
\begin{equation*}
\begin{tikzcd}
M_{r+s+t} \arrow[rr, "{\Delta_{r+s,t}}"] \arrow[d, "{\Delta_{r,s+t}}"'] &  & M_{r+s} \otimes M_t \arrow[d, "{\Delta_{r,s} \otimes 1}"] & M_0 \arrow[r, "{\Delta_{0,0}}"] \arrow[rd, equal] & M_0 \otimes M_0 \arrow[d, "1 \otimes \epsilon"] & M_{r+s} \arrow[rd, "{\Delta_{s,r}}"'] \arrow[r, "{\Delta_{r,s}}"] & M_r \otimes M_s \arrow[d, "\gamma"] \\
M_r \otimes M_{s+t} \arrow[rr, "{1 \otimes \Delta_{s,t}}"']             &  & {M_r \otimes M_s \otimes M_t~,}                           &                                            & {M_0~,}                                         &                                                               & M_s \otimes M_r~.                  
\end{tikzcd}
\end{equation*}
\end{definition}
\begin{definition} \label{graded-diff-mod}
An $R$-graded differential modality is a tuple $(\oc,m,u,\Delta,\epsilon,\partial)$ where $\oc=(\oc_r\colon\mathsf{C} \rightarrow \mathsf{C})$ is a family of endofunctors; 
$m=(m_{r,s}\colon\oc_{rs} A \rightarrow \oc_r \oc_s A)$ is a family of natural transformations and $u\colon\oc_1 A \rightarrow A$ is a natural transformation such that $(\oc,m,u)$
 is an $R$-graded comonad; $\Delta=(\Delta_{r,s}\colon\oc_{r+s} A \rightarrow \oc_r A \otimes \oc_s A)$ is a family of natural transformations and 
 $\epsilon\colon\oc_0 A \rightarrow I$ is a natural transformation such that $(\oc A, \Delta, \epsilon)$ is an $R$-graded cocommutative comonoid; 
 $\partial=(\partial_{r}\colon\oc_r A \otimes A \rightarrow \oc_{r+1}A)$ is a family of natural transformations; and the following identities are 
 satisfied (for all $r,s,t \in R$):
\begin{enumerate}
\item preservation of comonoid comultiplication by $m$: $m_{r+s,t};\Delta_{r,s}=\Delta_{rt,st};m_{r,t} \otimes m_{s,t}$, \label{def-d(-2)}
\item preservation of comonoid counit by $m$: $m_{0,r};\epsilon=\epsilon$. \label{def-d(-1)} 
\item linear rule: $\partial_0;u=\epsilon \otimes 1$, \label{def-d1}
\item product rule: $\partial_{r+s+1};\Delta_{r+1,s+1}=\Delta_{r+1,s} \otimes 1;1 \otimes \partial_s+\Delta_{r,s+1} \otimes 1;1 \otimes \gamma;\partial_r \otimes 1$, \label{def-d2}
\item chain rule: $\partial_{rs+r+s};m_{s+1,t+1}=\Delta_{st+s,t} \otimes 1;m_{s,t+1} \otimes \partial_t;\partial_s$, \label{def-d3}
\item symmetry rule: $1 \otimes \gamma;\partial_r \otimes 1;\partial_{r+1}=\partial_r \otimes 1;\partial_{r+1}$.  \label{def-d4}
\end{enumerate}
\end{definition}
\begin{definition} \label{diff-mod}
A differential modality is a $0$-graded differential modality where $0$ is the trivial rig.
\end{definition}
For a differential modality, we will not write the indexes since they
are always equal to $0$, for instance we will write $m\colon\oc A
\rightarrow \oc\oc
A$ instead of $m_{0,0}\colon\oc_0 A \rightarrow \oc_0\oc_0 A$. We will need the horizontal composition of natural transformations.
\begin{definition}
If $\alpha\colon F(A) \rightarrow G(A)$ and $\beta\colon H(A) \rightarrow I(A)$ are natural transformations, then the natural transformation 
$\alpha \bullet \beta\colon F(H(A)) \rightarrow G(I(A))$ is defined by:
\begin{equation} \label{bullet-1}
(\alpha \bullet \beta)_A := F(\beta_A);\alpha_{I(A)}.
\end{equation}
\end{definition}
It is well-known (\cite{MACLANE},~II.5) that $\alpha \bullet \beta$ is equivalently given by another expression. 
\begin{proposition}
If $\alpha\colon F(A) \rightarrow G(A)$ and $\beta\colon H(A) \rightarrow I(A)$ are natural transformations, then we have
\begin{equation} \label{bullet-2}
(\alpha \bullet \beta)_A=\alpha_{H(A)};G(\beta_A).
\end{equation}
\end{proposition}
We obtain a strict monoidal category $(\mathsf{End}(\mathsf{C}),\bullet,1_\mathsf{C})$
where if $F,G \in \mathsf{End}(\mathsf{C})$, then $F \bullet G:= F \circ G$. The monoidal unit is the identity functor 
$1_{\mathsf{C}}\colon \mathsf{C} \rightarrow \mathsf{C}$ defined by $1_{\mathsf{C}}(X)=X$ and $1_{\mathsf{C}}(f)=f$. 
Moreover, if $F \in \mathsf{End}(\mathsf{C})$, then the identity $1_F$ is the natural transformation defined by $1_F(A)=1_{F(A)}\colon F(A) \rightarrow F(A)$. Note that for every functor $J \in \mathsf{End}(\mathsf{C})$, we have $(1_J \bullet \beta)_A=J(\beta_A)$ and $(\alpha \bullet 1_J)_A=\alpha_{J(A)}$.
\begin{proposition}
In $\mathsf{C}$, $1_A\colon A \rightarrow A$ is a natural transformation and if $\alpha_A\colon F(A) \rightarrow G(A)$ and $\beta_A\colon H(A) \rightarrow I(A)$ 
are two natural transformations, then $(\alpha \otimes \beta)_A=\alpha_A \otimes \beta_A \colon (F \otimes H)(A)=F(A) \otimes H(A) \rightarrow 
(G \otimes I)(A)=G(A) \otimes I(A)$ is a natural transformation.
\end{proposition}
We can thus make $\mathsf{End}(\mathsf{C})$ into a symmetric monoidal category $(\mathsf{End}(\mathsf{C}),\otimes,I)$. 
Here, $I \in \mathsf{End}(\mathsf{C})$ denotes the constant functor equal to $I$ given by $I(X)=I$ and $I(f)=1_X$. The two tensor products on $\mathsf{End}(\mathsf{C})$ 
are compatible in the following way.
\begin{proposition} \label{comp-two-tensor-products}
Let $\alpha\colon F(A) \rightarrow G(A)$, $\beta\colon H(A) \rightarrow I(A)$ and $\delta\colon J(A) \rightarrow K(A)$ be natural transformations in $\mathsf{End}(\mathsf{C})$. 
We have
\begin{equation*}
(\alpha \otimes \beta) \bullet \delta=(\alpha \bullet \delta) \otimes (\beta \bullet \delta)\colon F(J(A)) \otimes H(J(A)) \rightarrow G(K(A)) \otimes I(K(A)).
\end{equation*}
\end{proposition}
\begin{proof}
We have
\begin{align*}
((\alpha \otimes \beta) \bullet \delta)_A&=(F \otimes H)(\delta_A);(\alpha \otimes \beta)_{K(A)} &\text{\scriptsize{\cref{bullet-1}}} \\
&=F(\delta_A) \otimes H(\delta_A);\alpha_{K(A)} \otimes \beta_{K(A)} \\
&=(F(\delta_A);\alpha_{K(A)}) \otimes (H(\delta_A);\beta_{K(A)}) \\
&=(\alpha \bullet \delta)_A \otimes (\beta \bullet \delta)_A &\text{\scriptsize{\cref{bullet-1}}} \\
&=((\alpha \bullet \delta) \otimes (\beta \bullet \delta))_A.
\end{align*}
\end{proof}
The following lemma is obvious but will be useful when applied to $(\mathsf{End}(\mathsf{C}),\bullet,1_C)$.
\begin{lemma}
Let $(\mathbb{X},\otimes,I)$ be a monoidal category. Suppose given morphisms of the following type: $f\colon A \rightarrow B \otimes C$, $g_1\colon A \rightarrow E$, 
$g_2\colon D \rightarrow E$, $g_3\colon B \rightarrow E$, $g_4\colon C \rightarrow E$, $h\colon E \rightarrow E \otimes E$ such that the diagram
\begin{equation*}
\begin{tikzcd}
A \arrow[r, "g_1"] \arrow[d, "f"']        & E \arrow[d, "h"] \\
B \otimes C \arrow[r, "g_3 \otimes g_4"'] & E \otimes E     
\end{tikzcd}
\end{equation*}
commutes. Then, the diagram
\begin{equation*}
\begin{tikzcd}
A \otimes D \arrow[r, "g_1 \otimes g_2"] \arrow[d, "f \otimes 1_D"'] & E \otimes E \arrow[d, "h \otimes 1_E"] \\
B \otimes C \otimes D \arrow[r, "g_3 \otimes g_4 \otimes g_2"']      & E \otimes E                           
\end{tikzcd}
\end{equation*}
commutes.
\end{lemma}
In $(\mathsf{End}(\mathsf{C}),\bullet,1_C)$, we obtain the following.
\begin{proposition} \label{prop-end-1}
Suppose given natural transformations of the following type in $\mathsf{End}(\mathsf{C})$: $\alpha\colon F \rightarrow G \circ H$, $\beta_1\colon F \rightarrow J$, 
$\beta_2\colon I \rightarrow J$, $\beta_3\colon G \rightarrow J$, $\beta_4\colon H \rightarrow J$, $\delta\colon J \rightarrow J \circ J$ such that the diagram
\begin{equation*}
\begin{tikzcd}
F \arrow[r, "\beta_1"] \arrow[d, "\alpha"']        & J \arrow[d, "\delta"] \\
G \circ H \arrow[r, "\beta_3 \bullet \beta_4"'] & J \circ J   
\end{tikzcd}
\end{equation*}
commutes. Then, the diagram
\begin{equation*}
\begin{tikzcd}
F \circ I \arrow[d, "\alpha \bullet 1_I"'] \arrow[rr, "\beta_1 \bullet \beta_2"] &  & J \circ J \arrow[d, "\delta \bullet 1_J"] \\
G \circ H \circ I \arrow[rr, "\beta_3 \bullet \beta_4 \bullet \beta_2"']      &  & J \circ J \circ J                
\end{tikzcd}
\end{equation*}
commutes.
\end{proposition}
This other lemma will also be useful when applied to $(\mathsf{End}(\mathsf{C}),\bullet,1_C)$.
\begin{lemma}
Let $(\mathbb{X},\otimes,I)$ be a monoidal category. Suppose given morphisms of the following type: $f\colon B \rightarrow C \otimes D$, $g_1\colon E \rightarrow A$, 
$g_2\colon E \rightarrow C$, $g_3\colon E \rightarrow D$, $g_4\colon E \rightarrow B$, $h\colon E \rightarrow E \otimes E$ such that the diagram
\begin{equation*}
\begin{tikzcd}
E \arrow[d, "g_4"'] \arrow[rr, "h"] &  & E \otimes E \arrow[d, "g_2 \otimes g_3"] \\
B \arrow[rr, "f"']                  &  & C \otimes D                             
\end{tikzcd}
\end{equation*}
commutes. Then, the diagram
\begin{equation*}
\begin{tikzcd}
E \otimes E \arrow[d, "g_1 \otimes g_4"'] \arrow[rr, "1_E \otimes h"] &  & E \otimes E \otimes E \arrow[d, "g_1 \otimes g_2 \otimes g_3"] \\
A \otimes B \arrow[rr, "1_A \otimes f"']                              &  & A \otimes C \otimes D                                         
\end{tikzcd}
\end{equation*}
commutes.
\end{lemma}
\begin{proposition} \label{prop-end-2}
Suppose given natural transformations of the following type in $\mathsf{End}(\mathsf{C})$: $\alpha\colon G \rightarrow H \circ I$, $\beta_1\colon J \rightarrow F$, 
$\beta_2\colon J \rightarrow H$, $\beta_3\colon J \rightarrow I$, $\beta_4\colon J \rightarrow G$, $\delta\colon J \rightarrow J \circ J$ such that the diagram
\begin{equation*}
\begin{tikzcd}
J \arrow[d, "\beta_4"'] \arrow[rr, "\delta"] &  & J \circ J \arrow[d, "\beta_2 \bullet \beta_3"] \\
G \arrow[rr, "\alpha"']                      &  & H \circ I                                     
\end{tikzcd}
\end{equation*}
commutes. Then, the diagram
\begin{equation*}
\begin{tikzcd}
J \circ J \arrow[d, "\beta_1 \bullet \beta_4"'] \arrow[rr, "1_J \bullet \delta"] &  & J \circ J \circ J \arrow[d, "\beta_1 \bullet \beta_2 \bullet \beta_3"] \\
F \circ G \arrow[rr, "1_F \bullet \alpha"']                                      &  & F \circ H \circ I                                                     
\end{tikzcd}
\end{equation*}
commutes.
\end{proposition}
\begin{definition}
An ordered rig $(R,\le)$ is a rig $R$ together with a partial order $\le$ on $R$ such that $r+s \le t+u$ and $rs \le tu$ whenever $r \le t$ and $s \le u$ and $rs \le tu$. 
If $(R,\le)$ and $(S,\le)$ are two ordered rigs, then ordered rig homomorphism from $R$ to $S$ is any rig homomorphism $\rho$ from $R$ to $S$ such that $\rho(r) \le \rho(s)$ 
whenever $r \le s$.
\end{definition}
\begin{example}
$(\mathbb{N},\le)$ and $(\mathbb{R}_{\ge 0}, \le)$ are ordered rigs. However, $(\mathbb{Z},\le)$ and $(\mathbb{R},\le)$ are not ordered rigs since $-2 \le -1$ and $-1 \le -1$, 
but $2>1$.
\end{example}
\begin{definition} \label{def:filtered-diff-mod}
Let $(R,\le)$ be an ordered rig. An $R$-filtered differential modality is an $R$-graded differential modality $(\oc,m,u,\Delta,\epsilon,\partial)$ together with a natural 
transformation $t_{a,b}\colon\oc_aA \rightarrow \oc_bA$ for every $(a,b) \in R^2$ with $a \ge b$, such that $t_{a,a}=1_{\oc_a A}$ for every $a \in R$, $t_{a,b};t_{b,c}=t_{a,c}$ 
for all $a \ge b \ge c$ and the following diagrams commute whenever they make sense (we omit the indexes of the natural transformations for readabilty):
\begin{equation*}
\begin{tikzcd}
\oc_{a+b}A \arrow[rr, "t"] \arrow[d, "\Delta"']     &  & \oc_{c+d}A \arrow[d, "\Delta"] \\
\oc_a A \otimes \oc_b A \arrow[rr, "t \otimes t"'] &  & \oc_c A \otimes \oc_d A      
\end{tikzcd}
\qquad
\begin{tikzcd}
\oc_{ab}A \arrow[rr, "t"] \arrow[d, "m"'] &  & \oc_{cd}A \arrow[d, "m"] \\
\oc_a\oc_bA \arrow[rr, "t\bullet t"']     &  & \oc_cA\oc_dA            
\end{tikzcd}
\qquad
\begin{tikzcd}
\oc_aA \otimes A \arrow[d, "\partial"'] \arrow[rr, "t \otimes 1"] &  & \oc_bA \otimes A \arrow[d, "\partial"] \\
\oc_{a+1}A \arrow[rr, "t"']                                       &  & \oc_{b+1}A~.                            
\end{tikzcd}
\end{equation*}
\end{definition}
Note that every $R$-graded differential modality is filtered in a trivial way by choosing $\le$ to be the discrete order on $R$, that is, $r \le s$ iff $r=s$. In this way, 
a differential modality is the same as a $0$-filtered modality where the trivial rig $0$ is equipped with the discrete order.
\begin{definition} \label{def:morphism-of-graded-diff-mod}
Let $M=(\oc,m,u,\Delta,\epsilon,\partial)$ be an $R$-graded differential modality, $M'=(\oc',m',u',\Delta',\epsilon',\partial')$ an $S$-graded differential modality and 
$\rho\colon S \rightarrow R$ a rig homomorphism. A $\rho$-graded morphism of graded differential modalities from $M$ to $M'$ is a family of natural transformations 
$\phi_r\colon\oc_{\rho(r)} A \rightarrow \oc'_{r}A$ such that the following diagrams commute (we omit the indexes of the natural transformations for readability):
\begin{equation*}
\begin{tikzcd}
\oc_{\rho(r)\rho(s)} A \arrow[rr, "{m}"] \arrow[d, "\phi"'] &  & \oc_{\rho(r)}\oc_{\rho(s)}A \arrow[d, "\phi \bullet \phi"] \\
\oc_{rs}' \arrow[rr, "{m'}"']                        &  & \oc_{r}'\oc_{s}'A~,                              
\end{tikzcd}
\qquad
\begin{tikzcd}
\oc_1A \arrow[rr, "u"] \arrow[d, "\phi"'] &  & A \\
\oc_1'A~, \arrow[rru, "u'"']                  &  &  
\end{tikzcd}
\qquad
\begin{tikzcd}
\oc_{\rho(r)+\rho(s)}A \arrow[rr, "{\Delta}"] \arrow[d, "\phi"'] &  & \oc_{\rho(r)} A \otimes \oc_{\rho(s)}A \arrow[d, "\phi \otimes \phi"] \\
\oc_{r+s}'A \arrow[rr, "{\Delta'}"']                       &  & \oc_{r}' A \otimes \oc_{s}'A~,                                
\end{tikzcd}
\end{equation*}
\begin{equation*}
\begin{tikzcd}
\oc_0A \arrow[r, "\phi"] \arrow[rd, "\epsilon"'] & \oc_0'A \arrow[d, "\epsilon'"] \\
                                                   & I~,                             
\end{tikzcd}
\qquad
\begin{tikzcd}
\oc_{\rho(n)}A\otimes A \arrow[rr, "\phi \otimes 1"] \arrow[d, "\partial"'] &  & \oc_{\rho(n)}'A\otimes A \arrow[d, "\partial'"] \\
\oc_{\rho(n)+1}A \arrow[rr, "\phi"']                                      &  & \oc_{n+1}'A~.                                
\end{tikzcd}
\end{equation*}
\end{definition}
\begin{definition} \label{def-morphism-filtered}
Let $(R,\le)$ and $(S,\le)$ be two ordered rigs. Let $M=(\oc,m,u,\Delta,\epsilon,\partial,t)$ be an $R$-filtered differential modality, 
$M'=(\oc',m',u',\Delta',\epsilon',\partial',t')$ an $S$-filtered differential modality and $\rho\colon S \rightarrow R$ an ordered rig homomorphism. 
A $\rho$-filtered morphism of filtered differential modalities from $M$ to $M'$ is a $\rho$-graded morphism of graded differential modalities from $M$ to $M$' 
such that the following diagram commutes whenever it makes sense (we omit the indexes of the natural transformations for readabilty):
\begin{equation*}
\begin{tikzcd}
\oc_{\rho(r)}A \arrow[d, "\phi"'] \arrow[rr, "t"] &  & \oc_{\rho(s)}A \arrow[d, "\phi"] \\
\oc_rA \arrow[rr, "t"']                           &  & \oc_sA~.                          
\end{tikzcd}
\end{equation*}
\end{definition}
\section{Technicalities on inductive definitions, partitions and permutations} \label{sec:3}
The content of this section will mainly be used to state and prove the higher-order rules for deriving transformations in \cref{prop:higher-order-identities}. 
We start with the definition of $\partial^n\colon \oc A \otimes A^{\otimes n} \rightarrow \oc A$ and three easy results about $\partial^n$.
\begin{definition} \label{def:higher-deriving}
Let $(\oc,m,u,\Delta,\epsilon,\partial)$ be a differential modality. We define $\partial^n\colon\oc A \otimes A^{\otimes n} \rightarrow \oc A$ for every $n \ge 0$ by induction: 
\begin{enumerate}
\item $\partial^0:=1_{\oc A}$, \label{def-d-ind-1}
\item $\partial^{n+1}:=\partial^n \otimes 1_A;\partial$. \label{def-d-ind-2}
\end{enumerate}
In particular, we have $\partial^1=\partial\colon\oc A \otimes A \rightarrow \oc A$.
\end{definition}
The next proposition can easily be proven by induction.
\begin{proposition} \label{nat-dn}
For every $n \ge 0$, $\partial^n\colon \oc A \otimes A^{\otimes n} \rightarrow \oc A$ is a natural transformation.
\end{proposition}
\begin{lemma}
Let $(\oc,m,u,\Delta,\epsilon,\partial)$ be a differential modality on an additive symmetric monoidal category $(\mathsf{C},\otimes,I)$. For every $k \in \mathbb{N}$, we have 
\begin{equation} \label{reverse-higher}
\partial^{k+1}=\partial \otimes 1_{A^{\otimes k}};\partial^k. 
\end{equation}
\end{lemma}
\begin{proof}
By induction on $k$.
\begin{itemize}
\item Base case ($k=0$): The LHS is $\partial^1=\partial$ and the RHS is $\partial \otimes 1_{I};1_{\oc A}=\partial$, hence the equality. 
\item Induction step: Suppose that \cref{reverse-higher} holds for some $k \in \mathbb{N}$. Then, we have:
\begin{align*}
\partial^{(k+1)+1} &= \partial^{k+1} \otimes 1_A;\partial &\text{\scriptsize{\cref{def-d-ind-2}}} \\
&=  (\partial \otimes 1_{A^{\otimes k}};\partial^k) \otimes 1_A; \partial &\text{\scriptsize{(induction hypothesis)}} \\
&= \partial \otimes 1_{A^{\otimes (k+1)}};\partial^k \otimes 1_A;\partial &\text{\scriptsize{(functoriality of }}\scriptstyle{-\otimes 1_A)} \\
&=  \partial \otimes 1_{A^{\otimes (k+1)}};\partial^{k+1} &\text{\scriptsize{\cref{def-d-ind-2}}}
\end{align*}
\end{itemize}
\end{proof}
\begin{lemma}
Let $(\oc,m,u,\Delta,\epsilon,\partial)$ be a differential modality on an additive symmetric monoidal category $(\mathsf{C},\otimes,I)$. For any $(k, l) \in \mathbb{N}^2$, we have:
\begin{equation} \label{d-k-l}
\partial^{k+l}=\partial^k \otimes 1_{A^{\otimes l}};\partial^l.
\end{equation}
\end{lemma}
\begin{proof}
By induction on $(k,l)$.
\begin{itemize}
\item Base case ($(k,l)=(0,0)$): the LHS is $\partial^0=1_{\oc A}$ and the RHS is 
$\partial^0 \otimes 1_{A^{\otimes 0}};\partial^0=1_{\oc A} \otimes 1_I;1_{\oc A}=1_{\oc A};1_{\oc A}=1_{\oc A}$. We thus have equality.
\item Induction step on $l$: Suppose that \cref{d-k-l} holds for some $(k,l) \in \mathbb{N}^2$. We then have: 
\begin{align*}
\partial^{k+(l+1)}&=\partial^{(k+l)+1}& \text{\scriptsize{(reparenthesing)}} \\
&=\partial^{k+l} \otimes 1_A;\partial& \text{\scriptsize{\cref{def-d-ind-2}}} \\
&=(\partial^k \otimes 1_{A^{\otimes l}};\partial^l) \otimes 1_A;\partial& \text{\scriptsize{(induction hypothesis)}} \\
&=\partial^k \otimes 1_{A^{\otimes (l+1)}};\partial^l \otimes 1_A;\partial&\text{\scriptsize{(functoriality of $-\otimes 1_A$)}} \\
&=\partial^k \otimes 1_{A^{\otimes (l+1)}};\partial^{l+1}& \text{\scriptsize{\cref{def-d-ind-2}}}\\
\end{align*}
\item Induction step on $k$: Suppose that \cref{d-k-l} holds for some $(k,l) \in \mathbb{N}^2$. We then have:
\begin{align*}
\partial^{(k+1)+l} &=\partial^{k+(l+1)}& \text{\scriptsize{(reparenthesing)}} \\
&=\partial^k \otimes 1_{A^{\otimes (l+1)}};\partial^{l+1}& \text{\scriptsize{(induction step on $l$)}} \\
&=\partial^k \otimes 1_{A^{\otimes (l+1)}};\partial \otimes 1_{A^{\otimes l}};\partial^l &\text{\scriptsize{\cref{reverse-higher}}} \\
&=\partial^k \otimes 1_A \otimes 1_{A^{\otimes l}};\partial \otimes 1_{A^{\otimes l}};\partial^l &\text{\scriptsize{functoriality of $- \otimes -$}}
\\
&=(\partial^k \otimes 1_A;\partial) \otimes 1_{A^{\otimes l}};\partial^l &\text{\scriptsize{functoriality of $-\otimes 1_{A^{\otimes l}}$}} \\
&=\partial^{k+1} \otimes 1_{A^{\otimes l}};\partial^l. &\text{\scriptsize{\cref{def-d-ind-2}}} 
\end{align*}
\end{itemize}
\end{proof}
The next definition is necessary to write down some of the identities in \cref{prop:higher-order-identities}.
\begin{definition} \label{def-tau-pi}
\begin{itemize}
\item If $X$ is a set, then a partition of $X$ is a subset $\pi \subseteq \mathcal{P}(X)$ consisting of nonempty disjoints subsets of $X$ whose union is $X$. 
\item For every $n \ge 0$, we define $\mathcal{H}(n) \subseteq \mathcal{P}(\mathcal{P}(\{1,\dots,n\}))$ as the set of all partitions of $\{1,\dots,n\}$. 
We thus have $\mathcal{H}(0)=\{\emptyset\}$.
\item Let $n \ge 0$ and $\pi$ be a partition of $\{1,\dots,n\}$. Write $\pi=\{s_1,\dots,s_{|\pi|}\}$ where $\mathrm{max}\,s_1<\dots<\mathrm{max}\,s_{|\pi|}$ 
and $s_i=\{k_i^1<\dots<k_i^{|s_i|}\}$ for every $1 \le i \le |\pi|$. We define the permutation $\tau_{\pi} \in S_{|\pi|+n}$ as follows:
\begin{equation*}
\left\{
\begin{aligned}
\tau_\pi(i):=&~(1+|s_1|)+\dots+(1+|s_{i-1}|)+1& \forall 1 \le i \le |\pi|, \\
\tau_\pi(|\pi|+k_i^j):=&~(1+|s_1|)+\dots+(1+|s_{i-1}|)+1+j& \forall 1 \le i \le |\pi| \text{ and }1 \le j \le |s_i|. \end{aligned}
\right.
\end{equation*}
In any symmetric monoidal category $(\mathsf{C},\otimes,I)$, we thus obtain a natural transformation\footnote{The natural transformation $\overline{\sigma}$ for 
any permutation $\sigma \in S_n$ is defined in the paragraph ``Notations and conventions'' of the introduction.}
\begin{equation*}
\begin{tikzcd}
A_1 \otimes \dots \otimes A_{|\pi|} \otimes B_1 \otimes \dots \otimes B_n \arrow[d, "{(\overline{\tau_\pi})_{A_1,\dots,A_{|\pi|},B_1,\dots,B_n}}"]                                   \\
A_1 \otimes B_{k_1^1} \otimes \dots \otimes B_{k_1^{|s_1|}} \otimes \dots \otimes A_{|\pi|} \otimes B_{k_{|\pi|}^1} \otimes \dots \otimes B_{k_{|\pi|}^{|s_{|\pi|}|}}~.
\end{tikzcd}
\end{equation*}
\item If $(\oc,m,u,\Delta,\epsilon,\partial)$ is a differential modality on an additive symmetric monoidal category $(\mathsf{C},\otimes,I)$, then we define 
$\Delta^n\colon\oc A \rightarrow (\oc A)^{\otimes n}$ for every $n \ge 1$ by induction: we set $\Delta^1=1_{\oc A}$ and 
\begin{equation} \label{Delta-rec}
\Delta^{n+1}=\Delta^n;1_{(\oc A)^{\otimes(n-1)}} \otimes \Delta.
\end{equation}
\end{itemize}
\end{definition}
In particular, we have $\Delta^2=\Delta\colon\oc A \rightarrow \oc A \otimes \oc A$.
\begin{proposition}
Let $(\oc,m,u,\Delta,\epsilon,\partial)$ be a differential modality on an additive symmetric monoidal category $(\mathsf{C},\otimes,I)$. For every $n \ge 1$, 
we have $\Delta^{n+1}=\Delta;\Delta^n \otimes 1_{\oc A}$.
\end{proposition}
If $X$ and $I$ are two sets and $A \in X^I$, then we define the disjoint union $\underset{i \in I}{\bigsqcup}A(i)$ as
\begin{equation} \label{def:disjoint-union}
\underset{i \in I}{\bigsqcup}A(i):=\{(i,a)~|~i \in I,~a \in A(i)\}.
\end{equation}
The two following propositions will be used in the proof of \cref{id5}.
\begin{proposition} \label{prop-bij}
For every $n \ge 0$, we have a bijection
\begin{equation} \label{the-bij}
\rho \colon \underset{\pi \in \mathcal{H}(n)}{\bigsqcup}\{1,\dots,|\pi|+1\} \simeq \mathcal{H}(n+1).
\end{equation}
\end{proposition}
\begin{proof}
\emph{Definition of $\rho$}: Let $(\pi,i) \in \underset{\pi \in \mathcal{H}(n)}{\bigsqcup}\{1,\dots,|\pi|+1\}$. Write 
\begin{equation*}
\pi=\{s_1,\dots,s_{|\pi|}\}
\end{equation*}
where $\mathrm{max}\,s_1<\dots<\mathrm{max}\,s_{|\pi|}$. 
We obtain a partition $\pi+1$ of $\{2,\dots,n+1\}$ as follows:
\begin{equation*}
\pi+1:=
\{\{t+1~|~t \in s\}~|~s \in \pi\}.
\end{equation*}
If $i=1$, we define
\begin{equation} \label{eq-rho-pi-1}
\rho(\pi,i):=(\pi+1) \cup \{1\}.
\end{equation}
If $2 \le i \le |\pi|+1$, we define 
\begin{equation} \label{eq-rho-pi-i}
\rho(\pi,i):=((\pi+1) \backslash \{s_{i-1}+1\}) \cup \{(s_{i-1}+1) \cup \{1\}\}.
\end{equation}

\emph{Definition of $\chi$, the inverse of $\rho$}: Let $\theta \in \mathcal{H}(n+1)$. Write
\begin{equation*}
\theta=\{t_1,\dots,t_{|\theta|}\}
\end{equation*}
where $\mathrm{max}\,t_1<\dots<\mathrm{max}\,t_{|\theta|}$. 
Let $1 \le j \le |\theta|$ such that $1 \in t_j$. If $|t_j| \ge 2$, we define
\begin{equation*}
\chi_0(\theta):=(\theta \backslash \{t_j\}) \cup \{t_j \backslash \{1\}\}
\end{equation*}
and $i=j+1$. If $|t_j|=1$, we define 
\begin{equation*}
\chi_0(\theta):=\theta \backslash \{t_j\}
\end{equation*}
and $i=1$. We obtain a partition $\chi_0(\theta)$ of $\{2,\dots,n+1\}$. We then define
\begin{equation*}
\chi_*(\theta):=\{\{t-1~|~t \in s\}~|~s \in \chi_0(\theta)\}.
\end{equation*}
Finally, we define $\chi(\theta):=(\chi_*(\theta),i)$.

We thus obtain two functions 
\begin{equation*}
\rho\colon \underset{\pi \in \mathcal{H}(n)}{\bigsqcup}\{1,\dots,|\pi|+1\} \longrightarrow \mathcal{H}(n+1)
\end{equation*}
and
\begin{equation*}
\chi\colon \mathcal{H}(n+1) \longrightarrow \underset{\pi \in \mathcal{H}(n)}{\bigsqcup}\{1,\dots,|\pi|+1\}.
\end{equation*}

\emph{Checking that $\chi \circ \rho=\mathsf{id}$}: Let $(\pi,i)$ in the domain of $\rho$. If $i=1$, let $\theta=(\pi+1) \cup \{1\}$. We must prove that 
$\chi(\theta)=(\pi,i)$. We have $|\theta|=|\pi|+1$, $t_1=\{1\}$, $t_2=s_1+1$, \dots, $t_{|\pi|+1}=s_{|\pi|}+1$. Thus $j=1$ and $|t_j|=1$. It follows that 
$\chi_0(\theta)=\{s_1+1,\dots,s_{|\pi|}+1\}$. Then $\chi_*(\theta)=\{s_1,\dots,s_{|\pi|}\}=\pi$. We obtain that $\chi(\theta)=(\pi,i)$.

If $2 \le i \le |\pi|+1$, let $\theta=((\pi+1) \backslash \{s_{i-1}+1\}) \cup \{(s_{i-1}+1) \cup \{1\}\}$. We must prove that $\chi(\theta)=(\pi,i)$. We have 
$|\theta|=|\pi|$, $t_1=s_1+1$, \dots, $t_{i-2}=s_{i-2}+1$, $t_{i-1}=(s_{i-1}+1) \cup \{1\}$, $t_i=s_i+1$, \dots, $t_{|\pi|}=s_{|\pi|}+1$. Thus $j=i-1$ and 
$|t_j|=|s_{i-1}|+1 \ge 2$. It follows that $\chi_0(\theta)=\{s_1+1,\dots,s_{|\pi|}+1\}$. Then $\chi_*(\theta)=\{s_1,\dots,s_{|\pi|}\}$. We obtain that $\chi(\theta)=(\pi,i)$.

\emph{Checking that $\rho \circ \chi=\mathsf{id}$}: Let $\theta$ in the domain of $\chi$. Let $\pi=\chi_*(\theta)$ and $i,j$ be as in the definition of 
$\chi$. If $|t_j| \ge 2$, we have $i=j+1 \ge 2$, $|\pi|=|\theta|$ and $\pi+1=\chi_0(\theta)=\{t_1,\dots,t_{j-1},t_j \backslash \{1\},t_{j+1},\dots,t_{|\theta|}\}$. 
It follows that $\rho(\pi,i)=(\chi_0(\theta) \backslash \{t_j \backslash \{1\}\}) \cup \{(t_j \backslash \{1\}) \cup \{1\}\}=(\chi_0(\theta) 
\backslash \{t_j \backslash \{1\}\}) \cup \{t_j\}=\theta$.

If $|t_j|=1$, we have $j=1$, $t_j=\{1\}$, $i=1$ and $\pi+1=\chi_0(\theta)=\{t_2,\dots,t_{|\theta|}\}$. Thus $\rho(\pi,i)=(\pi+1) \cup \{1\}=\theta$.

\emph{Remark on the case $n=0$}: The proof is valid for any $n \ge 0$, thus in particular when $n=0$. In this case the LHS in \cref{the-bij} is 
$\underset{\pi \in \{\emptyset\}}{\bigsqcup}\{1\}=\{(\emptyset,1)\}$ and the RHS is $\mathcal{H}(1)=\{\{1\}\}$. By what's above, $\rho$ is defined by 
$\rho((\emptyset,1))=\{1\}$ and $\chi$ is defined by $\chi(\{1\})=(\emptyset,1)$, which are by the way the unique functions with such domains and codomains.
\end{proof}
The following proposition records some facts about the bijection $\cref{the-bij}$.
\begin{proposition} \label{recorded-facts}
Let $n \ge 0$ and $\pi \in \mathcal{H}(n)$. Write $\pi=\{s_1,\dots,s_{|\pi|}\}$ where $\mathrm{max}\,s_1<\dots<\mathrm{max}\,s_{|\pi|}$ and for every 
$1 \le i \le |\pi|+1$, write $\rho(\pi,i)=\{t_1,\dots,t_{|\rho(\pi,i)|}\}$ where $\mathrm{max}\,t_1<\dots<\mathrm{max}\,t_{|\rho(1,\pi)|}$.
\begin{itemize}
\item 
If $i=1$, we have:
\begin{equation} \label{rho-1-first}
|\rho(\pi,1)|=|\pi|+1,
\end{equation}
\begin{equation} \label{rho-1-second}
|t_1|=1,
\end{equation}
\begin{equation} \label{rho-1-third}
|t_r|=|s_{r-1}| \text{ for every }2 \le r \le |\rho(\pi,1)|.
\end{equation}
\item If $2 \le i \le |\pi|+1$, we have
\begin{equation} \label{rho-2-first}
|\rho(\pi,i)|=|\pi|,
\end{equation}
\begin{equation} \label{rho-2-second}
|t_r|=|s_r| \text{ for every }1 \le r \le |\rho(\pi,i)| \text{ such that } r \neq i-1,
\end{equation}
\begin{equation} \label{rho-2-third}
|t_{i-1}|=|s_{i-1}|+1.
\end{equation}
\end{itemize}
\end{proposition}
The next definition is needed to write down \cref{id3}.
\begin{definition}
\begin{itemize}
\item Let $k,l \ge 0$. A $(k,l)$-unshuflle is any permutation $\sigma \in S_{k+l}$ such that the inequalities
\begin{align*}
&\sigma^{-1}(1)<\sigma^{-1}(2)<\dots<\sigma^{-1}(k) \\
\text{ and }&\sigma^{-1}(k+1)<\sigma^{-1}(k+2)<\dots<\sigma^{-1}(k+l)
\end{align*}
hold.
\item We define $\mathsf{Unsh}(n,k) \subseteq S_n$ as the set of all the $(k,n-k)$-unshuffles.
\item We define the natural transformation
\begin{equation*}
\mathsf{unsh}(n,k)_A \colon A^{\otimes n} \rightarrow A^{\otimes n}
\end{equation*}
for $n \ge 0$ and $0 \le k \le n$ by the formula
\begin{equation} \label{formula-unsh}
\mathsf{unsh}(n,k)_A:=\underset{\sigma \in \mathsf{Unsh}(n,k)}{\sum}\overline{\sigma}_{A,\dots,A}~.
\end{equation}
\end{itemize}
\end{definition}
The rest of the section is devoted to the proof of \cref{prop-unsh} which will be used to prove \cref{id3}.
It will be convenient to use the following definition.
In this definition, we use the group rigs $\mathbb{N}[S_n]$ which are defined in the paragraph ``Notations and conventions'' of the introduction.
\begin{definition}
The additive symmetric strict monoidal category $\mathsf{S}_+$ has as objects all the integers $n \ge 0$. If $n,p \ge 0$ are such that $n \neq p$, then 
$\mathsf{S}_+(n,p) := \emptyset$. For every $n \ge 0$, we define $\mathsf{S}_+(n,n) := \mathbb{N}[S_n]$. The composite $a \circ b$ where $a,b \in S_+(n,n)$ is 
their product $ab \in S_+(n,n)$ as defined in the introduction. The identity on $n$ is the multiplicative identity in $\mathsf{S}_+(n,n)$. Let $n,p \ge 0$.
 We define $n \otimes p:= n+p$. If $\sigma \in S_n$ and $\tau \in S_p$, then we define $\sigma \otimes \tau \in \mathsf{S}_+(n+p)$ by
\begin{equation*}
(\sigma \otimes \tau)(k)=
\left\{
\begin{aligned}
\sigma(k) &\text{ if }1 \le k \le n, \\
n+\tau(k-n) &\text{ if }n+1 \le k \le n+p. 
\end{aligned}
\right.
\end{equation*}
If $\sigma_1+\dots+\sigma_r \in \mathsf{S}_+(n,n)$ and $\tau_1+\dots+\tau_s \in \mathsf{S}_+(p,p)$, the we define
\begin{equation*}
(\sigma_1+\dots+\sigma_r) \otimes (\tau_1+\dots+\tau_s):=\underset{\substack{1 \le i \le r \\ 1 \le j \le s}}{\sum}\sigma_i \otimes \tau_j.
\end{equation*}
The monoidal unit is $0$. The symmetry $\gamma_{n,p}\colon n \otimes p \rightarrow p \otimes n$ is the permutation in $S_{n+p}$ defined by
\begin{equation*}
\gamma_{n,p}(k)=
\left\{
\begin{aligned}
k+p &\text{ if }1 \le k \le n, \\
k-n&\text{ if } n+1 \le k \le n+p.
\end{aligned}
\right.
\end{equation*}
\end{definition}
The next lemma will be useful when we consider the natural transformations $\mathsf{unsh}(n,k)$ in $S_+$.
\begin{lemma} \label{lemma-bar}
For all $n \ge 0$ and $\sigma \in S_n$, we have $\overline{\sigma}_{1,\dots,1}=\sigma \in \mathsf{S}_+(n,n)$. 
\end{lemma}
\begin{proof}
We follow the definition of $\overline{\sigma}_{A_1,\dots,A_n}$ in the paragraph ``Notations and conventions'' of the introduction. Here, we will have $A_1=\dots=A_n=1$.
 If $n=0$, or $n=1$, then the identity is clear. Suppose that $n \ge 2$. 

We first write $\sigma=(i_1,i_1+1);\dots;(i_r,i_r+1)$ for some $r \ge 0$ and $i_1,\dots,i_r \in \{1,\dots,n\}$. We then have 
\begin{align*}
\overline{\sigma}_{1,\dots,1}=&~1_{i_1-1} \otimes \gamma_{1,1} \otimes 1_{n-(i_1+1)}; \\
&~1_{i_2-1} \otimes \gamma_{1,1} \otimes 1_{n-(i_2+1)}; \\
&~\dots \\
&~1_{i_r-1} \otimes \gamma_{1,1} \otimes 1_{n-(i_r+1)}\colon \\
&~n \longrightarrow n.
\end{align*}
But it easy to check that for all $a,b \ge 0$ such that $a+b=n-2$, we have 
\begin{equation*}
1_a \otimes \gamma_{1,1} \otimes 1_b=(a+1,a+2) \in S_n
\end{equation*}
so that
\begin{align*}
\overline{\sigma}_{1,\dots,1}=(i_1,i_1+1);(i_2;i_2+1);\dots;(i_r,i_r+1)=\sigma \in S_n.
\end{align*}
\end{proof}
The following definition and proposition \cref{prop:strict-fun} will let us transfer certain identities from $(S_+,\otimes,I)$ to an arbitrary additive symmetric strict 
monoidal category.
\begin{definition} \label{def-mon-fun}
Let $(\mathsf{C},\otimes,I)$ and $(\mathsf{D},\otimes,I)$ be additive symmetric strict monoidal categories. An additive symmetric strict monoidal functor from 
$(\mathsf{C},\otimes,I)$ to $(\mathsf{D},\otimes,I)$ is a functor $F\colon \mathsf{C} \rightarrow \mathsf{D}$ such that we have the following equality of functors:
\begin{equation*}
F(-\otimes -)=F(-) \otimes F(-),
\end{equation*}
$F(I)=I$, $\gamma_{F(A),F(B)}=F(\gamma_{A,B})$ for all $A,B \in \mathsf{C}$, $F(f+g)=F(f)+F(g)$ and $F(0)=0$.
\end{definition}
The four following lemmas will be used to prove \cref{prop:strict-fun}.
\begin{lemma}
Let $(\mathsf{C},\otimes,I)$ and $(\mathsf{D},\otimes,I)$ be additive symmetric strict monoidal categories and let $F$ be an additive symmetric strict monoidal functor 
from $(\mathsf{C},\otimes,I)$ to $(\mathsf{D},\otimes,I)$. For all $n \ge 0$, $\sigma \in \mathfrak{S}_n$ and $A_1,\dots,A_n$, we have 
\begin{equation} \label{preserve-sigma}
F(\overline{\sigma}_{A_1,\dots,A_n})=\overline{\sigma}_{F(A_1),\dots,F(A_n)}.
\end{equation}
\end{lemma}
\begin{proof}
We first write $\sigma=(i_1,i_1+1);\dots;(i_r,i_r+1)$ for some $r \ge 0$ and $i_1,\dots,i_r \in \{1,\dots,n\}$. We then have 
\begin{align} \label{big-sigma-1}
\overline{\sigma}_{A_1,\dots,A_n}=&~1_{A_1^1 \otimes \dots \otimes A_{i_1-1}^1} \otimes \gamma_{A_{i_1}^1,A_{i_1+1}^1} \otimes 1_{A_{i_1+2}^1 \otimes \dots \otimes A_n^1}; \\
&~1_{A_1^2 \otimes \dots \otimes A_{i_1-1}^2} \otimes \gamma_{A_{i_2}^2,A_{i_2+1}^2} \otimes 1_{A_{i_2+2}^2 \otimes \dots \otimes A_n^2}; \notag \\
&~\dots \notag \\
&~1_{A_1^r \otimes \dots \otimes A_{i_r-1}^r} \otimes \gamma_{A_{i_r}^r,A_{i_r+1}^r} \otimes 1_{A_{i_r+2}^r \otimes \dots \otimes A_n^r}\colon \notag \\
&~A_1 \otimes \dots \otimes A_n \longrightarrow A_{\sigma^{-1}(1)} \otimes \dots \otimes A_{\sigma^{-1}(n)} \notag
\end{align}
where if we define
\begin{equation*}
\tau_k:=\Big((i_1,i_1+1);\dots;(i_k,i_k+1)\Big)^{-1}
\end{equation*}
for every $1 \le k \le r$, then, we have
\begin{equation*}
A^k_s=A_{\tau_{k-1}(s)}
\end{equation*}
for all $1 \le k \le r$ and $1 \le s \le n$. In the same way, we have
\begin{align} \label{big-sigma-2}
\overline{\sigma}_{F(A_1),\dots,F(A_n)}=&~1_{F(A_1^1) \otimes \dots \otimes F(A_{i_1-1}^1)} \otimes \gamma_{F(A_{i_1}^1),F(A_{i_1+1}^1)} \otimes 1_{F(A_{i_1+2}^1) 
\otimes \dots \otimes F(A_n^1)}; \\
&~1_{F(A_1^2) \otimes \dots \otimes F(A_{i_1-1}^2)} \otimes \gamma_{F(A_{i_2}^2),F(A_{i_2+1}^2)} \otimes 1_{F(A_{i_2+2}^2) \otimes \dots \otimes F(A_n^2)}; \notag \\
&~\dots \notag \\
&~1_{F(A_1^r) \otimes \dots \otimes F(A_{i_r-1}^r)} \otimes \gamma_{F(A_{i_r}^r),F(A_{i_r+1}^r)} \otimes 1_{F(A_{i_r+2}^r) \otimes \dots \otimes F(A_n^r)}\colon \notag\\
&~F(A_1) \otimes \dots \otimes F(A_n) \longrightarrow F(A_{\sigma^{-1}(1)}) \otimes \dots \otimes F(A_{\sigma^{-1}(n)}). \notag
\end{align}
By using \cref{def-mon-fun}, \cref{big-sigma-1,big-sigma-2}, we conclude that \cref{preserve-sigma} holds.
\end{proof}
The three following lemmas are consequences of Mac Lane's coherence theorem for symmetric monoidal categories. In each case and with the notations of \cite{COHERENCE}, we have two natural transformations which are morphisms 
in $\mathsf{Pit}_\otimes(\mathsf{C})$ with same domain and codomain, and thus are equal.
\begin{lemma} \label{comp-bar}
Let $(\mathsf{C},\otimes,I)$ be an additive symmetric strict monoidal category. For all $n \ge 0$, $\sigma,\tau \in S_n$, the two natural transformations\footnote{The 
first natural transformation is the composite of the natural transformations $\overline{\sigma}_{A_1,\dots,A_n}$ followed by the natural transformation
$(\overline{\tau}_{\sigma^{-1}})_{A_1,\dots,A_n}=\overline{\tau}_{A_{\sigma^{-1}(1)},\dots,A_{\sigma^{-1}(n)}}\colon A_{\sigma^{-1}(1)} \otimes \dots \otimes A_{\sigma^{-1}(n)} 
\rightarrow A_{\sigma^{-1}(\tau^{-1}(1))} \otimes \dots \otimes A_{\sigma^{-1}(\tau^{-1}(n))}$.}
\begin{equation} \label{first-nart}
\overline{\sigma}_{A_1,\dots,A_n};\overline{\tau}_{A_{\sigma^{-1}(1)},\dots,A_{\sigma^{-1}(n)}}\colon A_1 \otimes \dots \otimes A_n \rightarrow A_{\sigma^{-1}(\tau^{-1}(1))} 
\otimes \dots \otimes A_{\sigma^{-1}(\tau^{-1}(n))}
\end{equation}
and
\begin{equation}
\overline{\sigma;\tau}_{A_1,\dots,A_n} \colon A_1 \otimes \dots \otimes A_n \rightarrow A_{\sigma^{-1}(\tau^{-1}(1))} \otimes \dots \otimes A_{\sigma^{-1}(\tau^{-1}(n))}
\end{equation}
are equal.
\end{lemma}
\begin{lemma} \label{lemma-MAC1}
Let $(\mathsf{C},\otimes,I)$ be an additive symmetric strict monoidal category. We have the following equality of natural transformations:
\begin{equation*}
(\overline{\gamma_{n,p}})_{A_1,\dots,A_n} =\gamma_{A_1\otimes \dots \otimes A_n,A_{n+1} \otimes \dots \otimes A_{n+p}}.
\end{equation*}
\end{lemma}
\begin{lemma} \label{lemma-MAC2}
Let $(\mathsf{C},\otimes,I)$ be an additive symmetric strict monoidal category. We have the following equality of natural transformations:
$(\overline{\sigma \otimes \tau})_{A_1,\dots,A_{n+p}}=\overline{\sigma}_{A_1,\dots,A_n} \otimes \overline{\tau}_{A_{n+1},\dots,A_{n+p}}$.
\end{lemma}
\begin{proposition} \label{prop:strict-fun}
Let $(\mathsf{C},\otimes,I)$ be an additive symmetric strict monoidal category and let $A \in \mathsf{C}$. There exists a unique additive symmetric 
strict monoidal functor $F$ from $(\mathsf{S}_+,\otimes,0)$ to $(\mathsf{C},\otimes,I)$ such that $F(1)=A$. We have $F(n)=A^{\otimes n}$ for every 
$n \ge 0$ and $F(\sigma_1+\dots+\sigma_r)=(\overline{\sigma_1})_{A,\dots,A}+\dots+(\overline{\sigma_r})_{A,\dots,A}$ for every $\sigma_1+\dots+\sigma_r \in S(n,n)$.
\end{proposition}
\begin{proof}
Let $F$ be an additive symmetric strict monoidal functor from $(\mathsf{S}_+,\otimes,0)$ to an additive symmetric strict monoidal category $(\mathsf{C},\otimes,I)$ 
such that $F(1)=A$ for some $A \in \mathsf{C}$. Then we have $F(n)=F(1+\dots+1)=A^{\otimes n}$. Moreover, for every $n \ge 0$, if $\sigma_1+\dots+\sigma_r \in S_+(n,n)$, 
we have
\begin{align*}
F(\sigma_1+\dots+\sigma_r)=&~F(\sigma_1)+\dots+F(\sigma_r) &\text{\scriptsize{(\text{add. of }F)}} \\
=&~F((\overline{\sigma_1})_{1,\dots,1})+\dots+F((\overline{\sigma_r})_{1,\dots,1}) &\text{\scriptsize{(\cref{lemma-bar})}} \\
=&~(\overline{\sigma_1})_{A,\dots,A}+\dots+(\overline{\sigma_r})_{A,\dots,A}. &\text{\scriptsize{\cref{preserve-sigma}}} \\
\end{align*}
We have proven that if there exists an additive symmetric strict monoidal functor $F$ from $(\mathsf{S}_+,\otimes,0)$ to $(\mathsf{C},\otimes,I)$ such that $F(1)=A$, 
then it is the unique such functor. We will now prove that there exists such a functor.

We define a functor $F$ from $(\mathsf{S}_+,\otimes,0)$ to $(\mathsf{C},\otimes,I)$ by the formulas $F(n)=A^{\otimes n}$ and 
$F(\sigma_1+\dots+\sigma_r)=(\overline{\sigma_1})_{A,\dots,A}+\dots+(\overline{\sigma_r})_{A,\dots,A}$ for every $\sigma_1+\dots+\sigma_r \in S(n,n)$. 
We must first prove that $F$ is indeed a functor. We have $F(1_{S_n})=(\overline{1_{S_n}})_{A,\dots,A}$, but by definition $(\overline{1_{S_n}})_{A,\dots,A}=1_{A^{\otimes n}}$. 
It shows that identities are preserved. Let $\sigma,\tau \in S(n,n)$. To prove that $F(\sigma;\tau)=F(\sigma);F(\tau)$, we must prove that 
$(\overline{\sigma;\tau})_{A,\dots,A}=(\overline{\sigma})_{A,\dots,A};(\overline{\tau})_{A,\dots,A}$. But this is a consequence of \cref{comp-bar} by 
choosing $A_1=\dots=A_n:=A$. We conclude that $F$ is a functor.

The additivity of $F$ follows immediately from the definition of $F$ on morphisms. It only remains to be proven that $F$ is symmetric monoidal. 
The definition of $F$ of objects immediately gives that $F(0)=I$. The identity $F(\gamma_{n,p})=\gamma_{A^{\otimes n},A^{\otimes p}}$, that is, 
$(\overline{\gamma_{n,p}})_{A,\dots,A}=\gamma_{A^{\otimes n},A^{\otimes p}}$ follows from \cref{lemma-MAC1}. The identity 
$F(\sigma \otimes \tau)=F(\sigma) \otimes F(\tau)$, that is, $(\overline{\sigma \otimes \tau})_{A,\dots,A}=\overline{\sigma}_{A,\dots,A} \otimes \overline{\tau}_{A,\dots,A}$ 
follows from \cref{lemma-MAC2}.
\end{proof}
The identity in $(\mathsf{S}_+,\otimes,0)$ in the following lemma will be transferred to an arbitrary additive symmetric monoidal category in \cref{prop-unsh}.
\begin{lemma}
For all $n \ge 1$ and $0 \le k \le n-1$, the following identity holds in $(\mathsf{S}_+,\otimes,0)$:
\begin{equation} \label{unsh-id}
\mathsf{unsh}(n+1,k+1)_1=\mathsf{unsh}(n,k+1)_1 \otimes 1_1+\mathsf{unsh}(n,k)_1 \otimes 1_1;1_{k} \otimes \gamma_{n-k,1} \colon n+1 \rightarrow n+1.
\end{equation}
\end{lemma}
\begin{proof}
Using \cref{formula-unsh}, we must prove that
\begin{equation*}
\underset{\sigma \in \mathsf{Unsh}(n+1,k+1)}{\sum}\overline{\sigma}_{1,\dots,1}=\Big(\underset{a \in \mathsf{Unsh}(n,k+1)}{\sum}\overline{a}_{1,\dots,1}\Big) 
\otimes 1_1+\Big(\underset{b \in \mathsf{Unsh}(n,k)}{\sum}\overline{b}_{1,\dots,1}\Big) \otimes 1_1;1_{k} \otimes \gamma_{n-k,1},
\end{equation*}
that is, using \cref{lemma-bar}, we must prove that
\begin{equation*}
\underset{\sigma \in \mathsf{Unsh}(n+1,k+1)}{\sum}\sigma=\Big(\underset{a \in \mathsf{Unsh}(n,k+1)}{\sum}a\Big)\otimes 
1_1+\Big(\underset{b \in \mathsf{Unsh}(n,k)}{\sum}b\Big) \otimes 1_1;1_{k} \otimes \gamma_{n-k,1}
\end{equation*}
which is the same as
\begin{equation} \label{eq-sum-to-prove}
\underset{\sigma \in \mathsf{Unsh}(n+1,k+1)}{\sum}\sigma=\underset{a \in \mathsf{Unsh}(n,k+1)}{\sum}a \otimes 1_1+\underset{b \in \mathsf{Unsh}(n,k)}{\sum}b 
\otimes 1_1;1_{k} \otimes \gamma_{n-k,1}.
\end{equation}
We will prove that
\begin{equation} \label{unsh-to-prove}
\mathsf{Unsh}(n+1,k+1)=\{a \otimes 1_1,~a \in \mathsf{Unsh}(n,k+1)\}~\sqcup~\{b \otimes 1_1;1_k \otimes \gamma_{n-k,1},~b \in \mathsf{Unsh}(n,k)\}
\end{equation}
where $\sqcup$ denotes the internal disjoint union. That is, we will prove that the two sets on the RHS are disjoint and that their union is equal to the set on the LHS. 
This equality will imply \cref{eq-sum-to-prove}.

Firstly, note that $a \otimes 1_1 \in \mathsf{Unsh}(n+1,k+1)$ and $(a \otimes 1_1)^{-1}=a^{-1} \otimes 1_1$ for every $a \in \mathsf{Unsh}(n,k+1)$. Indeed, 
$(a^{-1} \otimes 1_1)(s)=a^{-1}(s)<(a^{-1} \otimes 1_1)(t)=a^{-1}(t)$ for all $s,t \in \{1,\dots,k+1\}$ such that $s<t$ and $(a^{-1} \otimes 1_1)(s)=a^{-1}(s)<(a^{-1} 
\otimes 1_1)(t)=a^{-1}(t)$ for all $s,t \in \{k+2,\dots,n\}$ such that $s<t$. If $k=n-1$, there is nothing more to prove. If $k<n-1$, we have 
$(a^{-1} \otimes 1_1)(n)=a^{-1}(n)<(a^{-1} \otimes 1_1)(n+1)=n+1$ since $a^{-1}(n) \in \{1,\dots,n\}$.

Secondly, note that $b \otimes 1_1;1_k \otimes \gamma_{n-k,1} \in \mathsf{Unsh}(n+1,k+1)$ for every $b \in \mathsf{Unsh}(n,k)$. Indeed, 
$(b \otimes 1_1;1_k \otimes \gamma_{n-k,1})^{-1}=1_k \otimes \gamma_{1,n-k};b^{-1} \otimes 1_1$ and $(1_k \otimes \gamma_{1,n-k};b^{-1} \otimes 1_1)(s)=b^{-1}(s)<(1_k 
\otimes \gamma_{1,n-k};b^{-1} \otimes 1_1)(t)=b^{-1}(t)$ for all $s,t \in \{1,\dots,k\}$ such that $s<t$. 
Moreover, $(1_k \otimes \gamma_{1,n-k};b^{-1} \otimes 1_1)(k)=b^{-1}(k)<(1_k \otimes \gamma_{1,n-k};b^{-1} \otimes 1_1)(k+1)=n+1$ since $b^{-1}(k) \in \{1,\dots,n\}$. 
We also have $(1_k \otimes \gamma_{1,n-k};b^{-1} \otimes 1_1)(s)=b^{-1}(s-1)<(1_k \otimes \gamma_{1,n-k};b^{-1} \otimes 1_1)(t)=b^{-1}(t-1)$ for all $s,t \in \{k+2,\dots,n+1\}$ 
such that $s<t$.
 
Thirdly, note that the two sets on the RHS of \cref{unsh-to-prove} are disjoint. Indeed, if $\sigma$ is in the first set on the RHS, then $\sigma^{-1}(n+1)=n+1$ 
whereas if $\sigma$ is in the second set on the RHS, then $\sigma^{-1}(n+1) \in \{1,\dots,n\}$.

Fourthly, we have $|\mathsf{Unsh}(n,k+1)|=\binom{n}{k+1}$ and $a \otimes 1_1=a' \otimes 1_1$ iff $a=a'$ for all $a,a' \in \mathsf{Unsh}(n,k+1)$. Thus the 
cardinal of the first set on the RHS is $\binom{n}{k+1}$.

Fifthly, we have $|\mathsf{Unsh}(n,k)|=\binom{n}{k}$ and $b \otimes 1_1;1_k \otimes \gamma_{n-k,1}=b' \otimes 1_1;1_k \otimes \gamma_{n-k,1}$ iff $b \otimes 1=b' 
\otimes 1_1$ iff $b=b'$ for all $b,b' \in \mathsf{Unsh}(n,k)$. Thus the cardinal of the second set on the RHS is $\binom{n}{k}$.

From these five points, it follows that we have a disjoint union on the RHS of \cref{unsh-to-prove}, which is included in $\mathsf{Unsh}(n+1,k+1)$ and whose cardinal 
is $\binom{n}{k+1}+\binom{n}{k}=\binom{n+1}{k+1}$. Since $|\mathsf{Unsh}(n+1,k+1)|=\binom{n+1}{k+1}$, it follows that \cref{unsh-to-prove} holds.
\end{proof}
We will use the following in the proof of \cref{id3}.
\begin{proposition} \label{prop-unsh}
Let $(\mathsf{C},\otimes,I)$ be an additive symmetric monoidal category. For all $n \ge 1$ and $0 \le k \le n-1$, the following equality of natural transformations holds:
\begin{equation} \label{ind-lemma} 
\mathsf{unsh}(n+1,k+1)_A=\mathsf{unsh}(n,k+1)_A \otimes 1_A+\mathsf{unsh}(n,k)_A \otimes 1_A;1_{A^{\otimes k}} \otimes \gamma_{A^{\otimes(n-k)},A}\colon A^{\otimes (n+1)} 
\rightarrow A^{\otimes(n+1)}.
\end{equation}
\end{proposition}
\begin{proof}
For simplicity, we suppose that $(\mathsf{C},\otimes,I)$ is a strict monoidal category. Let $A \in \mathsf{C}$ and let $F\colon\mathsf{S}_+ \rightarrow \mathsf{C}$ be 
the unique additive symmetric strict monoidal functor such that $F(1)=A$ which is given by \cref{prop:strict-fun}. The desired identity will be obtained by applying $F$ 
to \cref{unsh-id}. Since $F$ is an additive symmetric strict monoidal functor, we have 
\begin{equation*}
F(\mathsf{unsh}(n+1,k+1)_1)=F(\mathsf{unsh}(n,k+1)_1) \otimes 1_{F(1)}+F(\mathsf{unsh}(n,k)_1) \otimes 1_{F(1)};1_{F(k)} \otimes \gamma_{F(n-k),F(1)},
\end{equation*}
that is,
\begin{equation*}
F(\mathsf{unsh}(n+1,k+1)_1)=F(\mathsf{unsh}(n,k+1)_1) \otimes 1_A+F(\mathsf{unsh}(n,k)_1) \otimes 1_A;1_{A^{\otimes k}} \otimes \gamma_{A^{\otimes(n-k),A}}.
\end{equation*}
But for all $a,b \in \mathbb{N}$, we have
\begin{align*}
F(\mathsf{unsh}(a,b)_1)&=F\big(\underset{\sigma \in \mathsf{Unsh}(n,k)}{\sum}\overline{\sigma}_{1,\dots,1}\big) & \text{\scriptsize{\cref{formula-unsh}}} \\
&=\underset{\sigma \in \mathsf{Unsh}(n,k)}{\sum}F(\overline{\sigma}_{1,\dots,1}) &\text{\scriptsize{add.~of $F$}} \\
&=\underset{\sigma \in \mathsf{Unsh}(n,k)}{\sum}F(\sigma) &\text{\scriptsize{\cref{lemma-bar}}} \\
&=\underset{\sigma \in \mathsf{Unsh}(n,k)}{\sum}\overline{\sigma}_{A,\dots,A} &\text{\scriptsize{\cref{prop:strict-fun}}} \\
&=\mathsf{unsh}(n,k)_A. &\text{\scriptsize{\cref{formula-unsh}}}
\end{align*}
We conclude that \cref{ind-lemma} holds.
\end{proof}
\section{Higher-order rules for a differential modality} \label{sec:4}
\begin{proposition} \label{prop:higher-order-identities} 
Let $(\oc,m,u,\Delta,\epsilon,\partial)$ be a differential modality on an additive symmetric monoidal category $(\mathsf{C},\otimes,I)$. The following identities hold:
\begin{enumerate}
\item constant rule: \label{id1}
\begin{equation*}
\partial;\epsilon=0,
\end{equation*}
\item linear rule 2: \label{id2}
\begin{equation*}
\partial^2;u=0,
\end{equation*}
\item higher-order product rule ($n \ge 0$): \label{id3}
\begin{equation*}
\partial^n;\Delta=\underset{0 \le k \le n}{\sum}\Delta \otimes \mathsf{unsh}(n,k)_A;1_{\oc A} \otimes \gamma_{\oc A, A^{\otimes k}} \otimes 1_{A^{\otimes (n-k)}};
\partial^k \otimes \partial^{n-k},
\end{equation*}
\item higher-order product rule 2 ($n \ge 1$): \label{id4}
\begin{equation*}
\partial;\Delta^n=\underset{1 \le i \le n}{\sum}\Delta^n \otimes 1_A;1_{(\oc A)^{\otimes i}} \otimes \gamma_{(\oc A)^{\otimes(n-i)},A};1_{(\oc A)^{\otimes(i-1)}} 
\otimes \partial \otimes 1_{(\oc A)^{\otimes(n-i)}},
\end{equation*}
\item Faà di Bruno rule ($n \ge 0$): \label{id5}
\begin{equation*}
\partial^n;m=\underset{\pi \in \mathcal{H}(n)}{\sum}\Delta^{1+|\pi|} \otimes 1_{A^{\otimes n}};1_{\oc A} \otimes 
\overline{\tau_\pi}_{|\pi|\cdot \oc A,n \cdot A};m \otimes \partial^{|s_1|} \otimes \dots 
\otimes \partial^{|s_{|\pi|}|};\partial^{|\pi|}_{\oc A}
\end{equation*}
where for every $\pi \in \mathcal{H}(n)$, we write $\pi=\{s_1,\dots,s_{|\pi|}\}$ with $\mathrm{max}\,s_1<\dots<\mathrm{max}\,s_{|\pi|}$,
\item higher-order interchange rule ($n \ge 0$):
\begin{equation*}
1_{\oc A} \otimes \overline{\tau}_{A,\dots,A};\partial^n=\partial^n \label{id6}
\end{equation*}
for every $\tau \in S_n$.
\end{enumerate}
\end{proposition}
\begin{remark}
In the RHS of \cref{id5}, the formula $\overline{\tau}_{|\pi|\cdot \oc A,n \cdot A}$ is a notation for $\overline{\tau}_{\oc A,\dots,\oc A,A,\dots,A}$ where $\oc A$ is repeated $|\pi|$ times and $A$ is repeated $n$ times.
\end{remark}
\begin{proof}~\\
\begin{enumerate}
\item[\cref{id1}:] See \cite{DIFCATREV}.
\item[\cref{id2}:] We compute:
\begin{align*}
\partial^2;u&=\partial \otimes 1_A;\partial;u& \text{\scriptsize{\cref{def-d-ind-2}}} \\
&=\partial \otimes 1_A;\epsilon \otimes 1_A& \text{\scriptsize{\cref{def-d1}}} \\
&=(\partial;\epsilon) \otimes 1_A& \text{\scriptsize{(functoriality of $-\otimes 1_A$)}} \\
&=0 \otimes 1_A& \text{\scriptsize{\cref{id1}}} \\
&=0.& \text{\scriptsize{\cref{add-5}}} 
\end{align*}
\item[\cref{id3}:] See \cref{proof:higer-order-product}.
\item[\cref{id4}:] See \cref{proof:higer-order-product-2}.
\item[\cref{id5}:] See \cref{proof:faà-di-bruno}.
\item[\cref{id6}:] If $n=0$, then $S_0=\{1_{S_0}\}$, $\overline{1_{S_0}}=1_I$ and we have $1_{\oc A} \otimes \overline{1_{S_0}};\partial^0=1_{\oc A} \otimes 
1_I;1_{\oc A}=1_{\oc A}=\partial^0$. If $n=1$, then $S_0=\{1_{S_1}\}$, $\overline{1_{S_1}}_A=1_A$ and we have $1_{\oc A} \otimes \overline{1_{S_1}}_A;\partial=
1_{\oc A} \otimes 1_A;\partial=\partial$.

Suppose now that $n \ge 2$. According to its definition, for every $\tau \in S_n$, the morphism $\overline{\tau}_{A,\dots,A}$ is a finite composite of morphisms of the form 
\begin{equation}
1_{A^{\otimes k}} \otimes \gamma_{A,A} \otimes 1_{A^{\otimes l}}\colon A^{\otimes (k+l+2)} \rightarrow A^{\otimes (k+l+2)} 
\end{equation}
where $(k,l) \in \mathbb{N}^2$ is such that $k+l+2=n$. Therefore, to prove \cref{id6} for $n \ge 2$, it suffices to prove that
\begin{equation*}
1_{\oc A} \otimes 1_{A^{\otimes k}} \otimes \gamma_{A,A} \otimes 1_{A^{\otimes l}};\partial^{k+l+2}=\partial^{k+l+2}
\end{equation*}
for every $(k,l) \in \mathbb{N}^2$. Let $(k,l) \in \mathbb{N}^2$. We compute:
\begin{align*}
&1_{\oc A} \otimes 1_{A^{\otimes k}} \otimes \gamma_{A,A} \otimes 1_{A^{\otimes l}};\partial^{k+2+l} \\
&=1_{\oc A} \otimes 1_{A^{\otimes k}} \otimes \gamma_{A,A} \otimes 1_{A^{\otimes l}};\partial^{k+2} \otimes 1_{A^{\otimes l}};\partial^l& \text{\scriptsize{\cref{d-k-l}}} \\
&=1_{\oc A \otimes A^{\otimes k}} \otimes \gamma_{A,A} \otimes 1_{A^{\otimes l}};\partial^{k+2} \otimes 1_{A^{\otimes l}};\partial^l& 
\text{\scriptsize{(functoriality of $-\otimes-$)}} \\
&=1_{\oc A \otimes A^{\otimes k}} \otimes \gamma_{A,A} \otimes 1_{A^{\otimes l}};(\partial^k \otimes 1_{A^{\otimes 2}};\partial^2) \otimes 1_{A^{\otimes l}};
\partial^l& \text{\scriptsize{\cref{d-k-l}}} \\
&=1_{\oc A \otimes A^{\otimes k}} \otimes \gamma_{A,A} \otimes 1_{A^{\otimes l}};\partial^k \otimes 1_{A^{\otimes (l+2)}};\partial^2 \otimes 1_{A^{\otimes l}};
\partial^l& \text{\scriptsize{(functoriality of $-\otimes 1_{A^{\otimes l}}$)}}  \\
&=\partial^k \otimes \gamma_{A,A} \otimes 1_{A^{\otimes l}};\partial^2 \otimes 1_{A^{\otimes l}};\partial^l& \text{\scriptsize{(functoriality of $-\otimes-$)}} \\
&=(\partial^k \otimes \gamma_{A,A};\partial^2) \otimes 1_{A^{\otimes l}};\partial^l& \text{\scriptsize{(functoriality of $-\otimes 1_{A^{\otimes l}}$)}} \\
&=(\partial^k \otimes 1_{A^{\otimes 2}};1_{\oc A} \otimes \gamma_{A,A};\partial^2) \otimes 1_{A^{\otimes l}};\partial^l&\text{\scriptsize{(functoriality of $-\otimes-$)}} \\
&=(\partial^k \otimes 1_{A^{\otimes 2}};\partial^2) \otimes 1_{A^{\otimes l}};\partial^l&\text{\scriptsize{\cref{def-d4}}} \\
&=\partial^{k+2} \otimes 1_{A^{\otimes l}};\partial^l &\text{\scriptsize{\cref{d-k-l}}} \\
&=\partial^{k+2+l}. &\text{\scriptsize{\cref{d-k-l}}}
\end{align*}
\end{enumerate}
\end{proof}
\section{Proving the theorem} \label{sec:5}
In this section, we will prove \cref{theorem:extracting}. We therefore assume the hypotheses of this theorem: Let $(\oc,\epsilon,u,\Delta,m,\partial)$ be a 
differential modality on an additive symmetric monoidal category $(\mathsf{C},\otimes,I)$. Suppose that for all object $A$ and $n \in \mathbb{N}$,
\begin{itemize}
\item[1.] $s_n\colon \oc A \rightarrow \oc_{\le n}A$ is a cokernel of $\partial^{n+1}$,
\item[2.] $s_n \otimes 1_A$ is a cokernel of $\partial^{n+1} \otimes 1_A$.
\end{itemize}
\subsection{Defining the data for an $\mathbb{N}$-graded differential modality}
We will first define the functors $\oc_{\le n}$ and the natural transformations $\epsilon^\le,u^\le,\Delta_{n,p}^\le,m_{n,p}6\le, \partial_n^\le$.

\textbf{Definition of $\oc_{ \le n}f$:} Given any $f\colon A \rightarrow B$, we define $\oc_{\le n}f:\oc_{\le n}A \rightarrow \oc_{\le n}B$ by using the following diagram:
\begin{equation*}
\begin{tikzcd}
\oc A \otimes A^{\otimes n} \arrow[rr, "\partial^n", shift left=2] \arrow[rr, "0"', shift right=2] \arrow[d, "\oc f \otimes f^{\otimes n}"'] &  & \oc A 
\arrow[d, "\oc f"] \arrow[rr, "s_n"] &  & \oc_{\le n} A \arrow[d, "\oc_{\le n}f", dashed] \\
\oc B \otimes B^{\otimes n} \arrow[rr, "\partial^n", shift left=2] \arrow[rr, "0"', shift right=2]                                           &  & \oc B 
\arrow[rr, "s_n"']                   &  & \oc_{\le n} B~.                                
\end{tikzcd}\end{equation*}
Some comments are needed: We have $\partial^n;\oc f;s_n=\oc f \otimes f^{\otimes n};\partial^n;s_n$ by \cref{nat-dn}. But $\partial^n;s_n=0;s_n=0$ by 
definition of a cokernel. It follows that $\partial^n;\oc f;s_n=\oc f \otimes f^{\otimes n};0=0=0;\oc f;s_n$. By the universal property of cokernels, 
we deduce that there exists a unique morphism $\oc_{\le n}f$ such that the square on the right side of the diagram commutes. 

\textbf{Naturality of $s_n$:} Note that by the definition of $\oc_{\le n}f$, $s_n$ is automatically a natural transformation. We can thus use the tensor 
product of $(\mathsf{End}(\mathsf{C}),\bullet,1)$ to obtain the natural transformation $s_n \bullet s_p\colon \oc\oc A \rightarrow \oc_{\le n}\oc_{\le p}A$.

\textbf{Definition of $\epsilon^\le$:} We use the universal property of $s_0$:
\begin{equation*}
\begin{tikzcd}
\oc A \otimes A \arrow[rr, "\partial", shift left=2] \arrow[rr, "0"', shift right=2] &  & \oc A \arrow[rr, "s_0"] \arrow[rrd, "\epsilon"'] &  & \oc_{\le 0} A \arrow[d, "\epsilon^\le", dashed] \\
                                                                                     &  &                                                  &  & I~.                                              
\end{tikzcd}
\end{equation*}
Here, we have used \cref{id1}.

\textbf{Definition of $u^{\le}$:} We use the universal property of $s_1$: 
\begin{equation*}
\begin{tikzcd}
\oc A \otimes A^{\otimes 2} \arrow[rr, "\partial^2", shift left=2] \arrow[rr, "0"', shift right=2] &  & \oc A \arrow[rr, "s_1"] \arrow[rrd, "u"'] &  & \oc_{\le 1} A \arrow[d, "u^\le", dashed] \\
                                                                                                   &  &                                           &  & A~.                                     
\end{tikzcd}
\end{equation*}
Here, we have used \cref{id2}.

\textbf{Definition of $\Delta^\le_{n,p}$:} We use the universal property of $s_{n+p}$:
\begin{equation*}
\begin{tikzcd}
\oc A \otimes A^{\otimes (n+p+1)} \arrow[rr, "\partial^{n+p+1}", shift left=2] \arrow[rr, "0"', shift right=2] &  & \oc A \arrow[r, "s_{n+p}"] \arrow[rdd, "\Delta;s_n \otimes s_p"'] & \oc_{\le n+p}A \arrow[dd, "{\Delta_{n,p}^{\le}}", dashed] \\
                                                                                                               &  &                                                                   &                                                           \\
                                                                                                               &  &                                                                   & \oc_{\le n}A \otimes \oc_{\le p}A~.                      
\end{tikzcd}
\end{equation*}
We must prove that $\partial^{n+p+1};\Delta;s_n \otimes s_p=0$. We start by using \cref{id3} which gives
\begin{align*}
&\partial^{n+p+1};\Delta;s_n \otimes s_p \\
&=\Big(\underset{0 \le k \le n+p-1}{\sum}\Delta \otimes \mathsf{unsh}(n+p+1,k)_A;1_{\oc A} \otimes \gamma_{\oc A, A^{\otimes k}} \otimes 
1_{A^{\otimes ((n+p-1)-k)}};\partial^k \otimes \partial^{(n+p+1)-k}\Big);s_n \otimes s_p \\
&=\underset{0 \le k \le n+p-1}{\sum}\Delta \otimes \mathsf{unsh}(n+p+1,k)_A;1_{\oc A} \otimes \gamma_{\oc A, A^{\otimes k}} \otimes 
1_{A^{\otimes (n-k)}};\partial^k \otimes \partial^{(n+p+1)-k};s_n \otimes s_p.
\end{align*}
The desired identity will follow from the fact that $\partial^k \otimes \partial^{(n+p+1)-k};s_n \otimes s_p=0$ for every $0 \le k \le n+p+1$ 
which we will now prove. We first note that $\partial^k \otimes \partial^{(n+p+1)-k};s_n \otimes s_p=(\partial^k;s_n) \otimes (\partial^{(n+p+1)-k};s_p)$. 
We then distinguish two cases and make use of \cref{d-k-l}. 
\begin{itemize}
\item[1.] If $k \le n$: we have $n-k \ge 0$, thus
\begin{equation*}
\partial^{(n+p+1)-k};s_p=\partial^{(n-k)+(p+1)};s_p=\partial^{n-k} \otimes 1_{A^{\otimes(p+1)}};\partial^{p+1};s_p=\partial^{n-k}\otimes 1_{A^{\otimes(p+1)}};0=0.
\end{equation*}
\item[2.] If $k \ge n+1$: we have $k-(n+1) \ge 0$, thus
\begin{equation*}
\partial^k;s_n=\partial^{(k-(n+1))+(n+1)};s_n=\partial^{k-(n+1)} \otimes 1_{A^{\otimes(n+1)}};\partial^{n+1};s_n=\partial^{k-(n+1)} \otimes 1_{A^{\otimes(n+1)}};0=0.
\end{equation*}
\end{itemize}
In any case, $(\partial^k;s_n) \otimes (\partial^{(n+p-1)-k};s_p)=0$.

\textbf{Definition of $m^{\le}_{n,p}$:} We use the universal property of $s_{np}$:
\begin{equation*}
\begin{tikzcd}
\oc A \otimes A^{\otimes (np+1)} \arrow[rr, "\partial^{np+1}", shift left=2] \arrow[rr, "0"', shift right=2] &  & \oc A \arrow[r, "s_{np}"] \arrow[rdd, "m;s_n \bullet s_p"'] & \oc_{\le np}A \arrow[dd, "{m_{n,p}^{\le}}", dashed] \\
                                                                                                             &  &                                                             &                                                     \\
                                                                                                             &  &                                                             & \oc_{\le n}\oc_{\le p}A~.                          
\end{tikzcd}
\end{equation*}
We must prove that $\partial^{np+1};m;s_n \bullet s_p=0$. We start by using \cref{id5} which gives
\begin{align*}
\partial^{np+1};m;s_n \bullet s_p&=\Big(\underset{\pi \in \mathcal{H}(np+1)}{\sum}\Delta^{1+|\pi|} \otimes 1_{A^{\otimes (np+1)}};1_{\oc A} \otimes \overline{\tau_\pi}_{|\pi| \cdot \oc A,(np+1) \cdot A};
m \otimes \partial^{|s_1|} \otimes \dots \otimes \partial^{|s_{|\pi|}|};\partial^{|\pi|}_{\oc A}\Big);s_n \bullet s_p \\
&=\underset{\pi \in \mathcal{H}(np+1)}{\sum}\Delta^{1+|\pi|} \otimes 1_{A^{\otimes (np+1)}};1_{\oc A} \otimes \overline{\tau_\pi}_{|\pi| \cdot \oc A,(np+1) \cdot A};m \otimes \partial^{|s_1|} \otimes 
\dots \otimes \partial^{|s_{|\pi|}|};\partial^{|\pi|}_{\oc A};s_n \bullet s_p.
\end{align*}
The desired identity will follow from the fact that $m \otimes \partial^{|s_1|} \otimes \dots \otimes \partial^{|s_{|\pi|}|};\partial^{|\pi|}_{\oc A};s_n \bullet s_p=0$ 
for every $\pi \in \mathcal{H}(np+1)$ that we will now prove. We distinguish two cases.
\begin{itemize}
\item[1.] If $|\pi| \ge n+1$: We have:
\begin{align*}
\partial^{|\pi|}_{\oc A};s_n \bullet s_p&=\partial^{(|\pi|-(n+1))+(n+1)}_{\oc A};s_n(\oc A);\oc_{\le n}(s_p(A)) &\text{\scriptsize{\cref{bullet-2}}} \\
&=\partial^{|\pi|-(n+1)}_{\oc A} \otimes 1_{(\oc A)^{\otimes (n+1)}};\partial^{n+1}_{\oc A};s_n(\oc A);\oc_{\le n}(s_p(A)) &\text{\scriptsize{\cref{d-k-l}}} \\
&=\partial^{|\pi|-(n+1)}_{\oc A} \otimes 1_{(\oc A)^{\otimes (n+1)}};0;\oc_{\le n}(s_p(A)) &\text{\scriptsize{universal prop.~of $s_n(\oc A)$}} \\
&=0. &\text{\scriptsize{\cref{add-1,add-2}}} 
\end{align*}
\item[2.] If $|\pi| \le n$: Since $\pi$ is a partition of $\{1,\dots,np+1\}$, we have $|s_1|+\dots+|s_{|\pi|}|=np+1$. It follows that for at least one 
$1 \le i \le |\pi|$, we have $|s_i| \ge p+1$. Else, we would have $|s_1|+\dots+|s_{|\pi|}| \le |\pi|p \le np$, a contradiction. Let $1 \le i \le |\pi|$ 
be such an integer. We have:
\begin{align*}
&m \otimes \partial^{|s_1|} \otimes \dots \otimes \partial^{|s_{|\pi|}|};\partial^{|\pi|}_{\oc A};(s_n \bullet s_p)_A \\
&=m \otimes \partial^{|s_1|} \otimes \dots \otimes \partial^{|s_{|\pi|}|};\partial^{|\pi|}_{\oc A};\oc (s_p(A));s_n(\oc_{\le p}A) &\text{\scriptsize{\cref{bullet-1}}}  \\
&=m \otimes \partial^{|s_1|} \otimes \dots \otimes \partial^{|s_{|\pi|}|};\oc (s_p(A)) \otimes (s_p(A))^{\otimes |\pi|};\partial_{\oc_{\le p}A}^{|\pi|};s_n(\oc_{\le p}A) 
&\text{\scriptsize{\cref{nat-dn}}} \\
&=(m;\oc(s_p(A))) \otimes (\partial^{|s_1|};s_p(A)) \otimes \dots \otimes (\partial^{|s_{|\pi|}|};s_p(A));\partial_{\oc_{\le p}A}^{|\pi|};s_n(\oc_{\le p}A). 
&\text{\scriptsize{func.~of~$-\otimes-$}}
\end{align*}
But \cref{d-k-l} and the universal property of $s_p(A)$ give that
\begin{equation*}
\partial^{|s_{i}|};s_p(A)=\partial^{|s_i|-(p+1)} \otimes 1_{A^{\otimes(p+1)}};\partial^{p+1};s_p(A)=\partial^{|s_i|-(p+1)} \otimes 1_{A^{\otimes(p+1)}};0=0.
\end{equation*}
It follows that 
\begin{equation*}
(m;\oc(s_p(A))) \otimes (\partial^{|s_1|};s_p(A)) \otimes \dots \otimes (\partial^{|s_{|\pi|}|};s_p(A));\partial_{\oc_{\le p}A}^{|\pi|};s_n(\oc_{\le p}A)=0.
\end{equation*}
\end{itemize}

\textbf{Definition of $\partial_n^{\le}$:} We use the universal properties of $s_{n+1}$ and $s_n \otimes 1$:
\begin{equation*}
\begin{tikzcd}
\oc A \otimes A^{\otimes (n+2)} \arrow[rr, "\partial^{n+1} \otimes 1", shift left=2] \arrow[rr, "0"', shift right=2] \arrow[rrdd, "\partial^{n+2}", shift left=2] \arrow[rrdd, "0"', shift right=2] &  & \oc A \otimes A \arrow[dd, "\partial"'] \arrow[rr, "s_n\otimes 1"] &  & \oc_{\le n}A \otimes A \arrow[dd, "\partial_n^{\le}", dashed] \\
                                                                                                                                                                                                    &  &                                                                    &  &                                                               \\
                                                                                                                                                                                                    &  & \oc A \arrow[rr, "s_{n+1}"']                                       &  & \oc_{\le n+1} A~.                                            
\end{tikzcd}
\end{equation*}
Some comments are needed: We have $\partial^{n+1} \otimes 1;\partial;s_{n+1}=\partial^{n+2};s_{n+1}=0$ by \cref{def-d-ind-2} and the universal property of $s_{n+1}$. 
It follows from the universal property of $s_n \otimes 1$ that there exists a unique morphism $\partial_n^{\le}$ such that the square on the right side of the diagram commutes.
\subsection{Checking the equations for an $\mathbb{N}$-graded differential modality}
We will now check that $\oc_{\le n}, \epsilon^\le,u^\le,\Delta_{n,p}^\le,m_{n,p}^\le, \partial_n^\le$ satisfy the required identities to constitute an $\mathbb{N}$-graded 
differential modality. There are nineteen such identities.
\begin{itemize}
\item[1.] $\oc_n(f;g)=\oc_n(f);\oc_n(g)$:

The three cells in the following diagram commute by functoriality of $\oc$ and naturality of $s_n$:
\begin{equation*}
\begin{tikzcd}
\oc A \arrow[rr, "s_n"] \arrow[d, "\oc f"] \arrow[dd, "\oc(f;g)"', bend right=49] &  & \oc_{\le n}A \arrow[d, "\oc_{\le n} f"] \\
\oc B \arrow[rr, "s_n"] \arrow[d, "\oc g"]                                        &  & \oc_{\le n}B \arrow[d, "\oc_{\le n} g"] \\
\oc C \arrow[rr, "s_n"]                                                           &  & \oc_{\le n}B~.                     
\end{tikzcd}
\end{equation*}
We thus have $\oc(f;g);s_n=s_n;\oc_{\le n}f;\oc_{\le n}g$. By definition of $\oc_{\le n}(f;g)$, we conclude that $\oc_{\le n}(f;g)=\oc_{\le n}f;\oc_{\le n}g$.

\item[2.] $\oc_{\le n}(1_A)=1_{\oc_{\le n}(A)}$:

The two cells in the following diagram commute---the first one by functoriality of $\oc$:
\begin{equation*}
\begin{tikzcd}
\oc A \arrow[rr, "s_n"] \arrow[d, "\oc(1_A)"', bend right=67] \arrow[d, "1_{\oc A}"] &  & \oc_{\le n}A \arrow[d, "1_{\oc_{\le n}A}"] \\
\oc A \arrow[rr, "s_n"]                                                              &  & \oc_{\le n}A~.                          
\end{tikzcd}
\end{equation*}
It follows by definition of $\oc_{\le n}(1_A)$ that $\oc_{\le n}(1_A)=1_{\oc_{\le n}A}$.
\item[3.] Naturality of $\epsilon^\le$:

In the following diagram, the parallelogram commutes by naturality of $s_0$, the triangle on the left commutes by naturality of $\epsilon$ and the triangle on the right 
commutes by definition of $\epsilon^\le$:
\begin{equation*}
\begin{tikzcd}
\oc A \arrow[rd, "\oc f"'] \arrow[rr, "s_0"] \arrow[rdd, "\epsilon"', bend right] &                                               & \oc_{\le 0}A \arrow[rd, "\oc_{\le 0}f"] &                                          \\
                                                                                  & \oc B \arrow[rr, "s_0"] \arrow[d, "\epsilon"] &                                         & \oc_{\le 0}B \arrow[lld, "\epsilon^\le"] \\
                                                                                  & I~.                                             &                                         &                                         
\end{tikzcd}
\end{equation*}
It follows that $s_0;\oc_{\le 0}f;\epsilon^\le=\epsilon$. By definition of $\epsilon^\le$, we obtain that $\epsilon^\le=\oc_{\le 0}f;\epsilon^\le$.
\item [4.] Naturality of $u^\le$:

In the following (3D) diagram, the upper face commutes by naturality of $s_1$, the two triangles commute by definition of $u^{\le}$ and the left face commutes by naturality of u:
\begin{equation*}
\begin{tikzcd}
\oc A \arrow[rd, "\oc f"] \arrow[rr, "s_1"] \arrow[d, "u"'] &                                        & \oc_{\le 1}A \arrow[rd, "\oc_{\le 1}f"] \arrow[lld, "u^\le", dashed] &                                     \\
A \arrow[rd, "f"']                                          & \oc B \arrow[rr, "s_1"] \arrow[d, "u"] &                                                                      & \oc_{\le 1}B \arrow[lld, "u^{\le}"] \\
                                                            & B~.                                      &                                                                      &                                    
\end{tikzcd}
\end{equation*}
By diagram chasing, we deduce that $s_1;\oc_{\le 1}f;u^\le=s_1;u^\le;f$. Since $s_1$ is an epimorphism as a coequalizer, it follows that $\oc_{\le 1}f;u^\le=u^\le;f$, 
that is, the bottom face commutes.

\item [5.] Naturality of $\Delta_{n,p}^\le$:

In the following (3D) diagram, the upper and bottom faces commute by definition of $\Delta^\le_{n,p}$, the left face commutes by naturality of $s_{n+p}$, the right face 
commutes by naturality of $s_n \otimes s_p$, and the back face commutes by naturality of $\Delta$:
\begin{equation*}
\begin{tikzcd}
\oc A \arrow[d, "\oc f"'] \arrow[rr, "\Delta"] \arrow[rdd, "s_{n+p}"] &                                                                             & \oc A \otimes \oc A \arrow[d, "\oc f \otimes \oc f"', dashed] \arrow[rdd, "s_n \otimes s_p"] &                                                                                    \\
\oc B \arrow[rr, "\Delta"', dashed] \arrow[rdd, "s_{n+p}"']           &                                                                             & \oc B \otimes \oc B \arrow[rdd, "s_n \otimes s_p"', dashed]                                  &                                                                                    \\
                                                                      & \oc_{\le n+p}A \arrow[d, "\oc_{\le n+p}f"] \arrow[rr, "{\Delta_{n,p}^\le}"] &                                                                                              & \oc_{\le n}A \otimes \oc_{\le p}A \arrow[d, "\oc_{\le n} f \otimes \oc_{\le p} f"] \\
                                                                      & \oc_{\le n+p}B \arrow[rr, "{\Delta_{n,p}^\le}"]                             &                                                                                              & \oc_{\le n}B \otimes \oc_{\le p}B~.                                                 
\end{tikzcd}
\end{equation*}
By diagram chasing, we deduce that $s_{n+p};\Delta^\le_{n,p};\oc_{\le n}f \otimes \oc_{\le p}f=s_{n+p};\oc_{\le n+p}f;\Delta_{n,p}^\le$. Since $s_{n+p}$ is an epimorphism, it follows that 
the front face commutes.
\item [6.] Naturality of $m_{n,p}^\le$:

In the following (3D) diagram, the upper and bottom faces commute by definition of $m^\le_{n,p}$, the left face commutes by naturality of $s_{np}$, the right face commutes 
by naturality of $s_n \bullet s_p$ and the back face commutes by naturality of $m$:
\begin{equation*}
\begin{tikzcd}
\oc A \arrow[d, "\oc f"'] \arrow[rr, "m"] \arrow[rdd, "s_{np}"] &                                                                      & \oc\oc A \arrow[d, "\oc\oc f"', dashed] \arrow[rdd, "s_n \bullet s_p"] &                                                              \\
\oc B \arrow[rr, "m"', dashed] \arrow[rdd, "s_{np}"']           &                                                                      & \oc\oc B \arrow[rdd, "s_n \bullet s_p"', dashed]                       &                                                              \\
                                                                & \oc_{\le np}A \arrow[d, "\oc_{\le np}f"] \arrow[rr, "{m_{n,p}^\le}"] &                                                                        & \oc_{\le n}\oc_{\le p}A \arrow[d, "\oc_{\le n}\oc_{\le p}f"] \\
                                                                & \oc_{\le np}B \arrow[rr, "{m_{n,p}^\le}"]                            &                                                                        & \oc_{\le n}\oc_{\le p}B~.                                     
\end{tikzcd}
\end{equation*}
By diagram chasing, we deduce that $s_{np};m^\le_{n,p};\oc_{\le n}\oc_{\le p}f=s_{np};\oc_{\le np}f;m_{n,p}^\le$. Since $s_{np}$ is an epimorphism, it follows that 
the front face commutes.
\item[7.] Naturality of $\partial_n^\le$:
\begin{equation*}
\begin{tikzcd}
\oc A \otimes A \arrow[d, "\oc f \otimes f"', dashed] \arrow[rdd, "s_n \otimes 1"] \arrow[rr, "\partial"] &                                                                                          & \oc A \arrow[d, "\oc f"', dashed] \arrow[rdd, "s_{n+1}"] &                                            \\
\oc B \otimes B \arrow[rdd, "s_n \otimes 1"', dashed] \arrow[rr, "\partial"', dashed]                     &                                                                                          & \oc B \arrow[rdd, "s_{n+1}"', dashed]                    &                                            \\
                                                                                                          & \oc_{\le n} A \otimes A \arrow[d, "\oc_{\le n}f \otimes f"] \arrow[rr, "\partial_n^\le"] &                                                          & \oc_{\le n+1}A \arrow[d, "\oc_{\le n+1}f"] \\
                                                                                                          & \oc_{\le n} B \otimes B \arrow[rr, "\partial_n^\le"']                                    &                                                          & \oc_{\le n+1}B                            
\end{tikzcd}
\end{equation*}
Same reasoning as for $\Delta_{n,p}^\le$ and $m_{n,p}^\le$: The upper and bottom faces commute by definition of $\partial^\le_n$, the right face commutes by 
naturality of $s_{n+1}$, the left face commutes by naturality of $s_n \otimes 1$, and the back face commutes by naturality of $\partial$. By diagram chasing 
we obtain $s_n \otimes 1;\partial^\le_n;\oc_{\le n+1}f=s_n \otimes 1;\oc_{\le n}f \otimes ;\partial^\le_n$. Thus the front face commutes since $s_n \otimes 1$ is an epimorphism as a coequalizer.
\item[8.] First diagram in \cref{def-graded-comonad}:
\begin{equation*}
\begin{tikzcd}
\oc A \arrow[d, "m"'] \arrow[rr, "m"] \arrow[rdd, "s_{npq}"]             &                                                                      & \oc\oc A \arrow[d, "m"', dashed] \arrow[rdd, "s_{np} \bullet s_q"] &                                          \\
\oc\oc A \arrow[rr, "\oc m"', dashed] \arrow[rdd, "s_n \bullet s_{pq}"'] &                                                                      & \oc\oc\oc A \arrow[rdd, "s_n\bullet s_p\bullet s_q"', dashed]      &                                          \\
                                                                         & \oc_{\le npq} A \arrow[d, "{m_{n,pq}^\le}"] \arrow[rr, "{m_{np,q}^\le}"] &                                                                    & \oc_{\le np}\oc_{\le q} \arrow[d, "{m_{n,p}^\le}"] \\
                                                                         & \oc_{\le n}\oc_{\le pq}A \arrow[rr, "{\oc_{\le n}m_{p,q}^\le}"']                     &                                                                    & \oc_{\le n}\oc_{\le p}\oc_{\le q}A                        
\end{tikzcd}
\end{equation*}
The back face commutes by coassociativity of $m$. The upper face commutes by definition of $m_{np,q}^\le$. The left face commutes by definition of $m_{n,pq}^\le$. 
The right face commutes by applying \cref{prop-end-1} to the definition of $m^{\le}_{n,p}$. The bottom face commutes by applying \cref{prop-end-2} to the definition 
of $m_{p,q}^\le$. By diagram chasing we obtain $s_{npq};m_{np,q}^\le;m_{n,p}^\le=s_{npq};m_{n,pq}^\le;\oc_nm_{p,q}^\le$ from which it follows that the front face commutes.

\item[9.] Upper triangle in the second diagram in \cref{def-graded-comonad}:
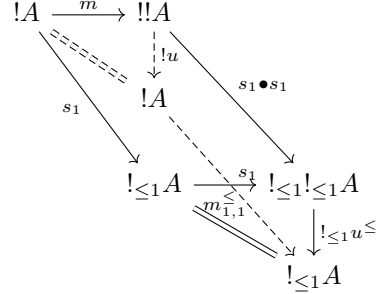
\begin{equation*}
\begin{tikzcd}
\oc A \arrow[r, "m"] \arrow[rd, dashed, equal] \arrow[rdd, "s_1"'] & \oc\oc A \arrow[d, "\oc u", dashed] \arrow[rdd, "s_1 \bullet s_1"] &                                               \\
                                                            & \oc A \arrow[rdd, "s_1", dashed]                                   &                                               \\
                                                            & \oc_{\le 1} A \arrow[rd, equal] \arrow[r, "{m_{1,1}^\le}"']                 & \oc_{\le 1}\oc_{\le 1} A \arrow[d, "\oc_{\le 1}u^\le"] \\
                                                            &                                                                    & \oc_{\le 1} A                                  
\end{tikzcd}
\end{equation*}
The back face commutes by counitiality of the comonad $\oc$. The upper face commutes by definition of $m_{1,1}^\le$. The bottom face commutes trivially. The right 
face commutes by tensoring with $s_1$ on the left in $\mathsf{End}(\mathsf{C})$ the diagram defining $u^\le$ . We deduce that the front face commutes by the same 
reasoning as for the previous diagrams.

\item[10.] Lower triangle in the second diagram in \cref{def-graded-comonad}:
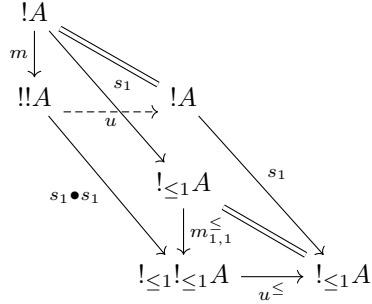
\begin{equation*}
\begin{tikzcd}
\oc A \arrow[rdd, "s_1"] \arrow[d, "m"'] \arrow[rd, equal]                             &                                                    &              \\
\oc\oc A \arrow[r, "u"', dashed, shift right=2] \arrow[rdd, "s_1 \bullet s_1"'] & \oc A \arrow[rdd, "s_1"]                           &              \\
                                                                                & \oc_{\le 1}A \arrow[rd, equal] \arrow[d, "{m_{1,1}^\le}"] &              \\
                                                                                & \oc_{\le 1}\oc_{\le 1} A \arrow[r, "u^\le"']       & \oc_{\le 1}A
\end{tikzcd}
\end{equation*}
The back face commutes by counitality of the comonad $\oc$. The upper face commutes trivially. The left face commutes by definition of $m_{1,1}^\le$. The bottom face 
commutes by tensoring with $s_1$ on the right in $\mathsf{End}(\mathsf{C})$ the diagram defining $u^\le$ . We deduce that the front face commutes as before.
\item[11.] First diagram in \cref{def:graded-comonoid}:
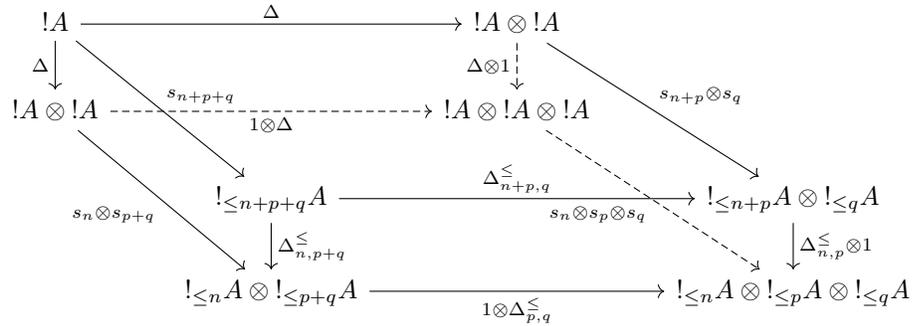
\begin{equation*}
\begin{tikzcd}
\oc A \arrow[d, "\Delta"'] \arrow[rr, "\Delta"] \arrow[rdd, "s_{n+p+q}"]                        &                                                                                       & \oc A \otimes \oc A \arrow[d, "\Delta \otimes 1"', dashed] \arrow[rdd, "s_{n+p} \otimes s_q"] &                                                                               \\
\oc A \otimes \oc A \arrow[rr, "1 \otimes \Delta"', dashed] \arrow[rdd, "s_n \otimes s_{p+q}"'] &                                                                                       & \oc A \otimes \oc A \otimes \oc A \arrow[rdd, "s_n \otimes s_p \otimes s_q"', dashed]         &                                                                               \\
                                                                                                & \oc_{\le n+p+q}A \arrow[d, "{\Delta_{n,p+q}^\le}"] \arrow[rr, "{\Delta_{n+p,q}^\le}"] &                                                                                               & \oc_{\le n+p}A \otimes \oc_{\le q}A \arrow[d, "{\Delta_{n,p}^\le \otimes 1}"] \\
                                                                                                & \oc_{\le n}A \otimes \oc_{\le p+q} A \arrow[rr, "{1 \otimes \Delta_{p,q}^\le}"']      &                                                                                               & \oc_{\le n}A \otimes \oc_{\le p}A \otimes \oc_{\le q}A                       
\end{tikzcd}
\end{equation*}
The back face commutes by coassociativity of $\Delta$. The upper face commutes by definition of $\Delta^\le_{n+p,q}$. The left face commutes by definition of 
$\Delta^\le_{n,p+q}$. The right face commutes by tensoring with $s_q$ on the right in $\mathsf{C}$ the diagram defining $\Delta_{n,p}^\le$. The bottom face 
commutes by tensoring with $s_n$ on the left in $\mathsf{C}$ the diagram defining $\Delta_{p,q}^\le$. We deduce that the front face commutes as before.

\item[12.] Second diagram in \cref{def:graded-comonoid}:
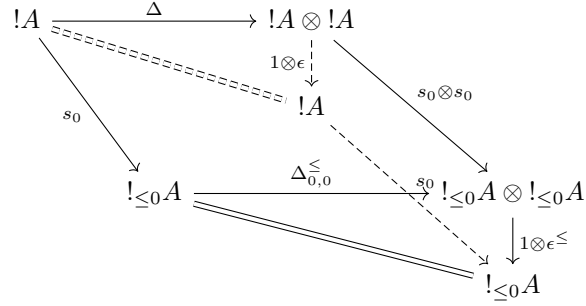
\begin{equation*}
\begin{tikzcd}
\oc A \arrow[rr, "\Delta"] \arrow[rdd, "s_0"'] \arrow[rrd, dashed, equal] &                                                           & \oc A \otimes \oc A \arrow[d, "1 \otimes \epsilon"', dashed] \arrow[rdd, "s_0 \otimes s_0"] &                                                                       &  & \ \\
                                                                   &                                                           & \oc A \arrow[rdd, "s_0", dashed]                                                            &                                                                       &  &   \\
                                                                   & \oc_{\le 0}A \arrow[rr, "{\Delta_{0,0}^\le}"] \arrow[rrd, equal] &                                                                                             & \oc_{\le 0}A \otimes \oc_{\le 0}A \arrow[d, "1 \otimes \epsilon^\le"] &  &   \\
                                                                   &                                                           &                                                                                             & \oc_{\le 0}A                                                          &  &  
\end{tikzcd}
\end{equation*}
The back face commutes by counitality of the comonoid $\oc A$. The upper face commutes by definition of $\Delta_{0,0}^\le$. The bottom face commutes trivially. 
The right face commutes by tensoring on the left with $s_0$ in $\mathsf{C}$ the diagram defining $\epsilon^\le$. We deduce that the front face commutes as usual.
\item[13.] Third diagram in \cref{def:graded-comonoid}:
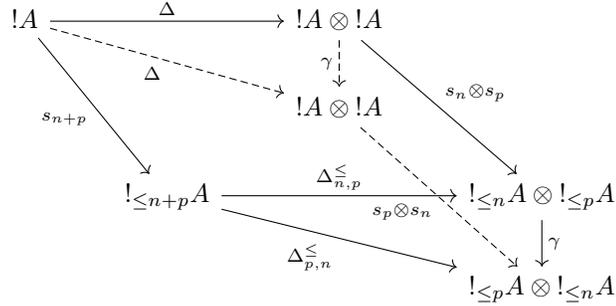
\begin{equation*}
\begin{tikzcd}
\oc A \arrow[rr, "\Delta"] \arrow[rdd, "s_{n+p}"'] \arrow[rrd, "\Delta"', dashed] &                                                                                    & \oc A \otimes \oc A \arrow[d, "\gamma"', dashed] \arrow[rdd, "s_n \otimes s_p"] &                                                       \\
                                                                                  &                                                                                    & \oc A \otimes \oc A \arrow[rdd, "s_p \otimes s_n"', dashed]                     &                                                       \\
                                                                                  & \oc_{\le n+p}A \arrow[rr, "{\Delta_{n,p}^\le}"] \arrow[rrd, "{\Delta_{p,n}^\le}"'] &                                                                                 & \oc_{\le n}A \otimes \oc_{\le p}A \arrow[d, "\gamma"] \\
                                                                                  &                                                                                    &                                                                                 & \oc_{\le p}A \otimes \oc_{\le n}A                    
\end{tikzcd}
\end{equation*}
The upper and bottom faces commute by the definitions of $\Delta_{n,p}^\le$ and $\Delta_{p,n}^\le$. The back face commutes by cocommutativity of $\Delta$. The right face commutes by 
naturality of $\gamma$. We deduce as before that the front face commutes.
\item[14.] \cref{def-d(-2)}:
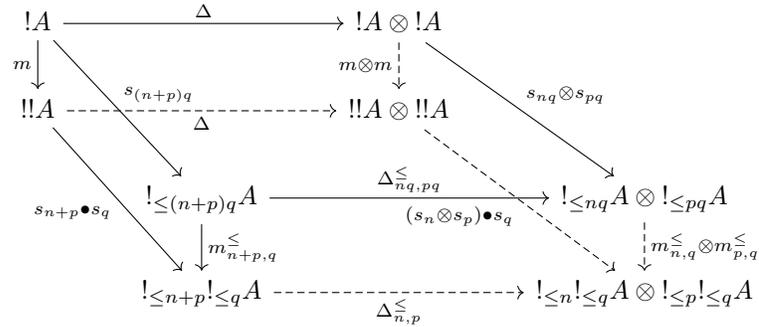
\begin{equation*}
\begin{tikzcd}
\oc A \arrow[rr, "\Delta"] \arrow[d, "m"'] \arrow[rdd, "s_{(n+p)q}"]       &                                                                                  & \oc A \otimes \oc A \arrow[d, "m \otimes m"', dashed] \arrow[rdd, "s_{nq} \otimes s_{pq}"] &                                                                                            \\
\oc\oc A \arrow[rr, "\Delta"', dashed] \arrow[rdd, "s_{n+p} \bullet s_q"'] &                                                                                  & \oc\oc A \otimes \oc\oc A \arrow[rdd, "(s_n \otimes s_p) \bullet s_q"', dashed]            &                                                                                            \\
                                                                           & \oc_{\le(n+p)q}A \arrow[rr, "{\Delta_{nq,pq}^\le}"] \arrow[d, "{m_{n+p,q}^\le}"] &                                                                                            & \oc_{\le nq}A \otimes \oc_{\le pq}A \arrow[d, "{m_{n,q}^\le \otimes m_{p,q}^\le}", dashed] \\
                                                                           & \oc_{\le n+p}\oc_{\le q}A \arrow[rr, "{\Delta_{n,p}^\le}"', dashed]              &                                                                                            & \oc_{\le n}\oc_{\le q}A \otimes \oc_{\le p}\oc_{\le q}A                                   
\end{tikzcd}
\end{equation*}
The back face commutes by \cref{def-d(-2)} for $\oc$. The upper face commutes by definition of $\Delta_{np,pq}^\le$. The bottom face commutes by tensoring with $s_q$ 
on the right in $\mathsf{End}(\mathsf{C})$ the diagram defining $\Delta_{n,p}^\le$. The left face commutes by definition of $m_{n+p,q}^\le$. For the right face, first 
note that we have 
\begin{equation}
(s_n \otimes s_p) \bullet s_q=(s_n \bullet s_q) \otimes (s_p \bullet s_q)
\end{equation}
by \cref{comp-two-tensor-products}. Using this equality, we obtain that the right face commutes by tensoring in $\mathsf{C}$ the diagram defining $m_{n,q}^\le$ with the diagram 
defining $m_{p,q}^\le$. We conclude as usual that the front face commutes.

\item[15.] \cref{def-d(-1)}:
\begin{equation*}
\begin{tikzcd}
\oc A \arrow[r, "m"] \arrow[rd, "\epsilon"', dashed] \arrow[rdd, "s_0"'] & \oc\oc A \arrow[d, "\epsilon", dashed] \arrow[rdd, "s_0 \bullet s_0"] &                                                    \\
                                                                         & I \arrow[rdd, dashed, equal]                                                 &                                                    \\
                                                                         & \oc_{\le 0}A \arrow[r, "{m_{0,0}^\le}"'] \arrow[rd, "\epsilon^\le"']  & \oc_{\le 0}\oc_{\le 0} A \arrow[d, "\epsilon^\le"] \\
                                                                         &                                                                       & I                                                 
\end{tikzcd}
\end{equation*}
The back face commutes by \cref{def-d(-1)} for $\oc$. The upper face commutes by definition of $m_{0,0}^\le$. The bottom face commutes by definition of $\epsilon^\le$. 
The right face commutes by tensoring on the right with $s_0$ in $(\mathsf{End}(\mathsf{C}),\bullet,1)$ the diagram defining $\epsilon^\le$. We conclude that the 
front face commutes.

\item[16.] \cref{def-d1}:
\begin{equation*}
\begin{tikzcd}
\oc A \otimes A \arrow[rd, "\epsilon \otimes 1"] \arrow[d, "\partial"'] \arrow[rdd, "s_0\otimes 1"'] &                                                                                      &   \\
\oc A \arrow[r, "u", dashed] \arrow[rdd, "s_1"']                                                     & A \arrow[rdd, equal]                                                                        &   \\
                                                                                                     & \oc_{\le 0} A \otimes A \arrow[d, "\partial_0^{\le}"'] \arrow[rd, "\epsilon^\le \otimes 1"'] &   \\
                                                                                                     & \oc_{\le 1} A \arrow[r, "u^\le"']                                                    & A
\end{tikzcd}
\end{equation*}
The back face commutes by \cref{def-d1} for $\oc$. The top face commutes by tensoring on the right with $1_A$ in $(\mathsf{C},\otimes,I)$ the diagram defining 
$\epsilon^\le$. The bottom face commutes by definition of $u^\le$. The left face commutes by definition of $\partial_0^\le$. We conclude that the front face 
commutes by using that $s_0 \otimes 1$ is an epimorphism as a coequalizer.

\item[17.] \cref{def-d2}:
Since this identity involves a sum, we will use inline calculations rather than diagrams. We will prove the identity obtained by precomposing each side of 
$\cref{def-d2}$ by $s_{n+p+1} \otimes 1_A$. We obtain $\cref{def-d2}$ from this by using that $s_{n+p+1} \otimes 1_A$ is an epimorphism. Here are the calculations:
\begin{align*}
&s_{n+p+1} \otimes 1_A;\partial_{n+p+1}^\le;\Delta_{n+1,p+1}^{\le} \\ 
&= \partial;s_{n+p+2};\Delta_{n+1,p+1}^{\le} &\text{\scriptsize{def.~of $\partial^\le_{n+p+1}$}} \\
&=\partial;\Delta;s_{n+1}\otimes s_{p+1} &\text{\scriptsize{def.~of $\Delta^\le_{n+1,p+1}$}} \\
&=\Big(\Delta \otimes 1_A;1_{\oc A} \otimes \partial+\Delta \otimes 1_A;1_{\oc A} \otimes \gamma;\partial \otimes 1_{\oc A}\Big);s_{n+1}\otimes s_{p+1} 
&\text{\scriptsize{\cref{def-d2} for $\oc$}} \\
&=\Delta \otimes 1_A;1_{\oc A} \otimes \partial;s_{n+1} \otimes s_{p+1} \\ 
&~+\Delta \otimes 1_A;1_{\oc A} \otimes \gamma;\partial \otimes 1_{\oc A};s_{n+1} \otimes s_{p+1} &\text{\scriptsize{\cref{add-3}}} \\
&=\Delta \otimes 1_A;s_{n+1} \otimes (\partial;s_{p+1}) &\text{\scriptsize{func.~of $-\otimes-$}} \\ 
&~+\Delta \otimes 1_A;1_{\oc A} \otimes \gamma;(\partial;s_{n+1}) \otimes s_{p+1} &\text{\scriptsize{func.~of $-\otimes-$}} \\
&=\Delta \otimes 1_A;s_{n+1} \otimes (s_{p} \otimes 1_A;\partial_p^{\le}) &\text{\scriptsize{def.~of $\partial_p^\le$}} \\ 
&~+\Delta \otimes 1_A;1_{\oc A} \otimes \gamma;(s_n \otimes 1_A;\partial_n^{\le}) \otimes s_{p+1} &\text{\scriptsize{def.~of $\partial_n^\le$}} \\
&=\Delta \otimes 1_A;s_{n+1} \otimes s_{p} \otimes 1_A;1_{\oc_{\le n+1}A} \otimes \partial_p^{\le} &\text{\scriptsize{func.~of $-\otimes-$}} \\ 
&~+\Delta \otimes 1_A;1_{\oc A} \otimes \gamma_{\oc A, A};s_n \otimes 1_A \otimes s_{p+1};\partial_n^{\le} \otimes 1_{\oc_{\le p+1}A} &\text{\scriptsize{func.~of $-\otimes-$}} \\
&=(\Delta;s_{n+1} \otimes s_p) \otimes 1_A;1_{\oc_{\le n+1}A} \otimes \partial_p^{\le} &\text{\scriptsize{func.~of $-\otimes-$}} \\
&~+\Delta \otimes 1_A;1_{\oc A} \otimes s_{p+1} \otimes 1_A;s_n \otimes \gamma_{\oc_{\le p+1}A, A};\partial_n^{\le} \otimes 1_{\oc_{\le p+1}A} 
&\text{\scriptsize{nat.~of $\gamma$}} \\
&=(\Delta;s_{n+1} \otimes s_p) \otimes 1_A;1_{\oc_{\le n+1}A} \otimes \partial_p^{\le} \\
&~+\Delta \otimes 1_A;s_n \otimes s_{p+1} \otimes 1_A;1_{\oc_{\le n}A} \otimes \gamma_{\oc_{\le p+1}A, A};\partial_n^{\le} \otimes 1_{\oc_{\le p+1}A} 
&\text{\scriptsize{func.~of $-\otimes-$}} \\
&=(\Delta;s_{n+1} \otimes s_p) \otimes 1_A;1_{\oc_{\le n+1}A} \otimes \partial_p^{\le} \\
&~+(\Delta;s_n \otimes s_{p+1}) \otimes 1_A;1_{\oc_{\le n}A} \otimes \gamma_{\oc_{\le p+1}A, A};\partial_n^{\le} \otimes 1_{\oc_{\le p+1}A} 
&\text{\scriptsize{func.~of $-\otimes-$}} \\
&=(s_{n+p+1};\Delta_{n+1,p}^{\le}) \otimes 1_A;1_{\oc_{\le n+1}A} \otimes \partial_p^{\le} &\text{\scriptsize{def.~of $\Delta_{n+1,p}^\le$}} \\
&~+(s_{n+p+1};\Delta_{n,p+1}^{\le}) \otimes 1_A;1_{\oc_{\le n}A} \otimes \gamma_{\oc_{\le p+1}A, A};\partial_n^{\le} \otimes 1_{\oc_{\le p+1}A} 
&\text{\scriptsize{def.~of $\Delta_{n,p+1}^\le$}} \\
&= s_{n+p+1} \otimes 1_A;\Delta_{n+1,p}^{\le} \otimes 1_A;1_{\oc_{\le n+1}A} \otimes \partial_p^{\le} &\text{\scriptsize{func.~of $-\otimes-$}} \\ 
&~+s_{n+p+1} \otimes 1_A;\Delta_{n,p+1}^{\le} \otimes 1_A;1_{\oc_{\le n}A} \otimes \gamma_{\oc_{\le p+1}A, A};\partial_n^{\le} \otimes 1_{\oc_{\le p+1}A} 
&\text{\scriptsize{func.~of $-\otimes-$}} \\
&= s_{n+p+1} \otimes 1_A;(
\Delta_{n+1,p}^{\le} \otimes 1_A;1_{\oc_{\le n+1}A} \otimes \partial_p^{\le} \\ 
&~+\Delta_{n,p+1}^{\le} \otimes 1_A;1_{\oc_{\le n}A} \otimes \gamma_{\oc_{\le p+1}A, A};\partial_n^{\le} \otimes 1_{\oc_{\le p+1}A}) &\text{\scriptsize{\cref{add-4}}}
\end{align*}
\item[18.] \cref{def-d3}:
\begin{equation*}
\begin{tikzcd}
\oc A \otimes A \arrow[rr, "\Delta \otimes 1"'] \arrow[d, "\partial"'] \arrow[rrddd, "s_{np+n+p} \otimes 1" description] &  & \oc A \otimes \oc A \otimes A \arrow[rr, "m \otimes \partial"'] \arrow[rrddd, "s_{np+n} \otimes s_p \otimes 1"'] &  & \oc\oc A \otimes \oc A \arrow[d, "\partial", dashed] \arrow[rrddd, "(s_n \bullet s_{p+1}) \otimes s_{p+1}"] &  &                                                                              \\
\oc A \arrow[rrrr, "m", dashed] \arrow[rrddd, "s_{(n+1)(p+1)}"']                                                         &  &                                                                                                                  &  & \oc\oc A \arrow[rrddd, "s_{n+1} \bullet s_{p+1}", dashed]                                                   &  &                                                                              \\
                                                                                                                         &  &                                                                                                                  &  &                                                                                                             &  &                                                                              \\
                                                                                                                         &  & \oc_{\le np+n+p}A \otimes A \arrow[rr, "{\Delta_{np+n,p}^\le \otimes 1}"'] \arrow[d, "\partial^\le_{np+n+p}"]    &  & \oc_{\le np+n}A \otimes \oc_{\le p}A \otimes A \arrow[rr, "{m_{n,p+1}^\le \otimes \partial_p^\le}"']        &  & \oc_{\le n}\oc_{\le p+1}A \otimes \oc_{\le p+1}A \arrow[d, "\partial^\le_n"] \\
                                                                                                                         &  & \oc_{\le(n+1)(p+1)}A \arrow[rrrr, "{m^\le_{n+1,p+1}}"']                                                          &  &                                                                                                             &  & \oc_{\le n+1}\oc_{\le p+1}A                                                 
\end{tikzcd}
\end{equation*}
The back face commutes by $\cref{def-d3}$ for $\oc$. The bottom face commutes by definition of $m_{n+1,p+1}^\le$. The left cell on the top face commutes by 
tensoring with $1_A$ on the right in $(\mathsf{C},\otimes,I)$ the diagram defining $\Delta_{np+n,p}^\le$. The right cell on the top face commutes by tensoring 
in $(\mathsf{C},\otimes,I)$ the diagram defining $m_{n,p+1}^\le$ with the diagram defining $\partial_p^\le$. We deduce that the top face commutes. The left face 
commutes by definition of $\partial_{np+n+p}^\le$. For the right face, first note that we have
\begin{align*}
(s_n \bullet s_{p+1}) \otimes s_{p+1}&=(s_n \bullet s_{p+1}) \otimes (1_{1_\mathsf{C}} \bullet s_{p+1}) \\
&=(s_n \otimes 1_{1_\mathsf{C}}) \bullet s_{p+1}
\end{align*}
by \cref{comp-two-tensor-products} where $1_{1_\mathsf{C}}\colon 1_\mathsf{C} \Rightarrow 1_\mathsf{C}$ is the natural transformation defined by $(1_{1_\mathsf{C}})(A):=1_A$.
We then obtain that the right face commutes by tensoring in $(\mathsf{End}(\mathsf{C}),\bullet,1_\mathsf{C})$ the diagram defining $\partial^\le_n$ with $s_{p+1}$ on the right. 
We deduce that the front face commutes since $s_{np+n+p} \otimes 1$ is an epimorphism.

\item[19.] \cref{def-d4}:
\begin{equation*}
\begin{tikzcd}
\oc A \otimes A \otimes A \arrow[rr, "1 \otimes \gamma"'] \arrow[d, "\partial \otimes 1"'] \arrow[rrddd, "s_n \otimes 1"] &  & \oc A \otimes A \otimes A \arrow[rr, "\partial \otimes 1"'] \arrow[rrddd, "s_n \otimes 1"]             &  & \oc A \otimes A \arrow[d, "\partial", dashed] \arrow[rrddd, "s_{n+1} \otimes 1"] &  &                                                           \\
\oc A \otimes A \arrow[rrrr, "\partial"', dashed] \arrow[rrddd, "s_{n+1} \otimes 1"]                                      &  &                                                                                                        &  & \oc A \arrow[rrddd, "s_{n+2}", dashed]                                           &  &                                                           \\
                                                                                                                          &  &                                                                                                        &  &                                                                                  &  &                                                           \\
                                                                                                                          &  & \oc_{\le n}A \otimes A \otimes A \arrow[rr, "1 \otimes \gamma"'] \arrow[d, "\partial_n^\le \otimes 1"] &  & \oc_{\le n}A \otimes A \otimes A \arrow[rr, "\partial_n^\le \otimes 1"']         &  & \oc_{\le n+1} A \otimes A \arrow[d, "\partial_{n+1}^\le"] \\
                                                                                                                          &  & \oc_{\le n+1} A \otimes A \arrow[rrrr, "\partial_{n+1}^\le"']                                          &  &                                                                                  &  & \oc_{\le n+2}A                                           
\end{tikzcd}
\end{equation*}
The back face commutes by \cref{def-d4} for $\oc$. The left cell on the top face commutes since the two paths from $\oc A \otimes A \otimes A$ to $\oc_nA \otimes A \otimes A$ 
are equal to $s_n \otimes \gamma\colon \oc A \otimes A \otimes A \rightarrow \oc_nA \otimes A \otimes A$. The right cell on the top face commutes by tensoring on the right 
with $1_A$ in $(\mathsf{C},\otimes,I)$ the diagram defining $\partial_n^\le$. We deduce that the top face commutes. The bottom and the right faces commute by the definition 
of $\partial_{n+1}^\le$. The left face commutes by tensoring on the right with $1_A$ in $(\mathsf{C},\otimes,I)$ the diagram defining $\partial_n^\le$. We conclude that the 
front face commutes by using that $s_n \otimes 1$ is an epimorphism.
\end{itemize}
\subsection{Checking the equations for a $(0\colon \mathbb{N} \rightarrow 0)$-graded morphism of graded differential modalities}
We must prove that for every $n \ge 0$, the family of morphisms
\begin{equation*}
s_n\colon \oc A \rightarrow \oc_nA
\end{equation*}
for $A \in \mathsf{C}$ is a natural transfomation. We must also check that the five diagrams in \cref{def:morphism-of-graded-diff-mod} commute.

\textbf{Naturality of $s_n$:} We have already explained that this is obtained immediately from the definition of $\oc_{\le n}f$.

\textbf{Commutativity of the five diagrams in \cref{def:morphism-of-graded-diff-mod}:} The diagram involving $m_{n,p}^\le$ commutes immediately from the definition of 
$m_{n,p}^\le$. The diagram involving $u$ commutes immediately from the definition of $u^\le$. The same is true of the diagrams involving $\Delta^\le_{n,p}, \epsilon^\le$, 
and $\partial^\le_n$.

\subsection{What has been proven so far}
We have currently proven that there exists an $\mathbb{N}$-graded differential modality $(\oc_{\le n},\epsilon^\le,u^\le,\Delta^\le_{n,p},m^\le _{n,p},\partial^\le_n)$ on $(\mathsf{C},\otimes,I)$ such that:
\begin{itemize}
\item[1.] the endofunctor $\oc_{\le n}\colon\mathsf{C} \rightarrow \mathsf{C}$ applied to any object $A$ gives $\oc_{\le n}A$,
\item[2.] $(s_n)_{n \in \mathbb{N}}$ is a $(0\colon\mathbb{N} \rightarrow 0)$-graded morphism of graded differential modalities.
\end{itemize}
We will now prove that this $\mathbb{N}$-graded differential modality is unique.
\subsection{Unicity of the $\mathbb{N}$-graded differential modality}
Suppose that $(\oc_{\le n}^*,\epsilon^{\le,*},u^{\le,*},\Delta^{\le,*}_{n,p},m^{\le,*}_{n,p},\partial^{\le,*}_n)$ is another $\mathbb{N}$-graded differential modality such that:
\begin{itemize}
\item[1.] the endofunctor $\oc_{\le n}^*\colon\mathsf{C} \rightarrow \mathsf{C}$ applied to any object $A$ gives $\oc_{\le n}A$,
\item[2.] $(s_n)_{n \in \mathbb{N}}$ is a $(0\colon\mathbb{N} \rightarrow 0)$-graded morphism of graded differential modalities.
\end{itemize}
We will prove that 
\begin{equation*}
(\oc_{\le n}^*,\epsilon^{\le,*},u^{\le,*},\Delta^{\le,*}_{n,p},m^{\le,*}_{n,p},\partial^{\le,*}_n)=(\oc_{\le n},\epsilon^\le,u^\le,\Delta^\le_{n,p},m^\le_{n,p},\partial^\le_n).
\end{equation*}

\textbf{Proof of $\oc_{\le n}^*=\oc_{\le n}$:} Let $f \in \mathsf{C}(A,B)$. Since $(s_n)$ is a $(0\colon\mathbb{N} \rightarrow 0)$-graded morphism of graded differential 
modalities, we obtain that the diagram
\begin{equation*}
\begin{tikzcd}
\oc A \arrow[d, "\oc f"] \arrow[r, "s_n"] & \oc_{\le n} A \arrow[d, "\oc_{\le n}^*f"] \\
\oc B \arrow[r, "s_n"]                    & \oc_{\le n} B                          
\end{tikzcd}
\end{equation*}
commutes. By the definition of $\oc_{\le n}f$, it implies that $\oc_{\le n}^*f=\oc_{\le n}f$.

\textbf{Proof of $\epsilon^{\le,*}=\epsilon^\le$:} Let $A \in \mathsf{C}$. Since $(s_n)$ is a $(0\colon\mathbb{N} \rightarrow 0)$-graded morphism of graded differential 
modalities, we obtain that the diagram 
\begin{equation*}
\begin{tikzcd}
\oc A \arrow[r, "s_0"] \arrow[rd, "\epsilon"'] & \oc_{\le 0}A \arrow[d, "{\epsilon^{\le,*}}"] \\
                                               & I                                           
\end{tikzcd}
\end{equation*}
commutes. By the definition of $\epsilon^{\le}$ it implies that $\epsilon^{\le,*}=\epsilon^\le\colon \oc A \rightarrow \oc_{\le 0}A$.

\textbf{Proof of $u^{\le,*}=u^\le$, $\Delta_{n,p}^{\le,*}=\Delta_{n,p}^\le$, $m_{n,p}^{\le,*}=m_{n,p}^\le$ and $\partial_n^{\le,*}=\partial_n^\le$:} These identities 
are proven in exactly the same way as $\epsilon^{\le,*}=\epsilon^\le$.
\subsection{What has been proven so far}
We have currently proven that there exists a \emph{unique} $\mathbb{N}$-graded differential modality $(\oc_{\le n},\epsilon^\le,u^\le,\Delta^\le_{n,p},m^\le _{n,p},\partial^\le_n)$ on $(\mathsf{C},\otimes,I)$
such that:
\begin{itemize}
\item[1.] the endofunctor $\oc_{\le n}\colon\mathsf{C} \rightarrow \mathsf{C}$ applied to any object $A$ gives $\oc_{\le n}A$,
\item[2.] $(s_n)_{n \in \mathbb{N}}$ is a $(0\colon\mathbb{N} \rightarrow 0)$-graded morphism of graded differential modalities.
\end{itemize}
We will now prove the second part of \cref{theorem:extracting}, that is:

``Moreover, there exists a unique family of morphisms $t_{n,p}\colon \oc_{\le n}A \rightarrow \oc_{\le p}A$ defined for all $(n,p) \in \mathbb{N}^2$ with $n \ge p$ 
and $A \in \mathsf{C}$ such that:
\begin{itemize}
\item[1.] $(\oc_{\le n},\epsilon^\le,u^\le,\Delta^\le_{n,p},m^\le _{n,p},\partial^\le_n,t_{n,p})$ is an $\mathbb{N}$-filtered differential modality,
\item[2.] $(s_n)_{n \in \mathbb{N}}$ is a $(0\colon\mathbb{N} \rightarrow 0)$-filtered morphism of $\mathbb{N}$-filtered differential modalities.''
\end{itemize}
\subsection{Defining the data for an $\mathbb{N}$-filtered differential modality}
We will define $t_{n,p}\colon \oc_{\le n}A \rightarrow \oc_{\le p}A$ for all $(n,p) \in \mathbb{N}^2$ with $n \ge p$. Let $(n,p) \in \mathbb{N}^2$ be such an ordered pair 
and let $A \in \mathsf{C}$. We use the following diagram:
\begin{equation*}
\begin{tikzcd}
\oc A \otimes A^{\otimes (n+1)} \arrow[rr, "\partial^{n-p} \otimes 1_{A^{\otimes(p+1)}}"'] \arrow[rrrr, "\partial^{n+1}", bend left] &  & \oc A \otimes A^{\otimes (p+1)} \arrow[rr, "\partial^{p+1}"'] &  & \oc A \arrow[rd, "s_p"'] \arrow[r, "s_n"'] & \oc_{\le n}A \arrow[d, "{t_{n,p}}", dashed] \\
                                                                                                                                     &  &                                                               &  &                                            & \oc_{\le p}A~.                               
\end{tikzcd}
\end{equation*}
Some comments are needed: The top cell commutes by \cref{d-k-l}. We thus have
\begin{equation*}
\partial^{n+1};s_p=\partial^{n-p} \otimes 1_{A^{\otimes (p+1)}};\partial^{p+1};s_p=\partial^{n-p} \otimes 1_{A^{\otimes (p+1)}};0=0
\end{equation*}
by definition of $s_p$. By definition of $s_n$, there exists a unique morphism $t_{n,p}$ such that the triangle commutes.
\subsection{Checking the equations for an $\mathbb{N}$-filtered differential modality}
We will now check that $(\oc_{\le n},\epsilon^\le,u^\le,\Delta^\le_{n,p},m^\le _{n,p},\partial^\le_n)$ is an $\mathbb{N}$-filtered differential modality. Recall that we 
already know this is an $\mathbb{N}$-graded differential modality. All we have to prove is that each $t_{n,p}\colon \oc_{\le n}A \rightarrow \oc_{\le p}A$ is a natural 
transformation, that the two identities in \cref{def:filtered-diff-mod} are satisfied and that  and that the three diagrams in \cref{def:filtered-diff-mod} commute.

\textbf{Naturality of $t_{n,p}$:} Let $(n,p) \in \mathbb{N}^2$ such that $n \ge p$ and let $f \in \mathsf{C}(A,B)$. We must prove that the diagram
\begin{equation*}
\begin{tikzcd}
\oc_{\le n}A \arrow[rr, "{t_{n,p}}"] \arrow[d, "\oc_{\le n}f"'] &  & \oc_{\le p}A \arrow[d, "\oc_{\le p}f"] \\
\oc_{\le n}B \arrow[rr, "{t_{n,p}}"']                           &  & \oc_{\le p}B                          
\end{tikzcd}
\end{equation*}
commutes. We will use the following diagram
\begin{equation*}
\begin{tikzcd}
\oc B \arrow[rdd, "s_n"'] &                                                                    & \oc A \arrow[ld, "s_n"'] \arrow[rd, "s_p"] \arrow[ll, "\oc f"'] \arrow[rr, "\oc f"] &                                         & \oc B \arrow[ldd, "s_p"] \\
                          & \oc_{\le n}A \arrow[rr, "{t_{n,p}}"] \arrow[d, "\oc_{\le n}f"] &                                                                                     & \oc_{\le p}A \arrow[d, "\oc_{\le p}f"'] &                          \\
                          & \oc_{\le n}B \arrow[rr, "{t_{n,p}}"']                          &                                                                                     & \oc_{\le p}B~.                            &                         
\end{tikzcd}
\end{equation*}
In this diagram, the cell on the left and the cell on the right commute by the naturalities of $s_n$ and $s_p$. The triangle commutes by definition of $t_{n,p}^\le$. 
Finally, the two exterior paths from $\oc A$ to $\oc_{\le p}B$ are equal since $s_n;t_{n,p}^\le=s_p$ by definition of $t_{n,p}^\le$. By diagram chasing, we obtain that
\begin{equation*}
s_n;\oc_{\le n}f;t_{n,p}=s_n;t_{n,p};\oc_{\le p}f.
\end{equation*}
Since $s_n$ is an epimorphism, we deduce that the square commutes.

\textbf{First identity in \cref{def:filtered-diff-mod}:} Let $n \in \mathbb{N}$. We have $s_{n};1_{\oc_{\le n}A}=s_n$. By definition of $t_{n,n}$, it implies that 
$t_{n,n}=1_{\oc_{\le n}A}$.

\textbf{Second identity in \cref{def:filtered-diff-mod}:} Let $n,p,q \in \mathbb{N}$ such that $n \ge p \ge q$. We must prove that $t_{n,p};t_{p,q}=t_{n,q}$. 
We will use the following diagram:
\begin{equation*}
\begin{tikzcd}
\oc A \arrow[r, "s_n"] \arrow[rd, "s_p"] \arrow[rdd, "s_q"'] & \oc_{\le n}A \arrow[d, "{t_{n,p}}"] \\
                                                             & \oc_{\le p}A \arrow[d, "{t_{p,q}}"] \\
                                                             & \oc_{\le q}A~.                      
\end{tikzcd}
\end{equation*}
In this diagram the two inner triangles commute by the definitions of $t_{n,p}$ and $t_{p,q}$. It follows that the exterior triangle commute. By definition of $t_{n,q}$, 
we thus have $t_{n,p};t_{p,q}=t_{n,q}$.

\textbf{Commutativity of the first diagram in \cref{def:filtered-diff-mod}:} Let $n,p,q,r \in \mathbb{N}$ such that $n \ge q$ and $p \ge r$, and let $A \in \mathsf{C}$.

We must prove that the diagram
\begin{equation*}
\begin{tikzcd}
\oc_{\le n+p}A \arrow[rr, "{t_{n+p,q+r}}"] \arrow[d, "{\Delta_{n,p}^\le}"'] &  & \oc_{\le q+r}A \arrow[d, "{\Delta_{q,r}^\le}"] \\
\oc_{\le n}A \otimes \oc_{\le p}A \arrow[rr, "{t_{n,q} \otimes t_{p,r}}"'] &  & \oc_{\le q}A \otimes \oc_{\le r}A              
\end{tikzcd}
\end{equation*}
commutes. We will use the following diagram:
\begin{equation*}
\begin{tikzcd}
                                                   & \oc A \arrow[d, "s_{n+p}"'] \arrow[rrd, "s_{q+r}"] \arrow[ld, "\Delta"'] \arrow[rrr, "\Delta"] &  &                                         & \oc A \otimes \oc A \arrow[ldd, "s_q \otimes s_r"] \\
\oc A \otimes \oc A \arrow[rd, "s_n \otimes s_p"'] & \oc_{\le n+p}A \arrow[rr, "{t_{n+p,q+r}}"'] \arrow[d, "{\Delta_{n,p}^\le}"']                           &  & \oc_{\le q+r}A \arrow[d, "{\Delta_{q,r}^\le}"'] &                                                    \\
                                                   & \oc_{\le n}A \otimes \oc_{\le p}A \arrow[rr, "{t_{n,q} \otimes t_{p,r}}"']                                 &  & \oc_{\le q}A \otimes \oc_rA~.                  &                                                   
\end{tikzcd}
\end{equation*}
In this diagram the two cells involving $\Delta$ commute by the definitions of $\Delta_{n,p}^\le$ and $\Delta_{q,r}^\le$. Moreover, the triangle commutes by definition of 
$t_{n+p,q+r}$. Finally, the two exterior paths going from $\oc A$ to $\oc_{\le q}A \otimes \oc_{\le r}A$ are equal since $s_n;t_{n,q}=s_q$ and $s_p;t_{p,r}=s_r$ by the definitions of 
$t_{n,q}$ and $t_{p,r}$. We obtain by diagram chasing that
\begin{equation*}
s_{n+p};\Delta^\le_{n,p};t_{n,q} \otimes t_{p,r}=s_{n+p};t_{n+p,q+r};\Delta^\le_{q,r}.
\end{equation*}
Since $s_{n+p}$ is an epimorphism, we deduce that the square commutes.

\textbf{Commutativity of the second diagram in \cref{def:filtered-diff-mod}:} Let $n,p,q,r \in \mathbb{N}$ such that $n \ge q$ and $p \ge r$. And let $A \in \mathsf{C}$. 
We must prove that the diagram
\begin{equation*}
\begin{tikzcd}
\oc_{\le np}A \arrow[rr, "{t_{np,qr}}"] \arrow[d, "{m_{n,p}^\le}"'] &  & \oc_{\le qr}A \arrow[d, "{m_{q,r}^\le}"] \\
\oc_{\le n}\oc_{\le p}A \arrow[rr, "{t_{n,q} \bullet t_{p,r}}"']    &  & \oc_{\le q}\oc_{\le r}A                 
\end{tikzcd}
\end{equation*}
commutes. We will use the following diagram:
\begin{equation*}
\begin{tikzcd}
                                        & \oc A \arrow[d, "s_{np}"'] \arrow[rrd, "s_{qr}"] \arrow[ld, "m"'] \arrow[rrr, "m"] &  &                                           & \oc\oc A \arrow[ldd, "s_q \bullet s_r"] \\
\oc\oc A \arrow[rd, "s_n \bullet s_p"'] & \oc_{\le np}A \arrow[rr, "{t_{np,qr}}"'] \arrow[d, "{m_{n,p}^\le}"']               &  & \oc_{\le qr}A \arrow[d, "{m_{q,r}^\le}"'] &                                         \\
                                        & \oc_{\le n}\oc_{\le p}A \arrow[rr, "{t_{n,q} \bullet t_{p,r}}"']                   &  & \oc_{\le q}\oc_{\le r}A                   &                                        
\end{tikzcd}
\end{equation*}
In this diagram the two cells involving $m$ commute by the definitions of $m_{n,p}^\le$ and $m_{q,r}^\le$. Moreover, the triangle commutes by definition of $t_{np,qr}$. 
Finally, the two exterior paths going from $\oc A$ to $\oc_{\le q}\oc_{\le r}A$ are equal since $s_n;t_{n,q}=s_q$ and $s_p;t_{p,r}=s_r$ by definitions of $t_{n,q}$ 
and $t_{p,r}$. We obtain by diagram chasing that
\begin{equation*}
s_{np};m^\le_{n,p};t_{n,q} \bullet t_{p,r}=s_{np};t_{np,qr};m_{q,r}^\le.
\end{equation*}
Since $s_{np}$ is an epimorphism, it follows that the square commutes.

\textbf{Commutativity of the third diagram in \cref{def:filtered-diff-mod}:} Let $n,p \in \mathbb{N}$ such that $n \ge p$ and let $A \in \mathsf{C}$. 
We must prove that the diagram
\begin{equation*}
\begin{tikzcd}
\oc_{\le n}A \otimes A \arrow[d, "\partial_n^\le"'] \arrow[rr, "{t_{n,p} \otimes 1}"] &  & \oc_{\le p}A \otimes A \arrow[d, "\partial_p^\le"] \\
\oc_{\le n+1}A \arrow[rr, "{t_{n+1,p+1}}"']                                           &  & \oc_{\le p+1}A                                    
\end{tikzcd}
\end{equation*}
commutes. We will use the following diagram:
\begin{equation*}
\begin{tikzcd}
                             & \oc A \otimes A \arrow[d, "s_n \otimes 1"'] \arrow[rrd, "s_p \otimes 1"] \arrow[ld, "\partial"'] \arrow[rrr, "\partial"] &  &                                                     & \oc A \arrow[ldd, "s_{p+1}"] \\
\oc A \arrow[rd, "s_{n+1}"'] & \oc_{\le n}A \otimes A \arrow[d, "\partial_n^\le"'] \arrow[rr, "{t_{n,p} \otimes 1}"]                                    &  & \oc_{\le p}A \otimes A \arrow[d, "\partial_p^\le"'] &                              \\
                             & \oc_{\le n+1}A \arrow[rr, "{t_{n+1,p+1}}"']                                                                              &  & \oc_{\le p+1}A~.                                     &                             
\end{tikzcd}
\end{equation*}
In this diagram the two cells involving $\partial$ commute by the definition of $\partial_n^\le$ and the definition of $\partial_p^\le$. Moreover, the triangle commutes 
by definition of $t_{n,p}$. Finally, the two exterior paths going from $\oc A \otimes A$ to $\oc_{\le p+1}A$ are equal since $s_{n+1};t_{n+1,p+1}=s_{p+1}$ by definition 
of $t_{n+1,p+1}$. We obtain by diagram chasing that
\begin{equation*}
s_n \otimes 1;\partial_n^\le;t_{n+1,p+1}=s_n \otimes 1;t_{n,p} \otimes 1;\partial_p^\le.
\end{equation*}
Since $s_n \otimes 1$ is an epimorphism as a coequalizer, we deduce that the square commutes.
\subsection{Checking the equations for a $(0\colon \mathbb{N} \rightarrow 0)$-filtered morphism of filtered differential modalities}
We want to prove that the family of all the natural transformations $s_n:\oc_{\le n}A \rightarrow \oc A$ is a $(0\colon \mathbb{N} \rightarrow 0)$-filtered morphism of 
filtered differential modalities from $(\oc_{\le n},\epsilon^\le,u^\le,\Delta^\le_{n,p},m^\le _{n,p},\partial^\le_n)$ to $(\oc,\epsilon,u,\Delta,m,\partial)$. 
We already know that $(s_n)$ is a morphism of graded differential modalities. All we have to prove is that the diagram in \cref{def-morphism-filtered} commutes. 
That is, we must prove that for all $(n,p) \in \mathbb{N}^2$ with $n \ge p$ and $A \in \mathsf{C}$, the diagram
\begin{equation*}
\begin{tikzcd}
\oc A \arrow[d, "s_n"'] \arrow[rrd, "s_p"] &  &              \\
\oc_{\le n}A \arrow[rr, "{t_{n,p}}"']      &  & \oc_{\le p}A
\end{tikzcd}
\end{equation*}
commutes. But this commutativity follows immediately from the definition of $t_{n,p}$.
\subsection{What has been proven so far}
So far, we have proven that there exists a unique $\mathbb{N}$-graded differential modality $(\oc_{\le n},\epsilon^\le,u^\le,\Delta^\le_{n,p},m^\le _{n,p},\partial^\le_n)$ on $(\mathsf{C},\otimes,I)$
such that:
\begin{itemize}
\item[1.] the endofunctor $\oc_{\le n}\colon\mathsf{C} \rightarrow \mathsf{C}$ applied to any object $A$ gives $\oc_{\le n}A$,
\item[2.] $(s_n)_{n \in \mathbb{N}}$ is a $(0\colon\mathbb{N} \rightarrow 0)$-graded morphism of graded differential modalities.
\end{itemize}
We have also proven that there exists a family of morphisms $t_{n,p}\colon \oc_{\le n}A \rightarrow \oc_{\le p}A$ defined for all $(n,p) \in \mathbb{N}^2$ with $n \ge p$ 
and $A \in \mathsf{C}$ such that:
\begin{itemize}
\item[1.] $(\oc_{\le n},\epsilon^\le,u^\le,\Delta^\le_{n,p},m^\le _{n,p},\partial^\le_n,t_{n,p})$ is an $\mathbb{N}$-filtered differential modality,
\item[2.] $(s_n)_{n \in \mathbb{N}}$ is a $(0\colon\mathbb{N} \rightarrow 0)$-filtered morphism of $\mathbb{N}$-filtered differential modalities.
\end{itemize}
We will now prove that the family of morphisms $t_{n,p}$ is unique.
\subsection{Unicity of the $t_{n,p}$}
Suppose that for every $(n,p) \in \mathbb{N}^2$ such that $n \ge p$, there exists a family of morphisms $t_{n,p}^* \colon \oc_{\le n}A \rightarrow \oc_{\le p}A$ where $A \in \mathsf{C}$ such that:
\begin{itemize}
\item[1.] $(\oc_{\le n},\epsilon^\le,u^\le,\Delta^\le_{n,p},m^\le _{n,p},\partial^\le_n,t_{n,p}^*)$ is an $\mathbb{N}$-filtered differential modality,
\item[2.] $(s_n)_{n \in \mathbb{N}}$ is a $(0\colon\mathbb{N} \rightarrow 0)$-filtered morphism of $\mathbb{N}$-filtered differential modalities.
\end{itemize}
Let $(n,p) \in \mathbb{N}^2$ such that $n \ge p$. From \cref{def-morphism-filtered}, the diagram
\begin{equation*}
\begin{tikzcd}
\oc A \arrow[d, "s_n"'] \arrow[rrd, "s_p"] &  &              \\
\oc_{\le n}A \arrow[rr, "{t_{n,p}^*}"']      &  & \oc_{\le p}A
\end{tikzcd}
\end{equation*}
commutes. By the definition of $t_{n,p}$, we thus have $t_{n,p}=t_{n,p}^*$. This concludes the proof of \cref{theorem:extracting}.
\section{The example of $\mathsf{Rel}$} \label{sec:example-rel}
The category $\mathsf{Rel}$ of sets and relations can be made into a cartesian symmetric monoidal category $(\mathsf{Rel},\times,\{*\})$. Moreover, this is an additive 
symmetric monoidal category where the sum of two relations $R$ and $S$ from $X$ to $Y$ is given by $R+S:=R \cup S$, and where the zero morphism from $X$ to $Y$ is 
$0:=\emptyset$. We define the differential modality as the multiset functor $\mathcal{M}\colon \mathsf{Rel} \rightarrow \mathsf{Rel}$. We can define the natural transformations 
\begin{align*}
m\colon \mathcal{M}A \rightarrow \mathcal{M}\mathcal{M}A, \\
u\colon \mathcal{M}A \rightarrow A, \\
\Delta\colon \mathcal{M}A \rightarrow \mathcal{M}A \times \mathcal{M}A, \\
\epsilon\colon \mathcal{M}A \rightarrow \{*\} \\
\partial\colon \mathcal{M}A \times A \rightarrow \mathcal{M}A
\end{align*}
as follows:
\begin{align*}
m=\{(M,N) \in \mathcal{M}A \times \mathcal{M}\mathcal{M}A~|~M=\sum N\}, \\
u=\{([x],x)~|~x \in A\}, \\
\Delta=\{(M,(N,P)) \in \mathcal{M}A \times (\mathcal{M}A \times \mathcal{M}A)~|~M=N+P\}, \\
\epsilon=\{([\,],*)\}, \\
\partial=\{((M,x),M+[x])~|~M \in \mathcal{M}A,~x \in A\}.
\end{align*}
This is a well-known differential modality \cite{DIFCAT}.

In $\mathsf{Rel}$, given any relation $R\colon X \rightarrow Y$, a cokernel of $R$ is given by
\begin{equation*}
S\colon Y \rightarrow Y \backslash \mathrm{im}(R)
\end{equation*}
where we define $\mathrm{im}(R):=\{y \in Y~|~\exists x \in X,~(x,y) \in R\}$ and $S:=\{(y,y),~y \in Y \backslash \mathrm{im}(R)\}$. Moreover, it is easy to check 
that for every other set $Z$, the following is a cokernel diagram:
\begin{equation*}
\begin{tikzcd}
X \times Z \arrow[r, "R \times 1"] & Y \times Z \arrow[r, "S \times 1"] & (Y \backslash \mathrm{im}(R)) \times Z~.
\end{tikzcd}
\end{equation*}
For every $n \ge 0$, the relation
\begin{equation*}
\partial^n\colon \mathcal{M}X \times X^n \rightarrow \mathcal{M}X
\end{equation*}
is given by
\begin{equation*}
\partial^n:=\{((M,(x_1,\dots,x_n)),M+[x_1,\dots,x_n])~|~M \in \mathcal{M}X,~(x_1,\dots,x_n) \in X^n\}
\end{equation*}
and we have $\mathrm{im}(\mathrm{\partial^n})=\{M \in \mathcal{M}(x)~|~|M| \ge n\}$. We thus obtain a cokernel diagram
\begin{equation*}
\begin{tikzcd}
\mathcal{M}X \times X^n \arrow[r, "\partial^{n+1}"] & \mathcal{M}X \arrow[r, "s_n"] & \mathcal{M}_{\le n}X
\end{tikzcd}
\end{equation*}
where $s_n=\{(M,M)~|~M \in \mathcal{M}_{\le n}X\}$.

We can apply \cref{theorem:extracting} to obtain an $\mathbb{N}$-filtered differential modality $(\mathcal{M}_{\le n})$ 
on $\mathsf{Rel}$ and a morphism of filtered differential modalities from $\mathcal{M}$ to $(\mathcal{M}_{\le n})$ where for all set $X$ and $n \ge 0$, we have 
$\mathcal{M}_{\le n}X=\{M \in \mathcal{M}X~|~|M| \le n\}$. We thus obtain natural transformations of the following type:
\begin{align*}
m_{n,p}^\le\colon \mathcal{M}_{\le np}A \rightarrow \mathcal{M}_{\le n}\mathcal{M}_{\le p}A, \\
u^\le\colon \mathcal{M}_{\le 1}A \rightarrow A, \\
\nabla_{n,p}^\le\colon \mathcal{M}_{\le n+p}A \rightarrow \mathcal{M}_{\le n}A \times \mathcal{M}_{\le p}A, \\
\eta^\le\colon \mathcal{M}_{\le 0}A \rightarrow \{*\}, \\
\partial_{\le n}^\le \colon \mathcal{M}_{\le n}A \times A \rightarrow \mathcal{M}_{\le n+1}A, \\
t_{n,p}\colon \mathcal{M}_{\le n}A \rightarrow \mathcal{M}_{\le p}A
\end{align*}
where $t_{n,p}$ is defined for all $n,p \ge 0$ such that $n \ge p$ by $t_{n,p}=\{(M,M)~M \in \mathcal{M}_{\le p}X\}$.

\section{The example of the symmetric algebra} \label{sec:example}
Let $k$ be a field. The additive symmetric monoidal category $\mathsf{Vec}_k^{\mathrm{op}}$ is a differential category. We define the differential modality as 
the symmetric algebra functor $S\colon \mathsf{Vec}_k \rightarrow \mathsf{Vec}_k$. This is automatically a monad since $S=U \circ F$ where 
$(U\colon\mathsf{CAlg}_k \rightarrow \mathsf{Vec}_k) \dashv (F\colon \mathsf{Vec}_k \rightarrow \mathsf{CAlg}_k)$ is the free commutative algebra adjunction. 
We thus have the following natural transformations:
\begin{align*}
m\colon SSA \rightarrow SA, \\
u\colon A \rightarrow SA, \\
\nabla\colon SA \otimes SA \rightarrow SA, \\
\eta\colon I \rightarrow SA.
\end{align*}
We have
\begin{equation*}
SA=\underset{n \ge 0}{\bigoplus}S^nA
\end{equation*}
where $S^0A=k$ and for every $n \ge 1$,
\begin{equation*}
S^nA=\frac{A^{\otimes n}}{\mathrm{span}_k\{a_1 \otimes \dots \otimes a_n - a_{\sigma(1)} \otimes \dots \otimes a_{\sigma(n)}~|~(a_1,\dots,a_n) \in A^n,~\sigma \in S_n\}}.
\end{equation*}
Let $n \ge 1$. If $(a_1,\dots,a_n) \in A^n$, then we write $a_1 \otimes_s \dots \otimes_s a_n$ for the image of $a_1 \otimes \dots \otimes a_n \in A^{\otimes n}$ 
under the quotient map $A^{\otimes n} \rightarrow S^nA$. We call $a_1 \otimes_s \dots \otimes_s a_n$ a pure symmetric tensor. Each element of $S^nA$ can be written 
as a finite sum of pure symmetric tensors, although almost never in a unique way. The universal property of $S^nA$ is a follows \cite{BOURBAKI}: given any vector space $B$ and 
symmetric multilinear map\footnote{A symmetric multilinear map $\phi\colon A^n \rightarrow B$ is a function such that for every permutation $\sigma \in S^n$, 
we have $m(a_{\sigma(1)},\dots,a_{\sigma(n)})=m(a_1,\dots,a_n)$, and for all $1 \le k \le n$ and $(a_1,\dots,a_{k-1},a_{k+1},\dots,a_n) \in A^{n-1}$, 
the map $a_k \mapsto m(a_1,\dots,a_n)$ is linear.} $\phi\colon A^n \rightarrow B$, there exists a unique linear map $\psi\colon S^nA \rightarrow B$ 
such that $\psi(a_1 \otimes_s \dots \otimes_s a_n)=\phi(a_1,\dots,a_n)$.

We now define the deriving transformation
\begin{equation*}
\partial\colon SA \rightarrow SA \otimes A
\end{equation*}
as follows:
\begin{equation*}
\partial(\lambda):=0
\end{equation*}
for every $\lambda \in S^0A=k$ and for every $n \ge 1$,
\begin{equation*}
\partial(a_1 \otimes_s \dots \otimes_s a_n):=\underset{1 \le i \le n}{\sum}(a_1 \otimes_s \dots \otimes_s a_{i-1} \otimes_s a_{i+1} \otimes_s \dots \otimes_s a_n) \otimes a_i.
\end{equation*}
The map $\partial$ is well-defined on $S^nA$ by using the universal property of $S^nA$. Indeed, the map $\phi\colon A^n \rightarrow SA \otimes A$ defined by 
$\phi(a_1,\dots,a_n):=\underset{1 \le i \le n}{\sum}(a_1 \otimes_s \dots \otimes_s a_{i-1} \otimes_s a_{i+1} \otimes_s \dots \otimes_s a_n) \otimes a_i$ is symmetric multilinear.

Note that all the natural transformations are in the opposite direction since the differential category is $\mathsf{Vec}_k^{\mathrm{op}}$ and not $\mathsf{Vec}_k$. 
\textbf{From here and until the end of this section, we will denote by $E$ a vector space with a basis $(e_i)_{i \in I}$ whereas $A$ will denote an arbitrary vector space.} 
We will simplify the description of the deriving transformation using polynomials. Let $n \ge 1$. We can define
\begin{equation*}
e_M:=e_1^{M(1)} \otimes_s \cdots \otimes_s e_{n}^{M(n)}
\end{equation*}
for any multiset $M \in \mathcal{M}_n(I)$. We obtain a basis $(e_M)_{M \in \mathcal{M}_n(I)}$ of $S^nE$. We can also define
\begin{equation*}
x^M:=x_1^{M(1)}\cdots x_n^{M(n)} \in k[x_i,~i \in I]
\end{equation*}
for every multiset $M \in \mathcal{M}_n(I)$. We have an isomorphism of $k$-algebras
\begin{equation} \label{iso-phi}
\phi\colon SE \simeq  k[x_i,~i\in I] 
\end{equation}
given by
\begin{equation*}
\phi(\lambda):=\lambda
\end{equation*}
for every $\lambda \in S^0E=k$ and
\begin{equation*}
\phi(e_M):=x^M
\end{equation*}
for every nonempty multiset $M \in \mathcal{M}(I)$. We also have an isomorphism of vector spaces
\begin{equation*}
\psi_1\colon SE \otimes E \simeq (k[x_i,~i \in I])^{(I)}
\end{equation*}
given by 
\begin{equation*}
e_M \otimes e_j \mapsto y_{M,j}
\end{equation*}
where $y_{M,j}$ is defined by
\begin{equation*}
y_{M,j}(i):=\left\{
\begin{aligned}
x^M~&\text{if }i=j, \\
0~&\text{else.}
\end{aligned}
\right.
\end{equation*}
If we define $\partial^*$ as the unique linear map such that the diagram
\begin{equation*}
\begin{tikzcd}
SE \arrow[rr, "\partial"] \arrow[d, "\phi"'] &  & SE \otimes E \arrow[d, "\psi_1"] \\
{k[x_i,~i \in I]} \arrow[rr, "\partial^*"']  &  & {(k[x_i,~i \in I])^{(I)}}       
\end{tikzcd}
\end{equation*}
commutes, then we have
\begin{equation*}
(\partial^* f)(i)=\frac{\partial f}{\partial x_i}
\end{equation*}
for all $f \in k[x_i,~i \in I]$ and $i \in I$.

More generally, for every $r \ge 0$, we have an isomorphism of vector spaces
\begin{equation*}
\psi_r\colon SE \otimes E^{\otimes r} \simeq (k[x_i,~i \in I])^{(I^r)}
\end{equation*}
given by
\begin{equation*}
e_M \otimes e_{j_1} \otimes \dots \otimes e_{j_r} \mapsto y_{M,j_1,\dots,j_r}
\end{equation*}
where $y_{M,j_1,\dots,j_r}$ is defined by
\begin{equation*}
y_{M,j_1,\dots,j_r}(i_1,\dots,i_r):=\left\{
\begin{aligned}
x^M~&\text{if }(i_1,\dots,i_r)=(j_1,\dots,j_r), \\
0~&\text{else.}
\end{aligned}
\right.
\end{equation*}
If we define $\partial^{r,*}$ as the unique linear map such that the diagram
\begin{equation*}
\begin{tikzcd}
SE \arrow[rr, "\partial^r"] \arrow[d, "\phi"']    &  & SE \otimes E^{\otimes r} \arrow[d, "\psi_1"] \\
{k[x_i,~i \in I]} \arrow[rr, "{\partial^{r,*}}"'] &  & {(k[x_i,~i \in I])^{(I^r)}}                 
\end{tikzcd}
\end{equation*}
commutes, then we have
\begin{equation*}
(\partial^{r,*}f)(i_1,\dots,i_r)=\frac{\partial^rf}{\partial x_{i_1}\dots x_{i_r}}.
\end{equation*}

The tensor product of vector spaces preserves kernels,\footnote{
It is well-known \cite{WEIBEL} that for every vector space $A$, the functor $- \otimes A\colon \mathsf{Vec}_k \rightarrow \mathsf{Vec}_k$ is exact, thus in particular left-exact. That is, if $0 \rightarrow E \overset{f}{\rightarrow} F \overset{g}{\rightarrow} G$ is an exact sequence, then $0 \rightarrow E \otimes A \overset{f \otimes 1}{\rightarrow} F \otimes A \overset{g \otimes 1}{\rightarrow} G \otimes A$ is also an exact sequence. But a diagram $X \overset{u}{\rightarrow} Y \overset{v}{\rightarrow} Z$ in $\mathsf{Vec}_k$ is a kernel diagram iff the sequence $0 \rightarrow X \overset{u}{\rightarrow} Y \overset{v}{\rightarrow} Z$ is exact. We thus obtain that the functor $- \otimes A\colon \mathsf{Vec}_k \rightarrow \mathsf{Vec}_k$ sends kernel diagrams to kernel diagrams.} 
thus in particular the kernel
\begin{equation*}
\begin{tikzcd}
\mathrm{ker}(\partial^{n+1}) \arrow[rr, "s_n"] &  & SA \arrow[rr, "\partial^{n+1}"] &  & SA \otimes A^{\otimes (n+1)}
\end{tikzcd}
\end{equation*}
where $s_n$ is the inclusion. It follows that the diagram
\begin{equation*}
\begin{tikzcd}
\mathrm{ker}(\partial^{n+1}) \otimes A \arrow[rr, "s_n"] &  & SA \otimes A \arrow[rr, "\partial^{n+1} \otimes 1_A"] &  & SA \otimes A^{\otimes (n+2)}
\end{tikzcd}
\end{equation*}
is a kernel diagram. We can thus apply \cref{theorem:extracting} to obtain an $\mathbb{N}$-filtered differential modality $(S_{\le r}\colon \mathsf{Vec}_k \rightarrow \mathsf{Vec}_k)_{r \ge 0}$ on $\mathsf{Vec}_k^{\mathrm{op}}$ and a morphism of filtered differential modalities from $(S_{\le r})$ to $S$. 
The $\mathbb{N}$-filtered differential modality and the morphism of filtered differential modalities are constituted of natural transformations of the following types:
\begin{align*}
m_{r,s}\colon S_{\le r}S_{\le s}A \rightarrow S_{\le rs}A \\
u\colon A \rightarrow S_{\le 1}A \\
\nabla_{r,s}\colon S_{\le r}A \otimes S_{\le s}A \rightarrow S_{\le r+s}A \\
\eta\colon k \rightarrow S_{\le 0}A \\
\partial_{\le r}\colon S_{\le r+1}A \rightarrow S_{\le r}A \otimes A \\
s_r\colon  SA \rightarrow S_{\le r}A \\
t_{r,s}\colon S_{\le s}A \rightarrow S_{\le r}A
\end{align*}
where $t_{r,s}$ is defined for every $(r,s) \in \mathbb{N}^2$ with $r \ge s$ and
\begin{equation*}
S_{\le r}A=\mathrm{ker}(\partial^{r+1}\colon SA \rightarrow SA \otimes A^{\otimes r}).
\end{equation*}
In the case where $A=E$, we have
\begin{equation*}
S_{\le r}E \simeq \mathrm{ker}(\partial^{r+1,*}\colon k[x_i,~i \in I] \rightarrow k[x_i,~i \in I]^{(I^{r+1})}).
\end{equation*}
We will be interested in the vector space on the right-hand-side for every $r \ge 0$. We thus define
\begin{equation*}
S_{\le r}^*E:=\mathrm{ker}(\partial^{r+1,*}\colon k[x_i,~i \in I] \rightarrow k[x_i,~i \in I]^{(I^{r+1})}).
\end{equation*}
\subsection{Summary of results}
In the next subsections, we will study $S_{ \le r}A$ depending on the characteristic of the field $k$ and the dimension of the vector space $A$. Recall that $E$ denotes 
a vector space with a basis $(e_i)_{i \in I}$. Whatever the characteristic of $k$ is, we have $\underset{0 \le i \le r}{\bigoplus}S^iA \subseteq S_{\le r}A$ 
(\cref{prop:inclusion}) and if $k$ is of characteristic $0$, we have $S_{\le r}A=\underset{0 \le i \le r}{\bigoplus}S^iA$ (\cref{prop:equality}).

However, if $k$ is of characteristic $p>0$, then $S_{\le r}A$ behaves differently. Suppose that $k$ is of characteristic $p>0$. We then have the following results. 
For all $r \ge 0$ and vector space $A$ of dimension $\ge 1$, we have $S_{\le r}A \neq \underset{0 \le i \le r}{\bigoplus}S^iA$ (\cref{dim-1-diff,dim-n-diff,dim-inf-diff}). 
We have $S_{\le r}A=SA$ for every $r \ge pn-1$ and vector space $A$ of finite dimension $n \ge 1$ (\cref{dim-1-finite,dim-n-finite}). However $S_{\le r}A \neq SA$ for 
all $r \ge 0$ and vector space $A$ of dimension $\infty$ (\cref{dim-inf-inf}). Finally, note that we explicitly compute $S_{\le r}^*E$ for every vector space $E$ with 
a basis made of a single element \cref{comp-dim-1}.
\subsection{Arbitrary characteristic}
In this section, $k$ is an arbitrary field. Let $r \ge 0$ and let $f \in k[x_i,~i \in I]$ be a polynomial of degree $\le r$. Write
\begin{equation*}
f=\underset{0 \le s \le r}{\sum}\,\underset{\alpha \in \mathcal{M}_s(I)}{\sum}a_\alpha x^\alpha.
\end{equation*}
Let $\alpha \in \mathcal{M}_s(I)$ where $0 \le s \le r$. Let $\beta \in \mathcal{M}_{r+1}(I)$. There exists $i \in I$ such that $\beta(i)>\alpha(i)$, else we 
would have $s=\underset{i \in I}{\sum}\alpha(i) \ge \underset{i \in I}{\sum}\beta(i)=r+1$ which is false. We thus have $\frac{\partial^{r+1}x^\alpha}{\partial x^\beta}=0$. 
We conclude that $\partial^{r+1,*}(x^\alpha)=0$. By linearity of $\partial^{r+1,*}$, we obtain that $\partial^{r+1,*}(f)=0$, that is, $f \in S_{\le r}^*(E)$.
\begin{proposition} \label{prop:inclusion}
Let $A$ be a $k$-vector space. We have $\underset{0 \le i \le r}{\bigoplus}S^iA \subseteq S_{\le r}A$.
\end{proposition}
\begin{proof}
Let $(e_i)_{i \in I}$ be a basis of $A$. Then $S_{\le r}A=\phi^{-1}(S_{\le r}^*A)$ where $\phi$ is the isomorphism \cref{iso-phi}. But the space of polynomials of 
degree $\le r$ is a subset of $S^*_{\le r}A$. It follows that the preimage of the space of polynomials of degree $\le r$ under $\phi$ is a subset of $S_{\le r}A$. 
But this preimage is exactly $\underset{0 \le i \le r}{\bigoplus}S^iA$.
\end{proof}
\subsection{Characteristic $0$}
In this subsection, $k$ is a field of characteristic $0$. We will show that in this case, $S_{\le r}E$ behaves as expected. We first note that Taylor expansion 
is valid in $k[x_i,~i \in I]$. If $r \ge 0$ and $\alpha \in \mathcal{M}_r(I)$, and if the only $i \in I$ such that $\alpha(i) \neq 0$ are $i_1,\dots,i_n$, then we define
\begin{equation*}
\frac{\partial^r f}{\partial x^\alpha}:=\frac{\partial^rf}{\partial x_{i_1}^{\alpha(i_1)}\cdots x_{i_n}^{\alpha(i_n)}}.
\end{equation*}
This is well-defined since the expression in the RHS does not depend on the order of $i_1,\dots,i_n$. Note that if 
$[\,]$ is the unique element in $\mathcal{M}_0(I)$, then we define $\frac{\partial^0f}{\partial x^{[\,]}}:=f$.
\begin{proposition} \label{Taylor}
For every $f \in k[x_i,~i \in I]$, we have
\begin{equation*}
f=\underset{r \ge 0}{\sum}\,\underset{\alpha \in \mathcal{M}_r(I)}{\sum}\frac{1}{\alpha !}\frac{\partial^rf}{\partial x^\alpha}(0)x^\alpha.
\end{equation*}
\end{proposition}
\begin{proof}
Write
\begin{equation*}
f=\underset{r \ge 0}{\sum}\,\underset{\beta \in \mathcal{M}_r(I)}{\sum}a_\beta x^\beta.
\end{equation*}
For every $\alpha \in \mathcal{M}(I)$, we have
\begin{equation*}
\frac{\partial^rf}{\partial x^\alpha}=\underset{s \ge r}{\sum}\underset{\substack{\beta \in \mathcal{M}_s(I) \\ \beta \ge \alpha}}{\sum}
a_\beta\beta^{\underline{\alpha}}x^{\beta-\alpha}
\end{equation*}
thus 
\begin{equation*}
\frac{\partial^rf}{\partial x^\alpha}(0)=a_\alpha\alpha^{\underline{\alpha}}=\alpha!a_\alpha
\end{equation*}
from which we get
\begin{equation*}
a_\alpha=\frac{1}{\alpha!}\frac{\partial^rf}{\partial x^\alpha}(0).\qedhere
\end{equation*}
\end{proof}
Let $r \ge 0$. If $\partial^{r+1,*}f=0$ then $\frac{\partial^{r+1}f}{\partial x^\alpha}=0$ for every $\alpha \in \mathcal{M}_{r+1}(I)$, and thus 
$\frac{\partial^sf}{\partial x^\alpha}=0$ for every $\alpha \in \mathcal{M}(I)$ such that $s \ge r+1$. It follows from \cref{Taylor} that $f$ is of degree $\le r$. Combining with \cref{prop:inclusion}, we obtain that for all $r \ge 0$ and $f \in k[x_i,~i \in I]$, we have $\partial^{r+1,*}f=0$ iff $f$ is of degree $\le r$.
\begin{proposition} \label{prop:equality}
Let $A$ be a $k$-vector space. We have $S_{\le r}A=\underset{0 \le i \le r}{\bigoplus}S^iA$.
\end{proposition}
\begin{proof}
Let $(e_i)_{i \in I}$ be a basis of $A$. Then $S_{\le r}A$ is the inverse image of the space of polynomials of degree $\le r$ in $k[x_i,~i \in I]$ under the isomorphism \cref{iso-phi}. But this inverse image is exactly $\underset{0 \le i \le r}{\bigoplus}S^iA$.
\end{proof}
\subsection{Characteristic $p$ - dimension $1$}
We will suppose from now on and until the end of this section that $k$ is a field of characteristic $p>0$ and show that $S_{\le r}^*E$ does not behave as expected. 
We suppose first that $E$ is of dimension $1$. We thus have $|I|=1$ and
\begin{equation*}
S_{\le r}^*E=\mathrm{ker}(\partial^{r+1,*}\colon k[x] \rightarrow k[x]).
\end{equation*}
where
\begin{equation*}
\partial^{r+1,*}f=f^{(r+1)}.
\end{equation*}
Let $f \in k[x]$ and write
\begin{equation*}
f=\underset{n \ge 0}{\sum}a_nx^n.
\end{equation*}
We have
\begin{equation*}
f^{(r+1)}=\underset{n \ge r+1}{\sum}a_nn(n-1)\dots(n-r)x^{n-(r+1)}.
\end{equation*}
It follows that
\begin{align*}
f^{(r+1)}=0 &\Leftrightarrow a_nn(n-1)\dots(n-r)=0\quad\forall n \ge r+1 \\
&\Leftrightarrow \Big(a_n=0 \text{ or }n(n-1)\dots(n-r)=0 \text{ in }k\Big)\quad\forall n \ge r+1 \\
&\Leftrightarrow \Big(a_n=0 \text{ or } n \in p\mathbb{N} \cup p\mathbb{N}+1 \cup \dots \cup p\mathbb{N}+r\Big)\quad \forall n \ge r+1 \\
&\Leftrightarrow \Big(a_n \neq 0 \Rightarrow n \in p\mathbb{N} \cup p\mathbb{N}+1 \cup \dots \cup p\mathbb{N}+r\Big) \quad\forall n \ge r+1 \\
&\Leftrightarrow f \in \mathrm{span}\Big(\{x^k~,0 \le k \le r\} \cup \underset{q \ge 0}{\bigcup}\{x^{pq},x^{pq+1},\dots,x^{pq+r}\}\Big).
\end{align*}
We conclude that
\begin{equation} \label{comp-dim-1}
S_{\le r}^*E=\mathrm{span}\Big(\{x^k,~0 \le k \le r\} \cup \underset{q \ge 0}{\bigcup}\{x^{pq},x^{pq+1},\dots,x^{pq+r}\}\Big).
\end{equation}

We thus see that $S^*_{\le r}E$ is not equal to the space of polynomials of degree $\le r$. In particular, it contains polynomials of arbitrary large degree for every 
$r \ge 0$ and is thus infinite-dimensional. We also see that $S_{\le r}^*E=k[x]$ for every $r \ge p-1$. We can conclude the following.
\begin{proposition} \label{dim-1-diff}
Let $A$ be a $k$-vector space of dimension $1$. For every $r \ge 0$, $S_{\le r}A \neq \underset{0 \le k \le r}{\bigoplus}S^kA$.
\end{proposition}
\begin{proof}
Choose a basis $(e_i)_{i \in I}$ of $A$. If we had $S_{\le r}A=\underset{0 \le k \le r}{\bigoplus}S^kA$, then we would have $S_{\le r}^*A=\mathrm{span}\{x^k,~0 \le k \le r\}$ 
which is false.
\end{proof}
\begin{proposition} \label{dim-1-finite}
Let $A$ be a $k$-vector space of dimension $1$. For every $r \ge p-1$, we have $S_{\le r}A=SA$.
\end{proposition}
\begin{proof}
Choose a basis $(e_i)_{i \in I}$ of $A$. We have $S_{\le r}^*A=k[x]$, thus $S_{\le r}A=SA$.
\end{proof}

We will see in the next subsection that these propositions extend to any finite dimension $n \ge 1$.
\subsection{Characteristic $p$ - finite dimension $n \ge 1$}
We suppose now that $E$ is of finite dimension $n \ge 1$. We thus have $|I|=n$. Hence $I \simeq \{1,\dots,n\}$. Up to a choice of a bijection, we can identity $I$ 
to $\{1,\dots,n\}$. We have
\begin{equation*}
S_{\le r}^*E=\mathrm{ker}(\partial^{r+1,*}\colon k[x_1,\dots,x_n] \rightarrow k[x_1,\dots,x_n]^{(\{1,\dots,n\}^{r+1})})
\end{equation*}
where
\begin{equation*}
(\partial^{r+1,*}f)(i_1,\dots,i_{r+1})=\frac{\partial^{r+1}f}{\partial x_{i_1}\dots\partial x_{i_{r+1}}}.
\end{equation*}
Let $r \ge pn-1$. We will prove that $\partial^{r+1,*}x_M=0$ for every $M \in \mathcal{M}(\{1,\dots,n\})$ which will imply that $S_{\le r}^*E=SE$.

Let $M \in \mathcal{M}(\{1,\dots,n\})$. Let $(i_1,\dots,i_{r+1}) \in \{1,\dots,n\}^{r+1}$. For every $s \in \{1,\dots,n\}$, write
\begin{equation*}
\alpha_s=|\{k \in \{1,\dots,r+1\}~\text{s.t.}~i_k=s\}|.
\end{equation*}
We have
\begin{equation*}
\alpha_1+\dots+\alpha_n=r+1 \ge pn.
\end{equation*}
Thus, there exists $t \in \{1,\dots,n\}$ such that $\alpha_t \ge p$. Else, we would have
\begin{equation*}
r+1=\alpha_1+\dots+\alpha_n<pn
\end{equation*}
which is false. We have
\begin{equation*}
(\partial^{r+1,*}x^M)(i_1,\dots,i_{r+1})=\frac{\partial^{r+1}x^M}{\partial x_{i_1}\dots\partial x_{i_{r+1}}}=
\frac{\partial^{r+1}(x_1^{M(1)}\dots x_n^{M(n)})}{\partial x_1^{\alpha_1}\dots x_n^{\alpha_n}}.
\end{equation*}
This polynomial is equal to $0$ if $\alpha_t>M(t)$. Else it is of the form
\begin{equation*}
M(t)(M(t)-1)\dots(M(t)-(\alpha_t-1))x_t^{M(t)-\alpha_t}g
\end{equation*}
for some $g \in k[x_1,\dots,x_n]$. But any element of the form $a(a-1)\dots(a-(b-1))$ where $a \in \mathbb{N}$ and $b \ge p$ is equal to $0$ in a field 
of characteristic $p$ since at least one of $a$, $a-1$, \dots, $a-(b-1)$ is divisible by $p$. Thus $(\partial^{r+1,*}x_M)(i_1,\dots,i_{r+1})$ is also 
equal to $0$ if $\alpha_t \le M(t)$. We conclude that $(\partial^{r+1,*}x_M)(i_1,\dots,i_{r+1})=0$ for any $(i_1,\dots,i_{r+1}) \in \{1,\dots,n\}^{r+1}$, 
thus $\partial^{r+1,*}x_M=0$.

We obtain the following.
\begin{proposition} \label{dim-n-finite}
Let $A$ be a $k$-vector space of finite dimension $n \ge 1$. For every $r \ge pn-1$, we have $S_{\le r}A=SA$.
\end{proposition}
\begin{proof}
Choose a basis $(e_1,\dots,e_n)$ of $A$. We have $S^*_{\le r}A=k[x_1,\dots,x_n]$, thus $S_{\le r}A=SA$.
\end{proof}
We will not compute a basis for $S_{\le r}A$ as in the case of dimension $1$. But we still have the following.
\begin{proposition} \label{dim-n-diff}
Let $A$ be a $k$-vector space of finite dimension $n \ge 1$. For every $r \ge 0$, we have $S_{\le r}A \neq \underset{0 \le k \le r}{\bigoplus}S^kA$.
\end{proposition}
\begin{proof}
Choose a basis $(e_1,\dots,e_n)$ of $A$. We have $S^*_{\le r}A=k[x_1,\dots,x_n]$. For every $i \ge 0$, we have $x_1^{pi} \in S^*_{\le r}A$. Indeed 
$\partial^*(x_1^{0i})(1)=\partial^*(1)(1)=\frac{\partial 1}{\partial x_1}=0$ and for every $i \ge 1$, $\partial^*(x_1^{pi})(1)=pix_1^{pi-1}=0$. 
Moreover for every $i \ge 0$, we have $\partial^*(x_1^{pi})(j)=\frac{\partial x_1^{pi}}{\partial x_j}=0$ for every $j \in \{2,\dots,n\}$. 
Thus $\partial^*(x_1^{pi})=0$ for every $i \ge 0$. We thus have $x_1^{pi} \in \mathrm{ker}(\partial^*)=S_{\le 0}^*A \subseteq \mathrm{ker}(\partial^{r+1,*}) = S_{\le r}^*A$ 
for every $r \ge 0$. Thus $S_{\le r}^*A$ is of dimension $\infty$, hence $S_{\le r}A$ is of dimension $\infty$. It follows that 
$S_{\le r}A \neq \underset{0 \le k \le r}{\bigoplus}S^kA$ since the RHS is of finite dimension.
\end{proof}
\subsection{Characteristic $p$ - dimension $\infty$}
We suppose now that $E$ is of dimension $\infty$. We thus have $|I|=\infty$. It follows that there exists an injection $\iota\colon \mathbb{N} \rightarrow I$. Recall that
\begin{equation*}
S_{\le r}^*E=\mathrm{ker}(\partial^{r+1,*}\colon k[x_i,~i \in I] \rightarrow k[x_i,~i \in I]^{(I^{r+1})}).
\end{equation*}
Let $r \ge 0$ and consider the polynomial
\begin{equation*}
x_{\iota(1)}\dots x_{\iota(r+1)} \in k[x_i,~i \in I].
\end{equation*}
We have
\begin{equation*}
(\partial^{r+1,*}x_{\iota(1)}\dots x_{\iota(r+1)})(\iota(1),\dots,\iota(r+1))=\frac{\partial^{r+1} x_{\iota(1)}
\dots x_{\iota(r+1)}}{\partial x_{\iota(1)}\dots x_{\iota(r+1)}}=1 \neq 0.
\end{equation*}
It follows that 
\begin{equation*}
\partial^{r+1,*}x_{\iota(1)}\dots x_{\iota(r+1)} \neq 0.
\end{equation*}
Thus $S^*_{\le r}E \neq k[x_i,~i \in I]$ from which we obtain that $S_{\le r}E \neq SE$. By choosing a basis of $A$, we obtain the next proposition.
\begin{proposition} \label{dim-inf-inf}
Let $A$ be a $k$-vector space of dimension $\infty$. For every $r \ge 0$, we have $S_{\le r}A \neq SA$.
\end{proposition}
But we still have the following.
\begin{proposition} \label{dim-inf-diff}
Let $A$ be a $k$-vector space of dimension $\infty$. For every $r \ge 0$, we have $S_{\le r}A \neq \underset{0 \le k \le r}{\bigoplus}S^kA$.
\end{proposition}
\begin{proof}
Choose a basis $(e_i)_{i \in I}$ of $A$ and let $\iota\colon\mathbb{N} \rightarrow I$ be an injection. As for the case of finite dimension, we have 
$x_{\iota(1)}^{pi} \in S_{\le r}^*A$ for every $i \ge 0$. It follows that $S_{\le r}A$ contains nonzero elements from $S^{pi}A$ for every $i \ge 0$. 
Thus $S_{\le r}A \neq \underset{0 \le k \le r}{\bigoplus}S^kA$.
\end{proof}
\bibliographystyle{plainurl}
\bibliography{biblio-extracting}
\appendix
\section{Deferred proofs} \label{sec:appendix}
\paragraph{Proof of \cref{id3}:} \label{proof:higer-order-product} The proof is by induction on $n \ge 0$.
\begin{itemize}
\item Base case ($n=0$): We have $\partial^0;\Delta=1_{\oc A};\Delta=\Delta$ and the RHS is equal to $(\Delta \otimes 1_I);(1_{\oc A} \otimes \gamma_{\oc A, I} 
\otimes 1_I);\partial^0 \otimes \partial^0=\Delta;\partial^0 \otimes \partial^0=\Delta;1_{\oc A} \otimes 1_{\oc A}=\Delta$, hence the equality.
\item Induction step: Suppose that the identity holds for some $n \ge 0$. We show that it holds for $n+1$. We compute:
\begin{align*}
&\partial^{n+1};\Delta \\
&=\partial^n \otimes 1_{A};\partial;\Delta &\text{\scriptsize{\cref{def-d-ind-2}}} \\
&=\partial^n \otimes 1_{A};(\Delta \otimes 1_A;1_{\oc A} \otimes \partial + \Delta \otimes 1_A;1_{\oc A} \otimes \gamma_{\oc A,A};
\partial \otimes 1_{\oc A}) &\text{\scriptsize{\cref{def-d2}}} \\
&=\partial^n \otimes 1_{A};\Delta \otimes 1_A;1_{\oc A} \otimes \partial + \partial^n \otimes 1_{A};\Delta \otimes 1_A;1_{\oc A} \otimes 
\gamma_{\oc A,A};\partial \otimes 1_{\oc A} &\text{\scriptsize{\cref{add-4}}} \\
&=(\partial^n;\Delta) \otimes 1_{A};1_{\oc A} \otimes \partial + (\partial^n;\Delta) \otimes 1_{A};1_{\oc A} \otimes \gamma_{\oc A,A};\partial 
\otimes 1_{\oc A} &\text{\scriptsize{func.~of $- \otimes 1_A$}} \\
&=\Big(\underset{0 \le k \le n}{\sum}\Delta \otimes \mathsf{unsh}(n,k)_A;1_{\oc A} \otimes \gamma_{\oc A, A^{\otimes k}} \otimes 1_{A^{\otimes(n-k)}};\partial^k 
\otimes \partial^{n-k}\Big) \otimes 1_{A};1_{\oc A} \otimes \partial &\text{\scriptsize{induction hyp.}} \\
&~+\Big(\underset{0 \le k \le n}{\sum}\Delta \otimes \mathsf{unsh}(n,k)_A;1_{\oc A} \otimes \gamma_{\oc A, A^{\otimes k}} \otimes 1_{A^{\otimes(n-k)}};~\partial^k 
\otimes \partial^{n-k}\Big) \otimes 1_{A}; \\
&~1_{\oc A} \otimes \gamma_{\oc A,A};\partial \otimes 1_{\oc A} \\
&=\underset{0 \le k \le n}{\sum}\Delta \otimes \mathsf{unsh}(n,k)_A \otimes 1_A;1_{\oc A} \otimes \gamma_{\oc A, A^{\otimes k}} \otimes 1_{A^{\otimes(n+1-k)}};\partial^k 
\otimes \partial^{n-k} \otimes 1_A;1_{\oc A} \otimes \partial &\text{\scriptsize{func.~of $- \otimes 1_A$}} \\
&~+\underset{0 \le k \le n}{\sum}\Delta \otimes \mathsf{unsh}(n,k)_A \otimes 1_A;1_{\oc A} \otimes \gamma_{\oc A, A^{\otimes k}} \otimes 1_{A^{\otimes(n+1-k)}};\partial^k 
\otimes \partial^{n-k} \otimes 1_A; \\
&~1_{\oc A} \otimes \gamma_{\oc A,A};\partial \otimes 1_{\oc A} \\
&=\underset{0 \le k \le n}{\sum}\Delta \otimes \mathsf{unsh}(n,k)_A \otimes 1_A;1_{\oc A} \otimes \gamma_{\oc A, A^{\otimes k}} \otimes 1_{A^{\otimes(n+1-k)}};\partial^k 
\otimes (\partial^{n-k} \otimes 1_A;\partial) &\text{\scriptsize{func.~of $- \otimes -$}} \\
&~+\underset{0 \le k \le n}{\sum}\Delta \otimes \mathsf{unsh}(n,k)_A \otimes 1_A;1_{\oc A} \otimes \gamma_{\oc A, A^{\otimes k}} \otimes 1_{A^{\otimes(n+1-k)}}; \\
&~\partial^k \otimes (\partial^{n-k} \otimes 1_A;\gamma_{\oc A,A});\partial \otimes 1_{\oc A}  \\
&=\underset{0 \le k \le n}{\sum}\Delta \otimes \mathsf{unsh}(n,k)_A \otimes 1_A;1_{\oc A} \otimes \gamma_{\oc A, A^{\otimes k}} \otimes 1_{A^{\otimes(n+1-k)}};\partial^k 
\otimes \partial^{n+1-k} &\text{\scriptsize{\cref{def-d-ind-2}}} \\
&~+\underset{0 \le k \le n}{\sum}\Delta \otimes \mathsf{unsh}(n,k)_A \otimes 1_A;1_{\oc A} \otimes \gamma_{\oc A, A^{\otimes k}} \otimes 1_{A^{\otimes(n+1-k)}}; \\
&~1_{\oc A \otimes A^{\otimes k}} \otimes (\partial^{n-k} \otimes 1_A;\gamma_{\oc A,A});(\partial^k \otimes 1_{A \otimes \oc A});(\partial \otimes 1_{\oc A}) 
&\text{\scriptsize{func.~of $-\otimes -$}} \\
&=\underset{0 \le k \le n}{\sum}\Delta \otimes \mathsf{unsh}(n,k)_A \otimes 1_A;1_{\oc A} \otimes \gamma_{\oc A, A^{\otimes k}} \otimes 1_{A^{\otimes(n+1-k)}};\partial^k 
\otimes \partial^{n+1-k}\\
&~+ \underset{0 \le k \le n}{\sum}\Delta \otimes \mathsf{unsh}(n,k)_A \otimes 1_A;1_{\oc A} \otimes \gamma_{\oc A, A^{\otimes k}} \otimes 1_{A^{\otimes(n+1-k)}};\\
&~1_{\oc A \otimes A^{\otimes k}} \otimes (\partial^{n-k} \otimes 1_A;\gamma_{\oc A,A});(\partial^k \otimes 1_{A} \otimes 1_{\oc A});(\partial \otimes 1_{\oc A}) 
&\text{\scriptsize{func.~of $-\otimes -$}} \\
&=\underset{0 \le k \le n}{\sum}\Delta \otimes \mathsf{unsh}(n,k)_A \otimes 1_A;1_{\oc A} \otimes \gamma_{\oc A, A^{\otimes k}} \otimes 1_{A^{\otimes(n+1-k)}};\partial^k 
\otimes \partial^{n+1-k}\\
&~+ \underset{0 \le k \le n}{\sum}\Delta \otimes \mathsf{unsh}(n,k)_A \otimes 1_A;1_{\oc A} \otimes \gamma_{\oc A, A^{\otimes k}} \otimes 1_{A^{\otimes(n+1-k)}};\\
&~1_{\oc A \otimes A^{\otimes k}} \otimes (\partial^{n-k} \otimes 1_A;\gamma_{\oc A,A});\partial^{k+1} \otimes 1_{\oc A} &\text{\scriptsize{\cref{def-d-ind-2}}} \\
&=\underset{0 \le k \le n}{\sum}\Delta \otimes \mathsf{unsh}(n,k)_A \otimes 1_A;1_{\oc A} \otimes \gamma_{\oc A, A^{\otimes k}} \otimes 1_{A^{\otimes(n+1-k)}};\partial^k 
\otimes \partial^{n+1-k}\\
&~+ \underset{0 \le k \le n}{\sum}\Delta \otimes \mathsf{unsh}(n,k)_A \otimes 1_A;1_{\oc A} \otimes \gamma_{\oc A, A^{\otimes k}} \otimes 1_{A^{\otimes(n+1-k)}};\\
&~1_{\oc A \otimes A^{\otimes k}} \otimes (\gamma_{\oc A \otimes A^{\otimes(n-k)}, A};1_A \otimes \partial^{n-k});\partial^{k+1} \otimes 1_{\oc A} 
&\text{\scriptsize{nat.~of $\gamma$}} \\
&=\underset{0 \le k \le n}{\sum}\Delta \otimes \mathsf{unsh}(n,k)_A \otimes 1_A;1_{\oc A} \otimes \gamma_{\oc A, A^{\otimes k}} \otimes 1_{A^{\otimes(n+1-k)}};\partial^k 
\otimes \partial^{n+1-k} \\
&~+\underset{0 \le k \le n}{\sum}\Delta \otimes \mathsf{unsh}(n,k)_A \otimes 1_A;1_{\oc A} \otimes \gamma_{\oc A, A^{\otimes k}} \otimes 1_{A^{\otimes(n+1-k)}};\\
&~1_{\oc A \otimes A^{\otimes k}} \otimes \gamma_{\oc A \otimes A^{\otimes(n-k)}, A};1_{\oc A \otimes A^{\otimes (k+1)}}\otimes \partial^{n-k};\partial^{k+1} 
\otimes 1_{\oc A} &\text{\scriptsize{func.~of $1_{\oc A \otimes A^{\otimes k}}$}} \\
&=\underset{0 \le k \le n}{\sum}\Delta \otimes \mathsf{unsh}(n,k)_A \otimes 1_A;1_{\oc A} \otimes \gamma_{\oc A, A^{\otimes k}} \otimes 1_{A^{\otimes(n+1-k)}};\partial^k 
\otimes \partial^{n+1-k} \\
&~+\underset{0 \le k \le n}{\sum}\Delta \otimes \mathsf{unsh}(n,k)_A \otimes 1_A;1_{\oc A} \otimes \gamma_{\oc A, A^{\otimes k}} \otimes 1_{A^{\otimes(n+1-k)}}; \\
&~1_{\oc A \otimes A^{\otimes k}} \otimes \gamma_{\oc A \otimes A^{\otimes(n-k)}, A};\partial^{k+1} \otimes \partial^{n-k}. &\text{\scriptsize{func.~of $-\otimes -$}} \\
&=\underset{0 \le k \le n}{\sum}\Delta \otimes \mathsf{unsh}(n,k)_A \otimes 1_A;1_{\oc A} \otimes \gamma_{\oc A, A^{\otimes k}} \otimes 1_{A^{\otimes(n+1-k)}};\partial^k 
\otimes \partial^{n+1-k} \\
&~+\underset{0 \le k \le n}{\sum}\Delta \otimes \mathsf{unsh}(n,k)_A \otimes 1_A;1_{\oc A \otimes \oc A} \otimes 1_{A^{\otimes k}} \otimes \gamma_{A^{\otimes (n-k)},A}; 
&\text{\scriptsize{\cref{Mac-0}}} \\
&~1_{\oc A} \otimes \gamma_{\oc A, A^{\otimes (k+1)}} \otimes 1_{A^{\otimes (n-k)}};\partial^{k+1} \otimes \partial^{n-k} \\
&=\Delta \otimes 1_{A^{\otimes (n+1)}};1_{\oc A} \otimes \partial^{n+1} &\text{\scriptsize{$\mathsf{unsh}(n,0)_A=1_{A^{\otimes n}}$}} \\
&~+\underset{1 \le k \le n}{\sum}\Delta \otimes \mathsf{unsh}(n,k)_A \otimes 1_A;1_{\oc A} \otimes \gamma_{\oc A, A^{\otimes k}} \otimes 1_{A^{\otimes(n+1-k)}};\partial^{k} 
\otimes \partial^{n+1-k} \\
&~+\underset{0 \le k \le n-1}{\sum}\Delta \otimes \mathsf{unsh}(n,k)_A \otimes 1_A;1_{\oc A \otimes \oc A} \otimes 1_{A^{\otimes k}} \otimes \gamma_{A^{\otimes (n-k)},A} \\
&~1_{\oc A} \otimes \gamma_{\oc A, A^{\otimes (k+1)}} \otimes 1_{A^{\otimes (n-k)}};\partial^{k+1} \otimes \partial^{n-k} \\
&~+\Delta \otimes 1_{A^{\otimes (n+1)}};1_{\oc A} \otimes \gamma_{\oc A, A^{\otimes (n+1)}};\partial^{n+1} \otimes 1_{\oc A} &\text{\scriptsize{$\mathsf{unsh}(n,n)_A
=1_{A^{\otimes n}}$}} \\
&=\Delta \otimes 1_{A^{\otimes (n+1)}};1_{\oc A} \otimes \partial^{n+1} \\
&~+\underset{0 \le k \le n-1}{\sum}\Delta \otimes \mathsf{unsh}(n,k+1)_A \otimes 1_A;1_{\oc A} \otimes \gamma_{\oc A, A^{\otimes (k+1)}} \otimes 1_{A^{\otimes (n-k)}};
\partial^{k+1} \otimes \partial^{n-k} &\text{\scriptsize{reindexing}} \\
&~+\underset{0 \le k \le n-1}{\sum}\Delta \otimes \mathsf{unsh}(n,k)_A \otimes 1_A;1_{\oc A \otimes \oc A} \otimes 1_{A^{\otimes k}} \otimes \gamma_{A^{\otimes (n-k)},A}; \\
&~1_{\oc A} \otimes \gamma_{\oc A, A^{\otimes (k+1)}} \otimes 1_{A^{\otimes (n-k)}};\partial^{k+1} \otimes \partial^{n-k} \\
&~+\Delta \otimes 1_{A^{\otimes (n+1)}};1_{\oc A} \otimes \gamma_{\oc A, A^{\otimes (n+1)}};\partial^{n+1} \otimes 1_{\oc A} \\
&=\Delta \otimes 1_{A^{\otimes (n+1)}};1_{\oc A} \otimes \partial^{n+1} \\
&~+\underset{0 \le k \le n-1}{\sum}\Delta \otimes \mathsf{unsh}(n,k+1)_A \otimes 1_A;1_{\oc A} \otimes \gamma_{\oc A, A^{\otimes (k+1)}} \otimes 1_{A^{\otimes (n-k)}};
\partial^{k+1} \otimes \partial^{n-k} \\
&~+\underset{0 \le k \le n-1}{\sum}(\Delta;1_{\oc A \otimes \oc A}) \otimes  (\mathsf{unsh}(n,k)_A \otimes 1_A;1_{A^{\otimes k}} \otimes \gamma_{A^{\otimes (n-k)},A}); 
&\text{\scriptsize{func.~of $-\otimes -$}} \\
&~1_{\oc A} \otimes \gamma_{\oc A, A^{\otimes (k+1)}} \otimes 1_{A^{\otimes (n-k)}};\partial^{k+1} \otimes \partial^{n-k} \\
&~+\Delta \otimes 1_{A^{\otimes (n+1)}};1_{\oc A} \otimes \gamma_{\oc A, A^{\otimes (n+1)}};\partial^{n+1} \otimes 1_{\oc A} \\
&=\Delta \otimes 1_{A^{\otimes (n+1)}};1_{\oc A} \otimes \partial^{n+1} \\
&~+\underset{0 \le k \le n-1}{\sum}\Delta \otimes \mathsf{unsh}(n,k+1)_A \otimes 1_A;1_{\oc A} \otimes \gamma_{\oc A, A^{\otimes (k+1)}} \otimes 1_{A^{\otimes (n-k)}};
\partial^{k+1} \otimes \partial^{n-k} \\
&~+\underset{0 \le k \le n-1}{\sum}\Delta \otimes (\mathsf{unsh}(n,k)_A \otimes 1_A;1_{A^{\otimes k}} \otimes \gamma_{A^{\otimes (n-k)},A}); &\text{\scriptsize{$\Delta;
1_{\oc A \otimes \oc A}=\Delta$}} \\
&~1_{\oc A} \otimes \gamma_{\oc A, A^{\otimes (k+1)}} \otimes 1_{A^{\otimes (n-k)}};\partial^{k+1} \otimes \partial^{n-k} \\
&~+\Delta \otimes 1_{A^{\otimes (n+1)}};1_{\oc A} \otimes \gamma_{\oc A, A^{\otimes (n+1)}};\partial^{n+1} \otimes 1_{\oc A} \\
&=\Delta \otimes 1_{A^{\otimes (n+1)}};1_{\oc A} \otimes \partial^{n+1} \\
&~+\underset{0 \le k \le n-1}{\sum} \Delta \otimes (\mathsf{unsh}(n,k+1)_A \otimes 1_A+\mathsf{unsh}(n,k)_A \otimes 1_A;1_{A^{\otimes k}} \otimes \gamma_{A^{\otimes(n-k)},A}); &\text{\scriptsize{\cref{add-8,add-3}}} \\
&~1_{\oc A} \otimes \gamma_{\oc A, A^{\otimes (k+1)}} \otimes 1_{A^{\otimes (n-k)}};\partial^{k+1} \otimes \partial^{n-k} \\
&~+\Delta \otimes 1_{A^{\otimes (n+1)}};1_{\oc A} \otimes \gamma_{\oc A, A^{\otimes (n+1)}};\partial^{n+1} \otimes 1_{\oc A} \\
&=\Delta \otimes 1_{A^{\otimes (n+1)}};1_{\oc A} \otimes \partial^{n+1} \\
&~+\underset{0 \le k \le n-1}{\sum}\Delta \otimes \mathsf{unsh}(n+1,k+1)_A;1_{\oc A} \otimes \gamma_{\oc A, A^{\otimes (k+1)}} \otimes 1_{A^{\otimes (n-k)}};
\partial^{k+1} \otimes \partial^{n-k} &\text{\scriptsize{\cref{ind-lemma}}} \\
&~+\Delta \otimes 1_{A^{\otimes (n+1)}};1_{\oc A} \otimes \gamma_{\oc A, A^{\otimes (n+1)}};\partial^{n+1} \otimes 1_{\oc A} \\
&=\Delta \otimes 1_{A^{\otimes (n+1)}};1_{\oc A} \otimes \partial^{n+1} \\
&~+\underset{0 \le k \le n}{\sum}\Delta \otimes \mathsf{unsh}(n+1,k+1)_A;1_{\oc A} \otimes \gamma_{\oc A, A^{\otimes (k+1)}} \otimes 1_{A^{\otimes (n-k)}};
\partial^{k+1} \otimes \partial^{n-k}&\text{\scriptsize{$\mathsf{unsh}(n+1,n+1)_A=1_{A^{\otimes (n+1)}}$}} \\
&=\Delta \otimes 1_{A^{\otimes (n+1)}};1_{\oc A} \otimes \partial^{n+1} \\
&~+\underset{1 \le k \le n+1}{\sum}\Delta \otimes \mathsf{unsh}(n+1,k)_A;1_{\oc A} \otimes \gamma_{\oc A, A^{\otimes k}} \otimes 1_{A^{\otimes (n+1-k)}};
\partial^k \otimes \partial^{n+1-k}&\text{\scriptsize{reindexing}} \\
&=\underset{0 \le k \le n+1}{\sum}\Delta \otimes \mathsf{unsh}(n+1,k)_A;1_{\oc A} \otimes \gamma_{\oc A, A^{\otimes k}} \otimes ~1_{A^{\otimes (n+1-k)}};
\partial^k \otimes \partial^{n+1-k}. &\text{\scriptsize{$1_{A^{\otimes (n+1)}}=\mathsf{unsh}(n+1,0)_A$}} \\
\end{align*}
\begin{remark} \label{Mac-0}
Mac Lane's coherence theorem for symmetric monoidal categories implies that for all $X,Y,Z,W \in \mathsf{C}$, we have:
\begin{equation*}
\gamma_{X,Y} \otimes 1_{Z \otimes W};1_Y \otimes \gamma_{X \otimes Z,W}=1_{X \otimes Y} \otimes \gamma_{Z,W};\gamma_{X,Y \otimes W} \otimes 1_Z\colon X \otimes Y \otimes Z \otimes W \rightarrow Y \otimes W \otimes X \otimes Z.
\end{equation*}
Choosing $X=\oc A$, $Y=A^{\otimes k}$, $Z=A^{\otimes (n-k)}$ and $W=A$, we obtain 
\begin{equation*}
\gamma_{\oc A,A^{\otimes k}} \otimes 1_{A^{\otimes (n+1-k)}};1_{A^{\otimes k}} \otimes \gamma_{\oc A \otimes A^{\otimes (n-k)},A}=1_{\oc A \otimes A^{\otimes k}} \otimes \gamma_{A^{\otimes (n-k)},A};\gamma_{\oc A,A^{\otimes (k+1)}} \otimes 1_{A^{\otimes (n-k)}}.
\end{equation*}
\end{remark}
\end{itemize}
\paragraph{Proof of \cref{id4}:} \label{proof:higer-order-product-2} The proof is by induction on $n \ge 1$.
\begin{itemize}
\item Base case ($n=1$): The LHS is equal to $\partial;\Delta^1=\partial;1_{\oc A}=\partial$ and the RHS is equal to $\Delta^1 \otimes 1_A;1_{\oc A} \otimes \gamma_{I,A};1_I \otimes \partial \otimes 1_I=1_{\oc A} \otimes 1_A;1_{\oc A} \otimes 1_A;\partial=\partial$, hence the equality.
\item Induction step: Suppose that the identity holds for some $n \ge 1$. We show that the identity holds for $n+1$. We compute:
\begin{align*}
&\partial;\Delta^{n+1} \\
&=\partial;\Delta^n;1_{(\oc A)^{\otimes(n-1)}} \otimes \Delta &\text{\scriptsize{\cref{Delta-rec}}} \\
&=\Big(\underset{1 \le i \le n}{\sum}\Delta^n \otimes 1_A;1_{(\oc A)^{\otimes i}} \otimes \gamma_{(\oc A)^{\otimes(n-i)},A};1_{(\oc A)^{\otimes(i-1)}} 
\otimes \partial \otimes 1_{(\oc A)^{\otimes(n-i)}}\Big);1_{(\oc A)^{\otimes(n-1)}} \otimes \Delta &\text{\scriptsize{induction~hyp.~}} \\
&=\underset{1 \le i \le n}{\sum}\Delta^n \otimes 1_A;1_{(\oc A)^{\otimes i}} \otimes \gamma_{(\oc A)^{\otimes(n-i)},A};
1_{(\oc A)^{\otimes(i-1)}} \otimes \partial \otimes 1_{(\oc A)^{\otimes(n-i)}};1_{(\oc A)^{\otimes (n-1)}} \otimes \Delta &\text{\scriptsize{\cref{add-3}}} \\
&=\underset{1 \le i \le n-1}{\sum}\Delta^n \otimes 1_A;1_{(\oc A)^{\otimes i}} \otimes \gamma_{(\oc A)^{\otimes(n-i)},A};
1_{(\oc A)^{\otimes(i-1)}} \otimes \partial \otimes 1_{(\oc A)^{\otimes(n-i)}};1_{(\oc A)^{\otimes (n-1)}} \otimes \Delta \\
&~+\Delta^n \otimes 1_A;1_{(\oc A)^{\otimes n}} \otimes \gamma_{I,A};1_{(\oc A)^{\otimes (n-1)}} \otimes \partial \otimes 1_I;1_{(\oc A)^{\otimes(n-1)}} \otimes \Delta \\
&=\underset{1 \le i \le n-1}{\sum}\Delta^n \otimes 1_A;1_{(\oc A)^{\otimes i}} \otimes \gamma_{(\oc A)^{\otimes(n-i)},A};
1_{(\oc A)^{\otimes(i-1)}} \otimes \partial \otimes  1_{(\oc A)^{\otimes((n-1)-i)}} \otimes 1_{\oc A}; &\text{\scriptsize{func.~of $-\otimes-$}} \\
&~1_{(\oc A)^{\otimes (n-1)}} \otimes \Delta+\Delta^n \otimes 1_A;1_{(\oc A)^{\otimes (n-1)}} \otimes \partial;1_{(\oc A)^{\otimes(n-1)}} \otimes \Delta 
&\text{\scriptsize{$\gamma_{I,A}=1_A$ and $\partial \otimes 1_I=\partial$}} \\
&=\underset{1 \le i \le n-1}{\sum}\Delta^n \otimes 1_A;1_{(\oc A)^{\otimes i}} \otimes \gamma_{(\oc A)^{\otimes(n-i)},A};
1_{(\oc A)^{\otimes(i-1)}} \otimes \partial \otimes  1_{(\oc A)^{\otimes((n-1)-i)}} \otimes \Delta &\text{\scriptsize{func.~of $-\otimes-$}} \\
&~+\Delta^n \otimes 1_A;1_{(\oc A)^{\otimes (n-1)}} \otimes (\partial;\Delta) &\text{\scriptsize{func.~of $-\otimes-$}} \\
&=\underset{1 \le i \le n-1}{\sum}\Delta^n \otimes 1_A;1_{(\oc A)^{\otimes i}} \otimes \gamma_{(\oc A)^{\otimes(n-i)},A};
1_{(\oc A)^{\otimes (i-1)}} \otimes 1_{\oc A \otimes A} \otimes 1_{(\oc A)^{\otimes ((n-1)-i)}} \otimes \Delta; \\
&~1_{(\oc A)^{\otimes (i-1)}} \otimes \partial \otimes 1_{(\oc A)^{\otimes ((n-1)-i)}} \otimes 1_{\oc A \otimes \oc A} &\text{\scriptsize{func.~of $-\otimes-$}} \\
&~+\Delta^n \otimes 1_A;1_{(\oc A)^{\otimes (n-1)}} \otimes (\Delta \otimes 1_A;1_{\oc A} \otimes \partial+\Delta \otimes 1_A;1_{\oc A} \otimes \gamma_{\oc A,A};
\partial \otimes 1_{\oc A}) &\text{\scriptsize{\cref{def-d2}}} \\
&=\underset{1 \le i \le n-1}{\sum}\Delta^n \otimes 1_A;1_{(\oc A)^{\otimes i}} \otimes \gamma_{(\oc A)^{\otimes(n-i)},A};
1_{(\oc A)^{\otimes i}} \otimes (1_A \otimes (1_{(\oc A)^{\otimes ((n-1)-i)}} \otimes \Delta)); &\text{\scriptsize{func.~of $-\otimes-$}} \\
&~1_{(\oc A)^{\otimes (i-1)}} \otimes \partial \otimes 1_{(\oc A)^{\otimes ((n-1)-i)}} \otimes 1_{\oc A \otimes \oc A} \\
&~+\Delta^n \otimes 1_A;1_{(\oc A)^{\otimes (n-1)}} \otimes (\Delta \otimes 1_A;1_{\oc A} \otimes \partial) &\text{\scriptsize{\cref{add-8}}} \\
&~+\Delta^n \otimes 1_A;1_{(\oc A)^{\otimes (n-1)}} \otimes (\Delta \otimes 1_A;1_{\oc A} \otimes \gamma_{\oc A,A};\partial \otimes 1_{\oc A}) \\
&=\underset{1 \le i \le n-1}{\sum}\Delta^n \otimes 1_A;
1_{(\oc A)^{\otimes i}} \otimes (\gamma_{(\oc A)^{\otimes(n-i)},A};1_A \otimes (1_{(\oc A)^{\otimes ((n-1)-i)}} \otimes \Delta)); &\text{\scriptsize{func.~of $-\otimes-$}} \\
&~1_{(\oc A)^{\otimes (i-1)}} \otimes \partial \otimes 1_{(\oc A)^{\otimes ((n-1)-i)}} \otimes 1_{\oc A \otimes \oc A} \\
&~+\Delta^n \otimes 1_A;1_{(\oc A)^{\otimes (n-1)}} \otimes \Delta \otimes 1_A;1_{(\oc A)^{\otimes (n-1)}} \otimes 1_{\oc A} \otimes \partial 
&\text{\scriptsize{func.~of $-\otimes-$}} \\
&~+\Delta^n \otimes 1_A;1_{(\oc A)^{\otimes (n-1)}} \otimes \Delta \otimes 1_A;1_{(\oc A)^{\otimes (n-1)}} \otimes 1_{\oc A} \otimes \gamma_{\oc A,A};
1_{(\oc A)^{\otimes (n-1)}} \otimes \partial \otimes 1_{\oc A} &\text{\scriptsize{func of $-\otimes-$}} \\
&=\underset{1 \le i \le n-1}{\sum}\Delta^n \otimes 1_A;
1_{(\oc A)^{\otimes i}} \otimes ((1_{(\oc A)^{\otimes (n-1-i)}} \otimes \Delta) \otimes 1_A;\gamma_{(\oc A)^{\otimes (n+1-i)},A}); &\text{\scriptsize{nat.~of $\gamma$}} \\
&~1_{(\oc A)^{\otimes (i-1)}} \otimes \partial \otimes 1_{(\oc A)^{\otimes ((n-1)-i)}} \otimes 1_{\oc A \otimes \oc A} \\
&~+(\Delta^n;1_{(\oc A)^{\otimes (n-1)}} \otimes \Delta) \otimes 1_A;1_{(\oc A)^{\otimes (n-1)}} \otimes 1_{\oc A} \otimes \partial &\text{\scriptsize{func.~of $-\otimes-$}} \\
&~+(\Delta^n;1_{(\oc A)^{\otimes (n-1)}} \otimes \Delta) \otimes 1_A;1_{(\oc A)^{\otimes (n-1)}} \otimes 1_{\oc A} \otimes \gamma_{\oc A,A};1_{(\oc A)^{\otimes (n-1)}} 
\otimes \partial \otimes 1_{\oc A} &\text{\scriptsize{func.~of $-\otimes-$}} \\
&=\underset{1 \le i \le n-1}{\sum}\Delta^n \otimes 1_A;
1_{(\oc A)^{\otimes i}} \otimes (1_{(\oc A)^{\otimes(n-1-i)}} \otimes \Delta \otimes 1_A);1_{(\oc A)^{\otimes i}} \otimes \gamma_{(\oc A)^{\otimes (n+1-i)},A}; 
&\text{\scriptsize{func.~of $-\otimes-$}} \\
&~1_{(\oc A)^{\otimes (i-1)}} \otimes \partial \otimes 1_{(\oc A)^{\otimes ((n-1)-i)}} \otimes 1_{\oc A \otimes \oc A} \\
&~+\Delta^{n+1} \otimes 1_A;1_{(\oc A)^{\otimes (n-1)}} \otimes 1_{\oc A} \otimes \partial &\text{\scriptsize{\cref{Delta-rec}}} \\
&~+\Delta^{n+1} \otimes 1_A;1_{(\oc A)^{\otimes (n-1)}} \otimes 1_{\oc A} \otimes \gamma_{\oc A,A};1_{(\oc A)^{\otimes (n-1)}} \otimes \partial \otimes 1_{\oc A} 
&\text{\scriptsize{\cref{Delta-rec}}} \\
&=\underset{1 \le i \le n-1}{\sum}\Delta^n \otimes 1_A;
1_{(\oc A)^{\otimes (n-1)}} \otimes \Delta \otimes 1_A;1_{(\oc A)^{\otimes i}} \otimes \gamma_{(\oc A)^{\otimes (n+1-i)},A}; &\text{\scriptsize{func. of $-\otimes-$}} \\
&~1_{(\oc A)^{\otimes (i-1)}} \otimes \partial \otimes 1_{(\oc A)^{\otimes ((n-1)-i)}} \otimes 1_{\oc A \otimes \oc A} \\
&~+\Delta^{n+1} \otimes 1_A;1_{(\oc A)^{\otimes (n-1)}} \otimes 1_{\oc A} \otimes \gamma_{\oc A,A};1_{(\oc A)^{\otimes (n-1)}} \otimes \partial \otimes 1_{\oc A} \\
&~+\Delta^{n+1} \otimes 1_A;1_{(\oc A)^{\otimes (n-1)}} \otimes 1_{\oc A} \otimes \partial \\
&=\underset{1 \le i \le n-1}{\sum}(\Delta^n;
1_{(\oc A)^{\otimes (n-1)}} \otimes \Delta) \otimes 1_A;1_{(\oc A)^{\otimes i}} \otimes \gamma_{(\oc A)^{\otimes (n+1-i)},A}; &\text{\scriptsize{func.~of $-\otimes-$}} \\
&~1_{(\oc A)^{\otimes (i-1)}} \otimes \partial \otimes 1_{(\oc A)^{\otimes ((n-1)-i)}} \otimes 1_{\oc A \otimes \oc A} \\
&~+\Delta^{n+1} \otimes 1_A;1_{(\oc A)^{\otimes n}} \otimes \gamma_{(\oc A)^{\otimes ((n+1)-n)},A};1_{(\oc A)^{\otimes (n-1)}} \otimes \partial \otimes 
1_{(\oc A)^{\otimes ((n+1)-n)}} &\text{\scriptsize{func.~of $-\otimes -$}} \\
&~+\Delta^{n+1} \otimes 1_A;1_{(\oc A)^{\otimes (n+1)}} \otimes \gamma_{(\oc A)^{\otimes ((n+1)-(n+1))},A};1_{(\oc A)^{\otimes ((n+1)-1)}} \otimes 
\partial \otimes 1_{(\oc A)^{\otimes ((n+1)-(n+1))}} &\text{\scriptsize{$\gamma_{I,A}=1_A$ and $f \otimes 1_I=f$}} \\
&=\underset{1 \le i \le n-1}{\sum}\Delta^{n+1} \otimes 1_A;1_{(\oc A)^{\otimes i}} \otimes \gamma_{(\oc A)^{\otimes (n+1-i)},A};1_{(\oc A)^{\otimes (i-1)}} \otimes 
\partial \otimes 1_{(\oc A)^{\otimes ((n+1)-i)}} &\text{\scriptsize{\cref{Delta-rec} and func.~of $\otimes$}} \\
&~+\Delta^{n+1} \otimes 1_A;1_{(\oc A)^{\otimes n}} \otimes \gamma_{(\oc A)^{\otimes ((n+1)-n)},A};1_{(\oc A)^{\otimes (n-1)}} \otimes \partial \otimes 
1_{(\oc A)^{\otimes ((n+1)-n)}} \\
&~+\Delta^{n+1} \otimes 1_A;1_{(\oc A)^{\otimes (n+1)}} \otimes \gamma_{(\oc A)^{\otimes ((n+1)-(n+1))},A};1_{(\oc A)^{\otimes ((n+1)-1)}} \otimes \partial \otimes 
1_{(\oc A)^{\otimes ((n+1)-(n+1))}} \\
&=\underset{1 \le i \le n+1}{\sum}\Delta^{n+1} \otimes 1_A;1_{(\oc A)^{\otimes i}} \otimes \gamma_{(\oc A)^{\otimes (n+1-i)},A};1_{(\oc A)^{\otimes (i-1)}} \otimes 
\partial \otimes 1_{(\oc A)^{\otimes ((n+1)-i)}}.
\end{align*}
\end{itemize}

\paragraph{Proof of \cref{id5}:} \label{proof:faà-di-bruno} The proof is by induction on $n \ge 0$.
\begin{itemize}
\item Base case ($n=0$): The LHS is equal to $m$ and the RHS is equal to $\Delta^1 \otimes 1_I;1_{\oc A} \otimes \overline{\tau_{\emptyset}};m=1_{\oc A} \otimes 1_I;1_{\oc A} \otimes 1_I;m=m$ (using \cref{def-tau-pi}).
\item Induction step: Suppose that the identity holds for some $n \ge 0$. We show that the identity holds for $n+1$. We have:
\begin{align*}
&\partial^{n+1};m \\
&=\partial \otimes 1_{A^{\otimes n}};\partial^n;m &\text{\scriptsize{\cref{reverse-higher}}} \\
&=\partial \otimes 1_{A^{\otimes n}};\underset{\pi \in \mathcal{H}(n)}{\sum}\Delta^{1+|\pi|} \otimes 1_{A^{\otimes n}};1_{\oc A} \otimes \overline{\tau_\pi}_{|\pi| \cdot \oc A,n \cdot A};m \otimes 
\partial^{|s_1|} \otimes \dots \otimes \partial^{|s_{|\pi|}|};\partial^{|\pi|} &\text{\scriptsize{induction hyp.}} \\
&=\underset{\pi \in \mathcal{H}(n)}{\sum}\partial \otimes 1_{A^{\otimes n}};\Delta^{1+|\pi|} \otimes 1_{A^{\otimes n}};1_{\oc A} \otimes \overline{\tau_\pi}_{|\pi| \cdot \oc A,n \cdot A};m \otimes 
\partial^{|s_1|} \otimes \dots \otimes \partial^{|s_{|\pi|}|};\partial^{|\pi|} &\text{\scriptsize{\cref{add-4}}} \\
&=\underset{\pi \in \mathcal{H}(n)}{\sum}(\partial;\Delta^{1+|\pi|}) \otimes 1_{A^{\otimes n}};
1_{\oc A} \otimes \overline{\tau_\pi}_{|\pi| \cdot \oc A,n \cdot A};m \otimes \partial^{|s_1|} \otimes \dots 
\otimes \partial^{|s_{|\pi|}|};\partial^{|\pi|}&\text{\scriptsize{func. of $-\otimes-$}} \\
&=\underset{\pi \in \mathcal{H}(n)}{\sum}\Big(\underset{1 \le i \le 1+|\pi|}{\sum}\Delta^{1+|\pi|} \otimes 1_A;1_{(\oc A)^{\otimes i}} \otimes 
\gamma_{(\oc A)^{\otimes(1+|\pi|-i)},A};1_{(\oc A)^{\otimes(i-1)}} \otimes \partial \otimes 1_{(\oc A)^{\otimes(1+|\pi|-i)}}\Big) \otimes 1_{A^{\otimes n}}; \\
&~1_{\oc A} \otimes \overline{\tau_\pi}_{|\pi| \cdot \oc A,n \cdot A};m \otimes \partial^{|s_1|} \otimes \dots 
\otimes \partial^{|s_{|\pi|}|};\partial^{|\pi|} &\text{\scriptsize{\cref{id4}}} \\
&=\underset{\pi \in \mathcal{H}(n)}{\sum}\underset{1 \le i \le 1+|\pi|}{\sum}\Big(\Delta^{1+|\pi|} \otimes 1_A;1_{(\oc A)^{\otimes i}} \otimes 
\gamma_{(\oc A)^{\otimes(1+|\pi|-i)},A};1_{(\oc A)^{\otimes(i-1)}} \otimes \partial \otimes 1_{(\oc A)^{\otimes(1+|\pi|-i)}}\Big) \otimes 1_{A^{\otimes n}}; \\
&~1_{\oc A} \otimes \overline{\tau_\pi}_{|\pi| \cdot \oc A,n \cdot A};m \otimes \partial^{|s_1|} \otimes \dots 
\otimes \partial^{|s_{|\pi|}|};\partial^{|\pi|} &\text{\scriptsize{\cref{add-7}}} \\
&=\underset{\pi \in \mathcal{H}(n)}{\sum}\underset{1 \le i \le 1+|\pi|}{\sum}\Delta^{1+|\pi|} \otimes 1_{A^{\otimes (n+1)}};1_{(\oc A)^{\otimes i}} 
\otimes \gamma_{(\oc A)^{\otimes(1+|\pi|-i)},A} \otimes 1_{A^{\otimes n}};&\text{\scriptsize{func.~of $-\otimes -$}} \\
&~1_{(\oc A)^{\otimes(i-1)}} \otimes \partial \otimes 1_{(\oc A)^{\otimes(1+|\pi|-i)}} \otimes 1_{A^{\otimes n}};1_{\oc A} \otimes \overline{\tau_\pi}_{|\pi| \cdot \oc A,n \cdot A};m \otimes 
\partial^{|s_1|} \otimes \dots 
\otimes \partial^{|s_{|\pi|}|};\partial^{|\pi|} \\
&=\underset{\pi \in \mathcal{H}(n)}{\sum}\Delta^{1+|\pi|} \otimes 1_{A^{\otimes (n+1)}};1_{\oc A} \otimes \gamma_{(\oc A)^{\otimes |\pi|},A} \otimes 1_{A^{\otimes n}}; &\text{\scriptsize{sep.~the sum}} \\
&~\partial \otimes 1_{(\oc A)^{\otimes |\pi|}} \otimes 1_{A^{\otimes n}};1_{\oc A} \otimes \overline{\tau_\pi}_{|\pi| \cdot \oc A,n \cdot A};m \otimes \partial^{|s_1|} \otimes \dots \otimes \partial^{|s_{|\pi|}|};\partial^{|\pi|} \\
&~+\underset{\pi \in \mathcal{H}(n)}{\sum}\underset{2 \le i \le 1+|\pi|}{\sum}\Delta^{1+|\pi|} \otimes 1_{A^{\otimes (n+1)}};1_{(\oc A)^{\otimes i}} 
\otimes \gamma_{(\oc A)^{\otimes(1+|\pi|-i)},A} \otimes 1_{A^{\otimes n}}; \\
&~1_{(\oc A)^{\otimes(i-1)}} \otimes \partial \otimes 1_{(\oc A)^{\otimes(1+|\pi|-i)}} \otimes 1_{A^{\otimes n}};1_{\oc A} \otimes \overline{\tau_\pi}_{|\pi| \cdot \oc A,n \cdot A};m \otimes  
\partial^{|s_1|} \otimes \dots 
\otimes \partial^{|s_{|\pi|}|};\partial^{|\pi|} \\
&=\underset{\pi \in \mathcal{H}(n)}{\sum}\Delta^{1+|\pi|} \otimes 1_{A^{\otimes (n+1)}};1_{\oc A} \otimes \gamma_{(\oc A)^{\otimes |\pi|},A} \otimes 
1_{A^{\otimes n}};\partial \otimes \overline{\tau_\pi}_{|\pi| \cdot \oc A,n \cdot A};m \otimes \partial^{|s_1|} \otimes \dots \otimes \partial^{|s_{|\pi|}|};\partial^{|\pi|} &\text{\scriptsize{func.~of $-\otimes -$}} \\
&~+\underset{\pi \in \mathcal{H}(n)}{\sum}\underset{2 \le i \le 1+|\pi|}{\sum}\Delta^{1+|\pi|} \otimes 1_{A^{\otimes (n+1)}};1_{(\oc A)^{\otimes i}} 
\otimes \gamma_{(\oc A)^{\otimes(1+|\pi|-i)},A} \otimes 1_{A^{\otimes n}}; \\
&~1_{\oc A} \otimes 1_{(\oc A)^{\otimes(i-2)}} \otimes \partial \otimes 1_{(\oc A)^{\otimes(1+|\pi|-i)}} \otimes 1_{A^{\otimes n}};1_{\oc A} \otimes 
\overline{\tau_\pi}_{|\pi| \cdot \oc A,n \cdot A};m \otimes \partial^{|s_1|} \otimes \dots 
\otimes \partial^{|s_{|\pi|}|};\partial^{|\pi|} &\text{\scriptsize{func.~of $-\otimes -$}} \\
&=\underset{\pi \in \mathcal{H}(n)}{\sum}\Delta^{1+|\pi|} \otimes 1_{A^{\otimes (n+1)}};1_{\oc A} \otimes \gamma_{(\oc A)^{\otimes |\pi|},A} \otimes 
1_{A^{\otimes n}};(\partial;m) \otimes \overline{\tau_\pi}_{|\pi| \cdot \oc A,n \cdot A}; \\
&1_{\oc\oc A} \otimes \partial^{|s_1|} \otimes \dots \otimes \partial^{|s_{|\pi|}|};\partial^{|\pi|} &\text{\scriptsize{func.~of $-\otimes -$}} \\
&~+\underset{\pi \in \mathcal{H}(n)}{\sum}\underset{2 \le i \le 1+|\pi|}{\sum}\Delta^{1+|\pi|} \otimes 1_{A^{\otimes (n+1)}};1_{(\oc A)^{\otimes i}} 
\otimes \gamma_{(\oc A)^{\otimes(1+|\pi|-i)},A} \otimes 1_{A^{\otimes n}}; \\
&~1_{\oc A} \otimes (1_{(\oc A)^{\otimes(i-2)}} \otimes \partial \otimes 1_{(\oc A)^{\otimes(1+|\pi|-i)}} \otimes 1_{A^{\otimes n}};\overline{\tau_\pi}_{|\pi| \cdot \oc A,n \cdot A});m 
\otimes \partial^{|s_1|} \otimes \dots 
\otimes \partial^{|s_{|\pi|}|};\partial^{|\pi|} &\text{\scriptsize{func.~of $-\otimes -$}} \\
&=\underset{\pi \in \mathcal{H}(n)}{\sum}\Delta^{1+|\pi|} \otimes 1_{A^{\otimes (n+1)}};1_{\oc A} \otimes \gamma_{(\oc A)^{\otimes |\pi|},A} 
\otimes 1_{A^{\otimes n}};(\Delta \otimes 1_A;m \otimes \partial_A;\partial_{\oc A}) \otimes \overline{\tau_\pi}_{|\pi| \cdot \oc A,n \cdot A}; &\text{\scriptsize{\cref{def-d3}}} \\
&~1_{\oc\oc A} \otimes \partial^{|s_1|} \otimes \dots \otimes \partial^{|s_{|\pi|}|};\partial^{|\pi|} \\
&~+\underset{\pi \in \mathcal{H}(n)}{\sum}\underset{2 \le i \le 1+|\pi|}{\sum}\Delta^{1+|\pi|} \otimes 1_{A^{\otimes (n+1)}};1_{(\oc A)^{\otimes i}} 
\otimes \gamma_{(\oc A)^{\otimes(1+|\pi|-i)},A} \otimes 1_{A^{\otimes n}}; \\ 
&~1_{\oc A} \otimes (\overline{\tau_\pi}_{(i-2) \cdot \oc A, \oc A \otimes A, (1+|\pi|-i) \cdot \oc A, n \cdot A};1_{(\oc A \otimes A^{\otimes |s_1|}) \otimes \dots \otimes (\oc A \otimes A^{\otimes |s_{i-2}|})} \otimes \partial 
\otimes \\
&~1_{A^{\otimes |s_{i-1}|} \otimes (\oc A \otimes A^{\otimes |s_i|}) \otimes \dots \otimes (\oc A \otimes A^{\otimes |s_\pi|})});m \otimes \partial^{|s_1|} \otimes \dots 
\otimes \partial^{|s_{|\pi|}|};\partial^{|\pi|}
&\text{\scriptsize{nat.~of $\overline{\tau_\pi}$}} \\
&=\underset{\pi \in \mathcal{H}(n)}{\sum}\Delta^{1+|\pi|} \otimes 1_{A^{\otimes (n+1)}};1_{\oc A} \otimes \gamma_{(\oc A)^{\otimes |\pi|},A} 
\otimes 1_{A^{\otimes n}};\Delta \otimes 1_{A \otimes (\oc A)^{\otimes |\pi|} \otimes A^{\otimes n}};1_{\oc A \otimes \oc A \otimes A} \\
&~\otimes \overline{\tau_\pi}_{|\pi| \cdot \oc A,n \cdot A}; \\
&~m \otimes \partial_A \otimes 1_{(\oc A \otimes A^{\otimes |s_1|}) \otimes \dots \otimes (\oc A \otimes A^{\otimes |s_{|\pi|}|})};\partial_{\oc A} \otimes \partial^{|s_1|} \otimes \dots \otimes \partial^{|s_{|\pi|}|};\partial^{|\pi|} &\text{\scriptsize{func.~of $-\otimes -$}} \\
&~+\underset{\pi \in \mathcal{H}(n)}{\sum}\underset{2 \le i \le 1+|\pi|}{\sum}\Delta^{1+|\pi|} \otimes 1_{A^{\otimes (n+1)}};1_{(\oc A)^{\otimes i}} 
\otimes \gamma_{(\oc A)^{\otimes(1+|\pi|-i)},A} \otimes 1_{A^{\otimes n}}; \\ 
&~m \otimes \Big(\overline{\tau_\pi}_{(i-2) \cdot \oc A, \oc A \otimes A, (1+|\pi|-i) \cdot \oc A, n \cdot A};\partial^{|s_1|} \otimes \dots \otimes \partial^{|s_{i-2}|} \otimes \big[(\partial \otimes 1_{A^{\otimes(|s_{i-1}|)}});
\partial^{|s_{i-1}|})\big] \otimes \\
&~\partial^{|s_i|} \otimes \dots \otimes \partial^{|s_{|\pi|}|}\Big);\partial^{|\pi|} &\text{\scriptsize{func.~of $-\otimes -$}} \\
&=\underset{\pi \in \mathcal{H}(n)}{\sum}\Delta^{1+|\pi|} \otimes 1_{A^{\otimes (n+1)}};\Delta \otimes \gamma_{(\oc A)^{\otimes |\pi|},A}
\otimes 1_{A^{\otimes n}};1_{\oc A \otimes \oc A \otimes A} \otimes \overline{\tau_\pi}_{|\pi| \cdot \oc A,n \cdot A}; &\text{\scriptsize{func.~of $-\otimes -$}} \\
&~m \otimes \partial_A \otimes 1_{(\oc A \otimes A^{\otimes |s_1|}) \otimes \dots \otimes (\oc A \otimes A^{\otimes |s_{|\pi|}|})};\partial_{\oc A} 
\otimes 1_{(\oc A \otimes A^{\otimes |s_1|}) \otimes \dots \otimes (\oc A \otimes A^{\otimes |s_{|\pi|}|})}; \\
&~1_{\oc\oc A} \otimes \partial^{|s_1|} \otimes \dots \otimes \partial^{|s_{|\pi|}|};\partial^{|\pi|} \\
&~+\underset{\pi \in \mathcal{H}(n)}{\sum}\underset{2 \le i \le 1+|\pi|}{\sum}\Delta^{1+|\pi|} \otimes 1_{A^{\otimes (n+1)}};1_{(\oc A)^{\otimes i}} 
\otimes \gamma_{(\oc A)^{\otimes(1+|\pi|-i)},A} \otimes 1_{A^{\otimes n}}; \\ 
&~m \otimes \Big(\overline{\tau_\pi}_{(i-2) \cdot \oc A, \oc A \otimes A, (1+|\pi|-i) \cdot \oc A, n \cdot A};\partial^{|s_1|} \otimes \dots \otimes \partial^{|s_{i-2}|} \otimes \partial^{|s_{i-1}|+1} \otimes \partial^{|s_i|} 
\otimes \dots \otimes \partial^{|s_{|\pi|}|}\Big);\partial^{|\pi|} &\text{\scriptsize{\cref{reverse-higher}}} \\
&=\underset{\pi \in \mathcal{H}(n)}{\sum}\Delta^{1+|\pi|} \otimes 1_{A^{\otimes (n+1)}};\Delta \otimes 1_{(\oc A)^{\otimes |\pi|} 
\otimes A^{\otimes(n+1)}};1_{\oc A \otimes \oc A} \otimes \gamma_{(\oc A)^{\otimes |\pi|},A} \otimes 1_{A^{\otimes n}}; &\text{\scriptsize{func.~of $-\otimes -$}} \\
&~1_{\oc A \otimes \oc A \otimes A} \otimes \overline{\tau_\pi}_{|\pi| \cdot \oc A,n \cdot A};m \otimes \partial_A \otimes 1_{(\oc A \otimes A^{\otimes |s_1|}) \otimes \dots \otimes (\oc A \otimes A^{\otimes |s_{|\pi|}|})};\\ 
&~\partial_{\oc A} 
\otimes 1_{(\oc A \otimes A^{\otimes |s_1|}) \otimes \dots \otimes (\oc A \otimes A^{\otimes |s_{|\pi|}|})};1_{\oc\oc A} \otimes \partial^{|s_1|} \otimes \dots \otimes \partial^{|s_{|\pi|}|};\partial^{|\pi|} \\
&~+\underset{\pi \in \mathcal{H}(n)}{\sum}\underset{2 \le i \le 1+|\pi|}{\sum}\Delta^{1+|\pi|} \otimes 1_{A^{\otimes (n+1)}};1_{(\oc A)^{\otimes i}} 
\otimes \gamma_{(\oc A)^{\otimes(1+|\pi|-i)},A} \otimes 1_{A^{\otimes n}}; \\ 
&~1_{\oc A} \otimes \overline{\tau_\pi}_{(i-2) \cdot \oc A, \oc A \otimes A, (1+|\pi|-i) \cdot \oc A, n \cdot A};m \otimes \partial^{|s_1|} \otimes \dots \otimes \partial^{|s_{i-2}|} \otimes \partial^{|s_{i-1}|+1} \otimes \partial^{|s_i|} 
\otimes \dots \otimes \partial^{|s_{|\pi|}|};\partial^{|\pi|} &\text{\scriptsize{func.~of $-\otimes -$}} \\
&=\underset{\pi \in \mathcal{H}(n)}{\sum}\Big(\Delta^{1+|\pi|};\Delta \otimes 1_{(\oc A)^{\otimes |\pi|}}\Big) \otimes 1_{A^{\otimes(n+1)}};1_{\oc A \otimes \oc A} 
\otimes \gamma_{(\oc A)^{\otimes |\pi|},A} \otimes 1_{A^{\otimes n}};1_{\oc A \otimes \oc A \otimes A} \otimes \overline{\tau_\pi}_{|\pi| \cdot \oc A,n \cdot A}; \\
&~m \otimes \partial_A \otimes 1_{(\oc A \otimes A^{\otimes |s_1|}) \otimes \dots \otimes (\oc A \otimes A^{\otimes |s_{|\pi|}|})}; &\text{\scriptsize{func.~of $-\otimes -$}} \\
&~\partial_{\oc A} \otimes 1_{(\oc A \otimes A^{\otimes |s_1|}) \otimes \dots \otimes (\oc A \otimes A^{\otimes |s_{|\pi|}|})};1_{\oc\oc A} \otimes \partial^{|s_1|} 
\otimes \dots \otimes \partial^{|s_{|\pi|}|};\partial^{|\pi|}  \\
&~+\underset{\pi \in \mathcal{H}(n)}{\sum}\underset{2 \le i \le 1+|\pi|}{\sum}\Delta^{1+|\pi|} \otimes 1_{A^{\otimes(n+1)}};1_{\oc A} \otimes 
\Big(1_{(\oc A)^{\otimes (i-1)}} \otimes \gamma_{(\oc A)^{\otimes (1+|\pi|-i)},A} \otimes 1_{A^{\otimes n}}; \\
&~\overline{\tau_\pi}_{(i-2) \cdot \oc A, \oc A \otimes A, (1+|\pi|-i) \cdot \oc A, n \cdot A}\Big);m \otimes \partial^{|s_1|} \otimes \dots \otimes \partial^{|s_{i-2}|} \otimes \partial^{|s_{i-1}|+1} \otimes \partial^{|s_i|} \otimes \dots 
\otimes \partial^{|s_{|\pi|}|};\partial^{|\pi|}  &\text{\scriptsize{func.~of $-\otimes -$}} \\
&=\underset{\pi \in \mathcal{H}(n)}{\sum}\Delta^{1+(|\pi|+1)} \otimes 1_{A^{\otimes(n+1)}};1_{\oc A \otimes \oc A} \otimes \gamma_{(\oc A)^{\otimes |\pi|},A} 
\otimes 1_{A^{\otimes n}};1_{\oc A \otimes \oc A \otimes A} \otimes \overline{\tau_\pi}_{|\pi| \cdot \oc A,n \cdot A}; &\text{\scriptsize{coassoc.~\& \cref{Delta-rec}}} \\
&~m \otimes \partial_A \otimes 1_{(\oc A \otimes A^{\otimes |s_1|}) \otimes \dots \otimes (\oc A \otimes A^{\otimes |s_{|\pi|}|})};\partial_{\oc A} 
\otimes 1_{(\oc A \otimes A^{\otimes |s_1|}) \otimes \dots \otimes (\oc A \otimes A^{\otimes |s_{|\pi|}|})}; \\
&~1_{\oc\oc A} \otimes \partial^{|s_1|} \otimes \dots \otimes \partial^{|s_{|\pi|}|};\partial^{|\pi|} \\ \\
&~+\underset{\pi \in \mathcal{H}(n)}{\sum}\underset{2 \le i \le 1+|\pi|}{\sum}\Delta^{1+|\pi|} \otimes 1_{A^{\otimes(n+1)}};1_{\oc A} \otimes 
\overline{\tau_{\rho(\pi,i)}}_{|\pi| \cdot \oc A,(n+1) \cdot A}; &\text{\scriptsize{\cref{Mac-1}}} \\
&~m \otimes \partial^{|s_1|} \otimes \dots \otimes \partial^{|s_{i-2}|} \otimes \partial^{|s_{i-1}|+1} \otimes \partial^{|s_i|} \otimes \dots 
\otimes \partial^{|s_{|\pi|}|};\partial^{|\pi|} \\
&=\underset{\pi \in \mathcal{H}(n)}{\sum}\Delta^{1+(|\pi|+1)} \otimes 1_{A^{\otimes(n+1)}};1_{\oc A} \otimes \Big(1_{\oc A} \otimes 
\gamma_{(\oc A)^{\otimes |\pi|},A} \otimes 1_{A^{\otimes n}};1_{\oc A \otimes A} \otimes \overline{\tau_\pi}_{|\pi| \cdot \oc A,n \cdot A}\Big); &\text{\scriptsize{func.~of $-\otimes -$}} \\
&~m \otimes \partial_A \otimes 1_{(\oc A \otimes A^{\otimes |s_1|}) \otimes \dots \otimes (\oc A \otimes A^{\otimes |s_{|\pi|}|})};\partial_{\oc A} 
\otimes 1_{(\oc A \otimes A^{\otimes |s_1|}) \otimes \dots \otimes (\oc A \otimes A^{\otimes |s_{|\pi|}|})}; \\
&~1_{\oc\oc A} \otimes \partial^{|s_1|} \otimes \dots \otimes \partial^{|s_{|\pi|}|};\partial^{|\pi|} \\
&~+\underset{\pi \in \mathcal{H}(n)}{\sum}\underset{2 \le i \le 1+|\pi|}{\sum}\Delta^{1+|\rho(\pi,i)|} \otimes 1_{A^{\otimes(n+1)}};1_{\oc A} 
\otimes \overline{\tau_{\rho(\pi,i)}}_{|\pi| \cdot \oc A,(n+1) \cdot A}; &\text{\scriptsize{\cref{rho-2-first}}} \\
&~m \otimes \partial^{|s_1|} \otimes \dots \otimes \partial^{|s_{i-2}|} \otimes \partial^{|s_{i-1}|+1} \otimes \partial^{|s_i|} \otimes \dots 
\otimes \partial^{|s_{|\pi|}|};\partial^{|\pi|} \\
&=\underset{\pi \in \mathcal{H}(n)}{\sum}\Delta^{1+(|\pi|+1)} \otimes 1_{A^{\otimes(n+1)}};1_{\oc A} \otimes \overline{\tau_{\rho(\pi,1)}}_{(|\pi|+1) \cdot \oc A,(n+1) \cdot A}; &\text{\scriptsize{\cref{Mac-2}}} \\
&~m \otimes \partial_A \otimes 1_{(\oc A \otimes A^{\otimes |s_1|}) \otimes \dots \otimes (\oc A \otimes A^{\otimes |s_{|\pi|}|})};\partial_{\oc A} 
\otimes 1_{(\oc A \otimes A^{\otimes |s_1|}) \otimes \dots \otimes (\oc A \otimes A^{\otimes |s_{|\pi|}|})}; \\
&~1_{\oc\oc A} \otimes \partial^{|s_1|} \otimes \dots \otimes \partial^{|s_{|\pi|}|};\partial^{|\pi|} \\
&~+\underset{\pi \in \mathcal{H}(n)}{\sum}\underset{2 \le i \le 1+|\pi|}{\sum}\Delta^{1+|\rho(\pi,i)|} \otimes 1_{A^{\otimes(n+1)}};1_{\oc A} \otimes 
\overline{\tau_{\rho(\pi,i)}}_{|\pi| \cdot \oc A,(n+1) \cdot A}; \\
&~m \otimes \partial^{|s_1|} \otimes \dots \otimes \partial^{|s_{i-2}|} \otimes \partial^{|s_{i-1}|+1} \otimes \partial^{|s_i|} \otimes \dots \otimes 
\partial^{|s_{|\pi|}|};\partial^{|\pi|} \\
&=\underset{\pi \in \mathcal{H}(n)}{\sum}\Delta^{1+|\rho(\pi,1)|} \otimes 1_{A^{\otimes(n+1)}};1_{\oc A} \otimes \overline{\tau_{\rho(\pi,1)}}_{(|\pi|+1) \cdot \oc A,(n+1) \cdot A};&\text{\scriptsize{\cref{rho-1-first}}} \\
&~m \otimes 
\partial_A \otimes 1_{(\oc A \otimes A^{\otimes |s_1|}) \otimes \dots \otimes (\oc A \otimes A^{\otimes |s_{|\pi|}|})};\partial_{\oc A} \otimes 1_{(\oc A \otimes A^{\otimes |s_1|}) \otimes \dots \otimes (\oc A \otimes A^{\otimes |s_{|\pi|}|})}; \\
&~1_{\oc\oc A} \otimes 
\partial^{|s_1|} \otimes \dots \otimes \partial^{|s_{|\pi|}|};\partial^{|\pi|} \\
&~+\underset{\pi \in \mathcal{H}(n)}{\sum}\underset{2 \le i \le 1+|\pi|}{\sum}\Delta^{1+|\rho(\pi,i)|} \otimes 1_{A^{\otimes(n+1)}};1_{\oc A} \otimes 
\overline{\tau_{\rho(\pi,i)}}_{|\pi| \cdot \oc A,(n+1) \cdot A}; \\
&~m \otimes \partial^{|s_1|} \otimes \dots \otimes \partial^{|s_{i-2}|} \otimes \partial^{|s_{i-1}|+1} \otimes \partial^{|s_i|} \otimes \dots \otimes 
\partial^{|s_{|\pi|}|};\partial^{|\pi|} \\
&=\underset{\pi \in \mathcal{H}(n)}{\sum}\Delta^{1+|\rho(\pi,1)|} \otimes 1_{A^{\otimes(n+1)}};1_{\oc A} \otimes \overline{\tau_{\rho(\pi,1)}}_{(|\pi|+1) \cdot \oc A,(n+1) \cdot A}; \\ 
& m \otimes  \partial_A \otimes 1_{(\oc A \otimes A^{\otimes |s_1|}) \otimes \dots \otimes (\oc A \otimes A^{\otimes |s_{|\pi|}|})};\partial_{\oc A} \otimes \partial^{|s_1|} \otimes \dots \otimes \partial^{|s_{|\pi|}|};\partial^{|\pi|}  &\text{\scriptsize{func.~of~$-\otimes -$}} \\
&~+\underset{\pi \in \mathcal{H}(n)}{\sum}\underset{2 \le i \le 1+|\pi|}{\sum}\Delta^{1+|\rho(\pi,i)|} \otimes 1_{A^{\otimes(n+1)}};1_{\oc A} 
\otimes \overline{\tau_{\rho(\pi,i)}}_{|\pi| \cdot \oc A,(n+1) \cdot A}; \\
&~m \otimes \partial^{|s_1|} \otimes \dots \otimes \partial^{|s_{i-2}|} \otimes \partial^{|s_{i-1}|+1} \otimes \partial^{|s_i|} \otimes \dots 
\otimes \partial^{|s_{|\pi|}|};\partial^{|\pi|} \\
&=\underset{\pi \in \mathcal{H}(n)}{\sum}\Delta^{1+|\rho(\pi,1)|} \otimes 1_{A^{\otimes(n+1)}};1_{\oc A} \otimes \overline{\tau_{\rho(\pi,1)}}_{(|\pi|+1) \cdot \oc A,(n+1) \cdot A};m \otimes 
\partial_A \otimes \partial^{|s_1|} \otimes \dots \otimes \partial^{|s_{|\pi|}|}; \\
&\partial_{\oc A} \otimes 1_{(\oc A)^{\otimes |\pi|}};\partial^{|\pi|} &\text{\scriptsize{func.~of~$-\otimes -$}} \\
&~+\underset{\pi \in \mathcal{H}(n)}{\sum}\underset{2 \le i \le 1+|\pi|}{\sum}\Delta^{1+|\rho(\pi,i)|} \otimes 1_{A^{\otimes(n+1)}};1_{\oc A} 
\otimes \overline{\tau_{\rho(\pi,i)}}_{|\pi| \cdot \oc A,(n+1) \cdot A}; \\
&~m \otimes \partial^{|s_1|} \otimes \dots \otimes \partial^{|s_{i-2}|} \otimes \partial^{|s_{i-1}|+1} \otimes \partial^{|s_i|} \otimes \dots 
\otimes \partial^{|s_{|\pi|}|};\partial^{|\pi|} \\
&=\underset{\pi \in \mathcal{H}(n)}{\sum}\Delta^{1+|\rho(\pi,1)|} \otimes 1_{A^{\otimes(n+1)}};1_{\oc A} \otimes \overline{\tau_{\rho(\pi,1)}}_{(|\pi|+1) \cdot \oc A,(n+1) \cdot A};m \otimes 
\partial_A \otimes \partial^{|s_1|} \otimes \dots \otimes \partial^{|s_{|\pi|}|};\partial^{|\pi|+1} &\text{\scriptsize{\cref{reverse-higher}}} \\
&~+\underset{\pi \in \mathcal{H}(n)}{\sum}\underset{2 \le i \le 1+|\pi|}{\sum}\Delta^{1+|\rho(\pi,i)|} \otimes 1_{A^{\otimes(n+1)}};1_{\oc A} 
\otimes \overline{\tau_{\rho(\pi,i)}}_{|\pi| \cdot \oc A,(n+1) \cdot A}; \\
&~m \otimes \partial^{|s_1|} \otimes \dots \otimes \partial^{|s_{i-2}|} \otimes \partial^{|s_{i-1}|+1} \otimes \partial^{|s_i|} \otimes \dots 
\otimes \partial^{|s_{|\pi|}|};\partial^{|\pi|} \\
&=\underset{\pi \in \mathcal{H}(n)}{\sum}\Delta^{1+|\rho(\pi,1)|} \otimes 1_{A^{\otimes(n+1)}};1_{\oc A} \otimes \overline{\tau_{\rho(\pi,1)}}_{(|\rho(\pi,1)|) \cdot \oc A,(n+1) \cdot A};m \otimes 
\partial^{|t_1|} \otimes \partial^{|t_2|} \otimes \dots \otimes \partial^{|t_{|\rho(\pi,1)|}|}; \\
& \partial^{|\rho(\pi,1)|} 
&\text{\scriptsize{\cref{rho-1-first,rho-1-second,rho-1-third}}} \\
&~+\underset{\pi \in \mathcal{H}(n)}{\sum}\underset{2 \le i \le 1+|\pi|}{\sum}\Delta^{1+|\rho(\pi,i)|} \otimes 1_{A^{\otimes(n+1)}};1_{\oc A} 
\otimes \overline{\tau_{\rho(\pi,i)}}_{|\rho(\pi,i)| \cdot \oc A,(n+1) \cdot A}; \\
&~m \otimes \partial^{|t_1|} \otimes \dots \otimes \partial^{|t_{i-2}|} \otimes \partial^{|t_{i-1}|} \otimes \partial^{|t_i|} \otimes \dots \otimes 
\partial^{|t_{|\rho(\pi,i)|}|};\partial^{|\rho(\pi,i)|} &\text{\scriptsize{\cref{rho-2-first,rho-2-second,rho-2-third}}} \\
&=\underset{\pi \in \mathcal{H}(n)}{\sum}\underset{1 \le i \le 1+|\pi|}{\sum}\Delta^{1+|\rho(\pi,i)|} \otimes 1_{A^{\otimes(n+1)}};1_{\oc A} \otimes 
\overline{\tau_{\rho(\pi,i)}}_{|\rho(\pi,i)| \cdot \oc A,(n+1) \cdot A};m \otimes \partial^{|t_1|} \otimes \dots \otimes \partial^{|t_{|\rho(\pi,i)|}|}; \\
&~\partial^{|\rho(\pi,i)|} \\
&=\underset{(\pi,i) \in \underset{\pi \in \mathcal{H}(n)}{\bigsqcup}\{1,\dots,|\pi|+1\}}{\sum}\Delta^{1+|\rho(\pi,i)|} \otimes 1_{A^{\otimes(n+1)}};1_{\oc A} 
\otimes \overline{\tau_{\rho(\pi,i)}}_{|\rho(\pi,i)| \cdot \oc A,(n+1) \cdot A}; &\text{\scriptsize{\cref{def:disjoint-union}}} \\
&~m \otimes \partial^{|t_1|} \otimes \dots \otimes \partial^{|t_{|\rho(\pi,i)|}|};\partial^{|\rho(\pi,i)|} \\
&=\underset{\rho \in \mathcal{H}(n+1)}{\sum}\Delta^{1+|\rho|} \otimes 1_{A^{\otimes(n+1)}};1_{\oc A} \otimes \overline{\tau_{\rho}}_{|\rho| \cdot \oc A,(n+1) \cdot A};m \otimes \partial^{|t_1|} \otimes 
\dots \otimes \partial^{|t_{|\rho|}|};\partial^{|\rho|}. &\text{\scriptsize{\cref{prop-bij}}} \\
\end{align*}
\end{itemize}
\begin{remark} \label{Mac-1}
Mac Lane's coherence theorem for symmetric monoidal categories implies that for all $n \ge 1$, $\pi \in \mathcal{H}(n)$ and $2 \le i \le |\pi|+1$, we have the following equality of natural transformations: 
\begin{equation*}
1_{(\oc A)^{\otimes (i-1)}} \otimes \gamma_{(\oc A)^{\otimes (1+|\pi|-i)},A} \otimes 1_{A^{\otimes n}};\overline{\tau_\pi}_{(i-2) \cdot \oc A, \oc A \otimes A, (1+|\pi|-i) \cdot \oc A, n \cdot A}=\overline{\tau_{\rho(\pi,i)}}_{|\pi| \cdot \oc A,(n+1) \cdot A}.
\end{equation*}
Write $\pi=\{s_1,\dots,s_{|\pi|}\}$ where $\mathrm{max}\,s_1<\dots<\mathrm{max}\,s_{|\pi|}$ and $s_j=\{k_j^1<\dots<k_{j}^{|s_j|}\}$ for every $1 \le j \le |\pi|$. 

Mac Lane's coherence theorem implies the following equality of natural transformations:
\begin{align*}
&1_{A_1 \otimes \dots \otimes A_{i-1}} \otimes \gamma_{A_i \otimes \dots \otimes A_{|\pi|},B_1} \otimes 1_{B_2 \otimes \dots \otimes B_{n+1}};\overline{\tau_\pi}_{A_1,\dots,A_{i-2},A_{i-1} \otimes B_1,A_i,\dots,A_{|\pi|},B_2,\dots,B_{n+1}} \\
&=\overline{\tau_{\rho(\pi,i)}}_{A_1,\dots,A_{|\pi|},B_1,\dots,B_{n+1}}
\end{align*}
for every $2 \le i \le |\pi|+1$ since these are two morphisms in $\mathsf{Pit}_\otimes(\mathsf{C})$ with domain
\begin{equation*}
A_1\otimes \dots \otimes A_{|\pi|} \otimes B_1 \otimes \dots \otimes B_{n+1}
\end{equation*}
and codomain
\begin{align*}
&A_1 \otimes B_{k_1^1+1} \otimes \dots \otimes B_{k_1^{|s_1|}+1} \otimes \dots \otimes A_{i-2} \otimes B_{k_{i-2}^1+1} \otimes \dots \otimes B_{k_{i-2}^{|s_{i-2}|}+1} \otimes A_{i-1} \\
&\otimes B_1 \otimes B_{k_{i-1}^1+1} \otimes \dots \otimes B_{k_{i-1}^{|s_{i-1}|}+1} \otimes A_i \otimes B_{k_i^1+1} \otimes \dots \otimes B_{k^{|s_i|}_i+1} \\
&\otimes \dots \otimes A_{|\pi|} \otimes B_{k_{|\pi|}^1+1} \otimes \dots \otimes B_{k_{|\pi|}^{|s_{\pi}|}+1}.
\end{align*}
Note that the codomain of the LHS is given by \cref{def-tau-pi} and the domain and codomain of the RHS are given by \cref{def-tau-pi,eq-rho-pi-i,recorded-facts}.

Choosing $A_1=\dots=A_{|\pi|}:=\oc A$ and $B_1=\dots=B_{n+1}:=A$, we obtain the desired identity.
\end{remark}
\begin{remark} \label{Mac-2}
Mac Lane's coherence theorem for symmetric monoidal categories implies that for all $n \ge 0$ and $\pi \in \mathcal{H}(n)$, we have the following equality of natural transformations: 
\begin{equation*}
1_{\oc A} \otimes 
\gamma_{(\oc A)^{\otimes(|\pi|)},A} \otimes 1_{A^{\otimes n}};1_{\oc A \otimes A} \otimes \overline{\tau_\pi}_{|\pi| \cdot \oc A,n \cdot A}=\overline{\tau_{\rho(\pi,1)}}_{(|\pi|+1) \cdot \oc A,(n+1) \cdot A}.
\end{equation*}
Write $\pi=\{s_1,\dots,s_{|\pi|}\}$ where $\mathrm{max}\,s_1<\dots<\mathrm{max}\,s_{|\pi|}$ and $s_j=\{k_j^1<\dots<k_{j}^{|s_j|}\}$ for every $1 \le j \le |\pi|$.

Mac Lane's coherence theorem implies the following equality of natural transformations:
\begin{align*}
&1_{A_1} \otimes \gamma_{A_2 \otimes \dots \otimes A_{|\pi|+1},B_1} \otimes 1_{B_2 \otimes \dots \otimes B_{n+1}};1_{A_1 \otimes B_1} \otimes \overline{\tau_\pi}_{A_2,\dots,A_{|\pi|+1},B_2,\dots,B_{n+1}} \\
&=\overline{\tau_{\rho(\pi,1)}}_{A_1,\dots,A_{|\pi|+1},B_1,\dots,B_{n+1}}
\end{align*}
since these are two morphisms in $\mathsf{Pit}_\otimes(\mathsf{C})$ with domain
\begin{equation*}
A_1 \otimes \dots \otimes A_{|\pi|+1} \otimes B_1 \otimes \dots \otimes B_{n+1}
\end{equation*}
and codomain
\begin{equation*}
A_1 \otimes B_1 \otimes A_2 \otimes B_{k_1^1+1} \otimes \dots \otimes B_{k_1^{|s_1|}+1} \otimes \dots \otimes A_{|\pi|+1} \otimes B_{k_{|\pi|}^1+1} \otimes \dots \otimes B_{k_{|\pi|}^{|s_{|\pi|}|}+1}.
\end{equation*}
Note that the codomain of the LHS is given by \cref{def-tau-pi} and the domain and codomain of the RHS are given by \cref{def-tau-pi,eq-rho-pi-1,recorded-facts}.

Choosing $A_1=\dots=A_{|\pi|+1}:=\oc A$ and $B_1=\dots=B_{n+1}:=A$, we obtain the desired identity.
\end{remark}
\end{document}